\documentclass[12pt]{article}
\usepackage{amsmath,amssymb,amsthm}
\usepackage{graphicx}
\usepackage{hyperref}
\newtheorem{theorem}{Theorem}[section]
\newtheorem{lemma}[theorem]{Lemma}
\newtheorem{proposition}[theorem]{Proposition}

\newtheorem{problem}[theorem]{Problem}
\theoremstyle{definition}
\newtheorem{definition}[theorem]{Definition}

\def\bN{\mathbb{N}}

\def\bZ{\mathbb{Z}}
\def\cA{\mathcal{A}}

\def\cG{\mathcal{G}}

\def\cI{\mathcal{I}}
\def\cK{\mathcal{K}}

\def\al{\alpha}

\def\ga{\gamma}
\def\de{\delta}
\def\eps{\varepsilon}

\def\om{\omega}
\def\si{\sigma}

\def\Om{\Omega}

\def\subst{\mathsf{s}}
\def\TM{\subst_{\mathrm{TM}}}
\def\pd{\subst_{\mathrm{pd}}}
\def\sub{\subst_{\mathrm{Mu}}}
\def\xiTM{\xi_{\mathrm{TM}}}
\def\xipd{\xi_{\mathrm{pd}}}
\def\xiGE{\xi_{\mathrm{Mu}}}

\def\Homeo{\mathop{\mathrm{Homeo}}}
\def\Sch{\mathop{\mathrm{Sch}}}
\def\id{\mathrm{id}}
\def\0{\mathbf{0}}
\title{Thue-Morse sequence and groups of intermediate growth}
\author{Rostislav Grigorchuk \and Yaroslav Vorobets}
\date{}
\begin{document}

\maketitle

\begin{abstract}
We consider the substitution subshift generated by the Thue-Morse 
substitution $0\to01$, $1\to10$.  We prove that the topological full group 
of the subshift contains a subgroup of intermediate growth.  Namely, one 
group from the family known as the Grigorchuk groups embeds into this group.

To obtain our main result, we prove an embedding theorem for topological 
full groups, and also develop a technique to prove isomorphism of groups 
using the Schreier graphs.
\end{abstract}

\section{Introduction}\label{sect-intro}

There are many topics where the interests of theory of dynamical systems 
and algebra converge.  The interconnection between these two mathematical 
disciplines goes both ways.  Theory of dynamical systems is looking for 
algebraic invariants to use in various classification problems.  Besides, a 
lot of dynamical systems are of algebraic origin (see \cite{Schmidt, 
LindSchmidt}).  On the other hand, algebra (first of all, group theory) 
uses ideas and methods of theory of dynamical systems for new constructions 
and for solving old problems as demonstrated, for instance, in the recent 
book by V. Nekrashevych \cite{Nek2022}.

This paper is centered around one specific topic on the intersection of 
theory of dynamical systems with group theory, namely, the topological full 
groups associated with Cantor minimal systems.

The \emph{Cantor minimal systems} are classical objects of study in 
topological dynamics.  Such a system $(X,T)$ consists of a Cantor set $X$ 
and a homeomorphism $T:X\to X$, which defines a continuous action of the 
infinite cyclic group $\bZ$ on $X$.  The action is to be minimal in the 
sense that it allows no closed invariant set other than the empty set and 
$X$.  We are interested in systems coming from symbolic dynamics.  Recall 
that the two-sided \emph{shift} over a finite alphabet $\cA$ (of at least 
two symbols) is the transformation $\si$ of the space $\cA^{\bZ}$ of 
bi-infinite sequences of symbols from $\cA$ that acts by the rule 
$\si(\ldots x_{-2}x_{-1}x_0.x_1x_2\ldots)=\ldots 
y_{-2}y_{-1}y_0.y_1y_2\ldots$, where $y_n=x_{n+1}$ for all $n\in\bZ$.
A \emph{subshift} is the restriction of $\si$ to a closed shift-invariant 
set $\Om\subset\cA^{\bZ}$.  Any subshift is uniquely determined by the set 
of all forbidden words (i.e., finite strings of symbols that must not occur 
in elements of $\Om$) or, equivalently, by the set of all admissible 
words.  Two major classes of subshifts are subshifts of finite type and 
minimal subshifts.  They represent two different kinds of dynamic behavior 
(chaotic vs.\@ regular).  A subshift of finite type is determined by 
specifying a finite number of forbidden words.  The minimal subshifts are 
not so easy to construct.  The main source of examples are the substitution 
subshifts generated by primitive substitutions.  A \emph{substitution} over 
an alphabet $\cA$ is a procedure that replaces every symbol from the 
alphabet with a certain word over $\cA$.  Suppose that the substitution 
preserves an infinite sequence $\xi=x_1x_2x_3\ldots$ of symbols from $\cA$, 
and let $\bigl(\Om(\xi),\si\bigr)$ be the subshift for which the admissible 
words are words that occur infinitely often in $\xi$.  Then it is called a 
\emph{substitution subshift}.  The substitution is \emph{primitive} if, 
starting with any symbol and applying the substitution several times, we 
eventually end up with a word that includes all symbols in the alphabet.  
If this is the case then the subshift $\bigl(\Om(\xi),\si\bigr)$ is 
necessarily a Cantor minimal system (see, e.g., the books 
\cite{AlloucheShallit:automatic, Fogg:substitutions, LindMarcus}). 

The \emph{Thue-Morse} substitution $\TM$ over the binary alphabet 
$\{0,1\}$, given by $0\to01$, $1\to10$, is among the simplest and most 
famous substitutions that appear in many areas of mathematics.  Starting 
with $0$ and repeatedly applying the substitution $\TM$, we obtain a 
sequence of finite binary strings $0,01,0110,01101001,\dots$ where each 
string extends the previous one.  This results in an infinite binary string 
$\xiTM=01101001\dots$ invariant under the substitution rule.  The string 
$\xiTM$ is known as the \emph{Thue-Morse sequence}.  Since the substitution 
$\TM$ is primitive, the sequence $\xiTM$ generates a minimal subshift over 
the binary alphabet $\{0,1\}$, called the \emph{Morse subshift}.  One 
instance where the sequence $\xiTM$ appears in algebra is a criterion of
nilpotency of a group in terms of semigroup identities, which uses the idea 
of iteration in the spirit of the Thue-Morse substitution as shown 
implicitly by A. Maltsev \cite{Maltsev} and repeated in an explicit form in 
\cite{BoffaPoint}.  This is also related to the so-called Zimin words (see, 
e.g., the book \cite{Sapir}), which are used in theory of rings.  L. Rowen 
\cite{Rowen:Koethe} suggested that $\xiTM$ may be the key to solving one of 
the most famous problems in theory of rings, the Koethe conjecture.  For 
many more applications of the Thue-Morse sequence in many different 
contexts, see the survey \cite{AlloucheShallit:ubiquitous}.

\begin{figure}[t]
\centerline{\includegraphics[scale=0.8]{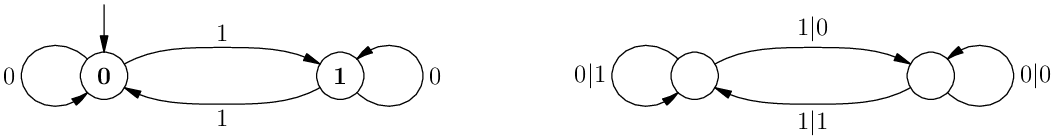}}
\caption{The automaton computing $\xiTM$ and the automaton generating the 
lamplighter group.}
\label{fig-automata}
\end{figure}

Here is another curious connection of the Thue-Morse sequence to group 
theory.  The sequence $\xiTM$ is $2$-automatic, which means there is an 
automaton that, given a number $n$ in its binary form as input, computes 
the $n$-th symbol of $\xiTM$.  If the symbols are counted starting from 
$0$, the task is done by a $2$-state automaton shown in Figure 
\ref{fig-automata} on the left (see \cite{AlloucheShallit:automatic}).  
It is an acceptor-like automaton in which the result is determined by the 
terminal state.  Now let us consider a transducer automaton with the same 
transition machine shown in Figure \ref{fig-automata} on the right.  This 
automaton outputs one symbol per each symbol of input.  Choosing one of the 
states as initial, the automaton generates a transformation on the set of 
binary strings.  The transformation happens to be related to a Gray code.  
Recall that a Gray code is an ordering of numbers $0,1,\dots,2^k-1$ such 
that the binary representations of any two neighbors differ in a single 
digit.  Given a number $n$ as a $k$-digit binary input, the automaton 
outputs the position of $n$ in the Gray code.  The two states of the 
automaton give rise to two transformations, which in turn generate a 
transformation group.  The group turns out to be the lamplighter group, 
i.e., the wreath product of an infinite cyclic group with a group of order 
$2$ (see Theorem 7 in \cite{Classification}).

The notion of a \emph{topological full group} (TFG) is relatively recent.  
Such a group can be associated to any homeomorphism $T:X\to X$, where $X$ 
is a Cantor set.  Informally, the TFG of $T$, denoted $[[T]]$, consists of 
all homeomorphisms of $X$ that locally coincide with elements of the cyclic 
group generated by $T$ (see Definition \ref{def-TFG}).  After the paper 
\cite{GPS} by T. Giordano, I. Putnam and C. Skau, the TFG emerged as an 
algebraic invariant of topological dynamics of the system $(X,T)$ provided 
it is a Cantor minimal system (see Theorem \ref{flip-conjugacy}).

The TFG's of Cantor minimal systems are groups with unique properties.  
Already in \cite{GPS} it was shown that they are indicable (admit a 
homomorphism onto $\bZ$).  The systematic study of algebraic properties of 
groups $[[T]]$ was spearheaded by H. Matui \cite{Matui} who showed that the 
commutator subgroup $[[T]]'$ is simple and, in the case of a minimal 
subshift, finitely generated.  The studies were continued by Matui 
\cite{Matui2}, K. Medynets and the first author \cite{GrigMed,GrigMed2}, K. 
Juschenko and N. Monod \cite{JuschMonod}, and others.  Any finite group and 
the free abelian group of infinite rank embed into $[[T]]$.  If the system 
$(X,T)$ is not topologically conjugate to a subshift, then the group 
$[[T]]$ is locally virtually abelian (i.e., every finitely generated 
subgroup $H\subset [[T]]$ contains an abelian subgroup of finite index in 
$H$).  Otherwise the lamplighter group embeds into $[[T]]'$ (see 
\cite{Matui2}).  The latter implies exponential growth of the group 
$[[T]]'$ as the lamplighter group contains a free semigroup on two 
generators.  Further remarkable properties of TFG's include non-elementary 
amenability \cite{JuschMonod} and factorization of $[[T]]'$ into a product 
of two locally finite subgroups \cite{Matui,GrigMed}.  Neither $[[T]]$ nor 
$[[T]]'$ can be finitely presented but an effective description of the 
system $(X,T)$ may help to obtain an infinite presentation.  For instance, 
if $(X,T)$ is a subshift such that all admissible words form a recursive 
language, then the commutator subgroup $[[T]]'$ allows a recursive 
presentation and, as a consequence, has a decidable word problem (see 
\cite{GrigMed2}).  The maximal subgroups of TFG's were studied by the 
authors in \cite{GrigVor24}.

N. Matte Bon \cite{MatteBon} observed that the TFG of a Cantor minimal 
system may contain a subgroup of intermediate (faster than polynomial but 
slower than exponential) growth.  Namely, each group from a family 
$\cG_\om$, $\om\in\{0,1,2\}^{\bN}$ constructed in \cite{Grig84} embeds into 
the TFG of a minimal subshift.  The growth of $\cG_\om$ is intermediate if 
the sequence $\om$ is not eventually constant.  The groups $\cG_\om$ have  
many other unusual properties.  They are branch, just-infinite groups, most 
of them are torsion groups (that is, every element has finite order).  Each 
group acts transitively on levels of the binary rooted tree, and the 
Schreier graphs of those actions have line-like shapes.  This defines a 
linear order on each level that gives rise to a Gray code (see 
\cite{MatteBon}).  The code is the same for all groups $\cG_\om$ and, 
somehow, it is related to the Gray code computed by the automaton in Figure 
\ref{fig-automata}.  Each group $\cG_\om$ also acts naturally on the 
boundary of the binary rooted tree, and the Schreier graphs associated with 
the orbits of that action also have line-like shapes.  The latter fact was 
used in numerous articles for various purposes, in particular, to initiate 
the study of a joint spectrum of self-similar groups (see \cite{BarGri00}) 
or to construct first examples of simple groups of intermediate growth 
\cite{Nek16}.  Linear structure of the Schreier graphs will be important in 
this paper as well.

Among groups $\cG_\om$, a special role belongs to the groups 
$\cG_{(012)^\infty}$ and $\cG_{(01)^\infty}$ associated to periodic 
sequences $012012\dots$ and $010101\dots$.  Unlike generic groups in the 
family, these two are self-similar groups, namely, each of them naturally 
embeds as a subgroup of finite index into the wreath product of itself with 
the group of order $2$.  The two groups share many properties but there are 
also essential differences.  For example, the first group is torsion while 
the second one has elements of infinite order.  Each maximal subgroup in 
$\cG_{(012)^\infty}$ has finite index while $\cG_{(01)^\infty}$ admits 
maximal subgroups of infinite index.

The Schreier graphs of the group $\cG_{(012)^\infty}$ were studied by the 
second author in \cite{Vor12}.  D. Lenz, T. Nagnibeda and the first author 
observed in \cite{GLN17} that $\cG_{(012)^\infty}$ can be embedded into the 
TFG of the substitution subshift generated by a substitution 
$\subst_{\mathrm{Ly}}$ given by $a\to aca$, $b\to d$, $c\to b$, $d\to c$.  
They used this fact to link the spectral problem for Schreier graphs of 
$\cG_{(012)^\infty}$ to the spectral theory of random Schroedinger 
operators.  The substitution $\subst_{\mathrm{Ly}}$ is not primitive but, 
in fact, the associated subshift coincides with the subshift generated by a 
primitive substitution $a\to ac$, $b\to ad$, $c\to ab$, $d\to ac$ (see 
\cite{AlloucheDQ:hidden}).  Note that the TFG of this subshift is finitely 
generated.  An explicit generating set was found in \cite{Vor20}.

The substitution $\subst_{\mathrm{Ly}}$ occurs naturally in the study of 
the group $\cG_{(012)^\infty}$.  Indeed, the group is generated by four 
elements denoted $a,b,c,d$ and the substitution $\subst_{\mathrm{Ly}}$ 
induces an injective endomorphism of $\cG_{(012)^\infty}$.  This was first 
observed by I. Lysenok \cite{Lys85} who obtained the following  
presentation of the group $\cG_{(012)^\infty}$ by generators and relators:
\[
\bigl\langle a,b,c,d \mid 1=a^2=b^2=c^2=d^2=bcd
=\subst_{\mathrm{Ly}}^k\bigl((ad)^4\bigr)
=\subst_{\mathrm{Ly}}^k\bigl((adacac)^4\bigr), \ k=1,2,\dots \bigr\rangle.
\]
Recursive group presentations of this kind are called 
\emph{$L$-presentations} by analogy with the $L$-systems that are used for 
the study of formal languages and in some other areas of science.  Note 
that the above presentation is optimal in the sense that no relator is 
redundant (as shown in \cite{Grig99}).  Interestingly, this optimality is 
preserved after the profinite completion of the group $\cG_{(012)^\infty}$ 
(as proved by M. G. Benli \cite{Benli}).

Remarkably, the group $\cG_{(01)^\infty}$ also admits an $L$-presentation  
\[
\bigl\langle a,c,d \mid \sub^k(r)=1, \ r\in R, \ k\ge0 \bigr\rangle,
\]
where the set of basic relators $R$ is
\[
\{a^2,(cd)^2,(ad)^4,[d,ac]^2,[d,cac]^2,[d,acaca]^2,
[d,acacac]^2,[d,cacacac]^2\}
\]
and $\sub$ is the substitution given by $a\to aca$, $c\to d$, $d\to c$.  
It was obtained by Y. Muntyan \cite{Muntyan} in his dissertation.  It is 
not known yet whether this presentation is optimal.  We do not need the 
presentation in this paper, but the substitution $\sub$ will be very 
important in what follows.  Just like $\subst_{\mathrm{Ly}}$, the 
substitution $\sub$ is not primitive.  However, as we will show in this 
paper, the associated substitution subshift is topologically conjugate to 
the subshift generated by the \emph{period doubling} substitution $\pd$: 
$x\to xy$, $y\to xx$, which is primitive.  Note that the substitution 
subshifts generated by the substitutions $\subst_{\mathrm{Ly}}$ and $\sub$ 
belong to a family of Toeplitz subshifts studied in \cite{Vor10}.

The main result of this paper adds another item to the list of properties 
of the group $\cG_{(01)^\infty}$ showing its connection to the classical 
Thue-Morse sequence.

\begin{theorem}\label{main}
The group of intermediate growth $\cG_{(01)^\infty}$ embeds into the 
topological full group of the substitution subshift generated by the 
Thue-Morse substitution $\TM$ as well as into the TFG of the substitution 
subshift generated by the period doubling substitution $\pd$.  
\end{theorem}

The embeddings in Theorem \ref{main} can be described explicitly (see 
Proposition \ref{two-more-inj} at the end of the paper).

The proof of Theorem \ref{main} relies on the following general result.

\begin{theorem}\label{main0}
Let $(X_1,T_1)$ and $(X_2,T_2)$ be aperiodic Cantor systems.  If 
$(X_2,T_2)$ is a continuous factor of the system $(X_1,T_1)$ then the TFG 
$[[T_2]]$ embeds into $[[T_1]]$.
\end{theorem}

Given a factor map $f:X_1\to X_2$, the induced embedding 
$\mathcal{E}_f:[[T_2]]\to[[T_1]]$ in Theorem \ref{main0} can 
be described explicitly.  Namely, every homeomorphism $S_2\in[[T_2]]$ is 
given by $S_2(y)=T_2^{\nu(y)}(y)$, $y\in X_2$ for some continuous function 
$\nu:X_2\to\bZ$.  Then the homeomorphism $S_1=\mathcal{E}_f(S_2)$ is given 
by $S_1(x)=T_1^{\nu(f(x))}(x)$, $x\in X_1$.

Theorem \ref{main} inspires a number of questions and topics to explore.  
We would like to formulate some of them.

\begin{problem}
Suppose $T_{\subst}$ is a substitution subshift generated by a primitive 
substitution $\subst$.
\begin{itemize}
\item[(i)]
Under what condition on the substitution $\subst$ the topological full  
group $[[T_{\subst}]]$ of the subshift $T_{\subst}$ contains a group of 
intermediate growth?
\item[(ii)]
Under what conditions on $\subst$ the group $[[T_{\subst}]]$ contains a 
subgroup of Burnside type (i.e., finitely generated infinite torsion group)?
\end{itemize}
In particular, what can be said about the TFG of the subshift generated by 
the Fibonacci substitution $0\to01$, $1\to0$?  Or the Rudin-Shapiro 
substitution $a\to ab$, $b\to ac$, $c\to db$, $d\to dc$?  Or the Chacon 
substitution $0\to 0012$, $1\to 12$, $2\to 012$?
\end{problem}

There are other remarkable groups that have finite $L$-presentations.  One 
example is the iterated monodromy group (IMG) of the polynomial $z^2+i$.  
This group first appeared in \cite{BarGrigNek} and then in many other 
articles including \cite{GrigSavSun} where a finite $L$-presentation for  
$\mathrm{IMG}(z^2+i)$ was found.  It is quite similar to the above 
presentation of $\cG_{(01)^\infty}$ but uses a substitution $a\to b$, $b\to 
c$, $c\to aba$.  This substitution preserves no infinite sequence.  
Instead, there are three sequences cyclically permuted by the 
substitution.  Each of the three sequences generates the same minimal 
subshift.  Note that the group $\mathrm{IMG}(z^2+i)$ acts naturally on the 
boundary of a rooted binary tree but, unlike the groups $\cG_\om$, the 
Schreier graphs of orbits are not of linear shape.

\begin{problem}\label{problem-IMG}
Does the group $\mathrm{IMG}(z^2+i)$ embed into the TFG of a minimal Cantor 
system?  Does it embed into the TFG of a substitution subshift generated by 
a primitive substitution?
\end{problem}

We should also mention that connections between the Thue-Morse substitution 
and self-similar groups were previously explored in the paper 
\cite{Barth:Thue-Morse} where the substitution $\TM$ was used to construct 
a certain self-similar group that turned out to be the IMG of the rational 
map $f(z)=\bigl(z-\frac12 z^2\bigr)^{-1}$.  Note that this group can be 
generated by a $3$-state automaton and, in fact, it is featured in the 
classification \cite{Classification} of such groups (see Automaton no.\@ 
2853 in \cite{Classification}).

The paper is organized as follows.  In Section \ref{sect-sub} we recall 
basic facts about substitution subshifts, discuss a few specific examples 
(including the subshift generated by the Thue-Morse substitution) and 
establish relations between them (Proposition \ref{3-subshifts-factors}).  
For the reader's convenience, this section is self-contained though some 
material in it is known and can be found in textbooks on symbolic dynamics.
In Section \ref{sect-TFG} we consider topological full groups associated
with aperiodic homeomorphisms of a Cantor set and prove the embedding 
Theorem \ref{main0} (Theorem \ref{TFG-embeds}).  This allows to determine 
at the end of the section a certain group generated by three involutive 
homeomorphisms that has intermediate growth and embeds into the TFG of the 
Thue-Morse subshift (Theorem \ref{sub-main}).  In Section \ref{sect-growth} 
we describe the family $\cG_\om$, $\om\in\{0,1,2\}^{\bN}$ of groups of 
intermediate growth constructed in \cite{Grig84} and single out the group 
$\cG_{(01)^\infty}$ as the one that is isomorphic to the group described in 
Section \ref{sect-TFG} (Proposition \ref{inj-homomorph}).  In Section 
\ref{sect-Schr} we consider the Schreier graphs of groups actions.  We 
study the Schreier graphs with linear structure and use them to prove 
Proposition \ref{inj-homomorph}.  Then we prove Theorem \ref{sub-main} and 
the main Theorem \ref{main}.  The section is concluded with the proof of 
Proposition \ref{two-more-inj}.

\medskip

\emph{Acknowledgements}.
The first author is supported by the Travel Support for Mathematicians 
grant MP-TSM-00002045 from the Simons Foundation.

\section{Substitution subshifts}\label{sect-sub}

An \emph{alphabet} $\cA$ is a finite set consisting of more than one 
element.  Elements of $\cA$ are referred to as \emph{letters} or 
\emph{symbols}.  Let $\cA^*$ denote the set of all finite strings 
$x_1x_2\ldots x_n$ of letters from $\cA$ (including the empty one 
$\varnothing$).  We refer to them as \emph{words} over $\cA$ and write 
without any delimiters.  The number of letters in a word is called its 
\emph{length}.  Elements of $\cA$ are identified with words of length $1$.  
For any words $u,w\in\cA^*$ let $uw$ denote their \emph{concatenation}: if 
$u=x_1x_2\ldots x_m$ and $w=y_1y_2\ldots y_n$, where each $x_i$ and $y_j$ 
is a letter, then $uw=x_1x_2\ldots x_my_1y_2\ldots y_n$.  This operation is 
associative and turns $\cA^*$ into a monoid.  Likewise, for any finite list 
of words $w_1,w_2,\dots,w_k$ their concatenation is denoted $w_1w_2\ldots 
w_k$.  The concatenation of $k$ copies of the same word $w$ is denoted 
$w^k$.  We say that a word $u\in\cA^*$ is a \emph{prefix} of a word 
$w\in\cA^*$ if $w=uv$ for some $v\in\cA^*$.  We say that $u$ is a 
\emph{subword} of $w$ if $w=vuv'$ for some $v,v'\in\cA^*$.

For any nonempty set $S$ let $\cA^S$ denote the set of all functions 
$\xi:S\to\cA$.  In the cases $S=\bN$ (positive integers), $S=-\bN$ 
(negative integers), and $S=\bZ$ (all integers), we represent an element 
$\xi\in\cA^S$ as respectively an infinite string $x_1x_2x_3\ldots$, a 
left-infinite string $\ldots x_{-3}x_{-2}x_{-1}$, and a bi-infinite string 
$\ldots x_{-2}x_{-1}x_0.x_1x_2x_3\ldots$, where $x_i=\xi(i)$ for each $i\in 
S$.  Notation for bi-infinite strings includes a dot that serves as a 
reference point.  We say that a nonempty word $w\in\cA^*$ is a 
\emph{subword} of the string $\xi$ if $w=x_nx_{n+1}\ldots x_m$ for some 
integers $n$ and $m$, $n\le m$ such that $\{n,n+1,\dots,m\}\subset S$.

Given a word $w=x_1x_2\ldots x_n\in\cA^*$, an infinite string 
$\eta=y_1y_2\ldots\in\cA^{\bN}$, and a left-infinite string $\zeta=\ldots 
z_2z_1\in\cA^{-\bN}$, we define \emph{concatenations} $w\eta$, $\zeta w$, 
and $\zeta.\eta$ as respectively the infinite string $x_1x_2\ldots x_n 
y_1y_2\ldots$, the left-infinite string $\ldots z_2z_1x_1x_2\ldots x_n$, 
and the bi-infinite string $\ldots z_2z_1.y_1y_2\ldots$.  We say that a 
word $w\in\cA^*$ is a \emph{prefix} of an infinite string $\xi\in\cA^{\bN}$ 
if $\xi=w\eta$ for some $\eta\in\cA^{\bN}$.  Note that for any words 
$u,w\in\cA^*$ and strings $\eta\in\cA^{\bN}$ and $\zeta\in\cA^{-\bN}$, we 
have $(uw)\eta=u(w\eta)$ and $\zeta(uw)=(\zeta u)w$ so that expressions 
$uw\eta$ and $\zeta uw$ make sense.  On the other hand, usually 
$\zeta.(u\eta)\ne(\zeta u).\eta$.

Let $w_1,w_2,w_3,\dots$ be an infinite list of words in $\cA^*$ in which at 
most finitely many words are empty.  Then there exists a unique string 
$\eta\in\cA^{\bN}$ such that each finite concatenation $w_1w_2\ldots w_k$, 
$k\ge1$ is a prefix of $\eta$.  We denote $\eta$ as an \emph{infinite 
concatenation} $w_1w_2w_3\ldots$.  If all $w_i$ are copies of the same word 
$w$, we also use notation $w^\infty$.  Likewise, the \emph{left-infinite 
concatenation} $\ldots w_3w_2w_1$ is defined as the only left-infinite 
string $\zeta\in\cA^{-\bN}$ such that for any $k\ge1$ we have 
$\zeta=\zeta'(w_k\ldots w_2w_1)$ for some $\zeta'\in\cA^{-\bN}$.  

Next we are going to introduce the \emph{shift} and \emph{subshifts} over 
an alphabet $\cA$, which are topological dynamical systems on the set 
$\cA^{\bZ}$ and its subsets.  First we need a topology on $\cA^{\bZ}$.  We 
regard $\cA$ as a discrete topological space, then $\cA^{\bZ}$ is endowed 
with the \emph{product topology}.  Given any nonempty finite set 
$S_0\subset\bZ$ and any function $\om_0:S_0\to\cA$, let 
$\Sigma_{\cA}(S_0,\om_0)$ denote the set of all functions $\om:\bZ\to\cA$ 
such that the restriction of $\om$ to $S_0$ coincides with $\om_0$.  Sets 
of the form $\Sigma_{\cA}(S_0,\om_0)$, referred to as \emph{cylinders}, are 
open subsets of $\cA^{\bZ}$.  Moreover, they form a base of the topology on 
$\cA^{\bZ}$.  The topological space $\cA^{\bZ}$ is compact and metrizable.  
Convergence in $\cA^{\bZ}$ is the pointwise convergence of functions.  In 
terms of bi-infinite strings, a sequence of strings $\om^{(k)}=\ldots 
x^{(k)}_{-2}x^{(k)}_{-1}x^{(k)}_0.x^{(k)}_1x^{(k)}_2x^{(k)}_3\ldots$, 
$k=1,2,\dots$ converges to a string $\eta=\ldots 
y_{-2}y_{-1}y_0.y_1y_2y_3\ldots$ in $\cA^{\bZ}$ if for every $n\in\bZ$ the 
sequence of letters $x^{(1)}_n,x^{(2)}_n,x^{(3)}_n,\dots$ eventually 
coincides with the constant sequence $y_n,y_n,y_n,\dots$.

\begin{definition}\label{def-shift}
The (two-sided) \textbf{shift} over an alphabet $\cA$ is a transformation 
of $\cA^{\bZ}$, denoted $\si_{\cA}$ or simply $\si$, that sends any 
function $\om:\bZ\to\cA$ to a function $\eta:\bZ\to\cA$ given by 
$\eta(n)=\om(n+1)$ for all $n\in\bZ$.  In terms of bi-infinite strings, 
$\si(\ldots x_{-2}x_{-1}x_0.x_1x_2x_3\ldots)=\ldots x_{-1}x_0x_1.x_2x_3x_4
\ldots$.  A \textbf{subshift} over $\cA$ is the restriction $\si_{\cA}|Y$ 
of the shift $\si_{\cA}$ to a nonempty closed set $Y\subset\cA^{\bZ}$ such 
that $\si_{\cA}(Y)=Y$.  
\end{definition}

The shift $\si_{\cA}$ is invertible and maps cylinders onto cylinders, 
hence it is a homeomorphism of the topological space $\cA^{\bZ}$ onto 
itself.  Any subshift $\si_{\cA}|Y$ is a homeomorphism of the compact 
topological space $Y$ onto itself.

Given an infinite string $\xi\in\cA^{\bN}$, let $\Omega(\xi)$ denote the 
set of all bi-infinite strings $\om\in\cA^{\bZ}$ such that every subword of 
$\om$ is also a subword of $\xi$.

\begin{lemma}\label{omega-xi}
For any $\xi\in\cA^{\bN}$, the set $\Omega(\xi)$ is nonempty and closed, 
and $\si_{\cA}(\Omega(\xi))=\Omega(\xi)$.  Every subword of any string in 
$\Omega(\xi)$ occurs infinitely many times as a subword of $\xi$.
\end{lemma}

\begin{proof}
For any bi-infinite string $\om\in\cA^{\bZ}$, the shifted string 
$\si_{\cA}(\om)$ has exactly the same subwords as $\om$.  It follows that 
$\si_{\cA}(\Omega(\xi))=\Omega(\xi)$.  Next we prove that $\Omega(\xi)$ is 
a closed set.  Given a nonempty word $w\in\cA^*$ of length $k\ge1$ and any 
$n\in\bZ$, let $\Sigma_{w,n}$ be the set of all bi-infinite strings $\ldots 
y_{-1}y_0.y_1y_2\ldots$ in $\cA^{\bZ}$ such that $y_ny_{n+1}\ldots 
y_{n+k-1}=w$.  The set $\Sigma_{w,n}$ is a cylinder.  The complement of 
$\Omega(\xi)$ is the union of cylinders $\Sigma_{w,n}$ over all words 
$w\in\cA^*$ that are not subwords of $\xi$ and over all $n\in\bZ$.  Hence 
the complement is an open set.  Then the set $\Omega(\xi)$ is closed.

Next we prove that every subword $u$ of any string $\om\in\Om(\xi)$ occurs 
infinitely often as a subword of $\xi$.  Let $\xi=x_1x_2x_3\ldots$ and 
$\om=\ldots y_{-1}y_0.y_1y_2\ldots$.  If $u$ is a subword of length $s\ge1$ 
of $\om$, we have $y_my_{m+1}\ldots y_{m+s-1}=u$ for some $m\in\bZ$.  Then 
for any $k\ge1$ the word $u_k=y_{m-k}y_{m-k+1}\ldots 
y_{m-1}y_my_{m+1}\ldots y_{m+s-1}$ is a subword of length $s+k$ of $\om$.  
Since $\om\in\Om(\xi)$, the same word $u_k$ is also a subword of the string 
$\xi$.  That is, $x_{n_k}x_{n_k+1}\ldots x_{n_k+k+s-1}=u_k$ for some 
$n_k\ge1$.  Then $x_{n_k+k}x_{n_k+k+1}\ldots x_{n_k+k+s-1}$ is another 
occurrence of $u$ as a subword of $\xi$.  Since $n_k+k\ge k+1$, there are 
infinitely many numbers of the form $n_k+k$, $k\ge1$.  Therefore $u$ occurs 
infinitely often as a subword of the string $\xi$.

It remains to prove that the set $\Om(\xi)$ is not empty.  Let $W_\xi$ be 
the set of all nonempty words that occur infinitely many times as subwords 
of $\xi$.  Since the alphabet $\cA$ is finite while the string $\xi$ is 
infinite, $W_\xi$ contains at least one $1$-letter word.  Let us show that 
for any word $w\in W_\xi$ there exist letters $l,r\in\cA$ such that the 
word $lwr$ belongs to $W_\xi$ as well.  Let $s$ be the length of $w$.  
There is a sequence $\{n_k\}$, $2\le n_1<n_2<n_3<\ldots$, such that 
$x_{n_k}x_{n_k+1}\ldots x_{n_k+s-1}=w$ for all $k\ge1$.  Since the 
alphabet $\cA$ is finite, we can find letters $l,r\in\cA$ such that 
$(x_{n_k-1},x_{n_k+s})=(l,r)$ for infinitely many values of $k$.  Then 
$lwr\in W_\xi$.

Now it follows by induction from the above that there exist two sequences 
of letters $r_0,r_1,r_2,\dots$ and $l_1,l_2,l_3,\dots$ such that for any 
$k\ge1$ the word $w_k=l_kl_{k-1}\ldots l_2l_1r_0r_1r_2\ldots r_k$ belongs 
to $W_\xi$.  Then the bi-infinite string $\om=\ldots l_2l_1r_0.r_1r_2 
\ldots$ belongs to the set $\Om(\xi)$ so that this set is not empty.  
Indeed, any subword $u$ of $\om$ is also a subword of $w_k$ for all 
sufficiently large $k$.  Since each $w_k$ is a subword of $\xi$, so is $u$.
\end{proof}

In view of Lemma \ref{omega-xi}, the restriction of the shift $\si_{\cA}$ 
to the set $\Om(\xi)$, where $\xi\in\cA^{\bN}$, is a subshift.  We refer to 
it as the subshift generated by the string $\xi$.  Even though $\xi$ does 
not belong to $\cA^{\bZ}$, the shift directly relates it to the set 
$\Om(\xi)$ via the following construction.

\begin{lemma}\label{limit-points}
For any infinite string $\xi\in\cA^{\bN}$ and left-infinite string 
$\zeta\in\cA^{-\bN}$, the set of all limit points of the sequence 
$\zeta.\xi,\si_{\cA}(\zeta.\xi),\si_{\cA}^2(\zeta.\xi),\dots$ in 
$\cA^{\bZ}$ coincides with $\Omega(\xi)$.
\end{lemma}

\begin{proof}
Let $\ldots x_{-1}x_0.x_1x_2\ldots$ be the bi-infinite string $\zeta.\xi$.  
Then $\xi=x_1x_2\ldots$.  For any $k\ge0$ we have $\si_{\cA}^k(\zeta.\xi)= 
\ldots x^{(k)}_{-1}x^{(k)}_0.x^{(k)}_1x^{(k)}_2\ldots$, where 
$x^{(k)}_n=x_{n+k}$ for all $n\in\bZ$.  Consider an arbitrary bi-infinite 
string $\om=\ldots y_{-1}y_0.y_1y_2\ldots\in\cA^{\bZ}$.  First assume that 
$\om\in\Om(\xi)$.  For any $k\ge1$ the word $w_k=y_{-k}\ldots 
y_{-1}y_0.y_1\ldots y_k$ is a subword of $\om$ of length $2k+1$.  Hence it 
is also a subword of $\xi$.  That is, $x_{n_k}x_{n_k+1}\ldots 
x_{n_k+2k}=w_k$ for some $n_k\ge1$.  Moreover, $n_k$ can be chosen 
arbitrarily large since $w_k$ occurs infinitely many times as a subword of 
$\xi$ (due to Lemma \ref{omega-xi}).  Therefore we can build the sequence 
$\{n_k\}_{k\ge1}$ inductively in such a way that $1<n_1+1<n_2+2<n_3+3< 
\ldots$.  Then $\{\si_{\cA}^{n_k+k}(\zeta.\xi)\}_{k\ge1}$ is a subsequence 
of the sequence $\zeta.\xi,\si_{\cA}(\zeta.\xi),\si_{\cA}^2(\zeta.\xi), 
\dots$.  This subsequence converges to $\om$ since 
$x^{(n_k+k)}_m=x_{m+n_k+k}=y_m$ whenever $k\ge|m|$.

Conversely, assume that $\om$ is a limit point of the sequence 
$\zeta.\xi,\si_{\cA}(\zeta.\xi),\si_{\cA}^2(\zeta.\xi),\dots$.  That is, 
$\si_{\cA}^{k_i}(\zeta.\xi)\to\om$ as $i\to\infty$, where $0\le 
k_1<k_2<k_3<\ldots$.  Take any nonempty subword $u$ of the string $\om$.  
We have $y_{n_1}y_{n_1+1}\ldots y_{n_2}=u$ for some integers $n_1$ and 
$n_2$, $n_1\le n_2$.  The convergence implies that the word 
$x^{(k_i)}_{n_1}x^{(k_i)}_{n_1+1}\ldots x^{(k_i)}_{n_2}= 
x_{n_1+k_i}x_{n_1+1+k_i}\ldots x_{n_2+k_i}$ coincides with $u$ for all 
sufficiently large $i$.  If $i$ is large enough then, additionally, 
$n_1+k_i\ge1$ so that $u$ occurs as a subword of the string $\xi$.  Thus 
every subword of $\om$ is also a subword of $\xi$, which means that 
$\om\in\Om(\xi)$.
\end{proof}

\begin{definition}\label{def-sub}
A \textbf{substitution} over an alphabet $\cA$ is a map 
$\mathsf{s}:\cA\to\cA^*\setminus\{\varnothing\}$.  The substitution can be 
applied term by term to any finite or infinite string $x_1x_2x_3\ldots$ of 
letters from $\cA$, which yields a finite or infinite concatenation 
$\mathsf{s}(x_1)\mathsf{s}(x_2)\mathsf{s}(x_3)\ldots$.  Hence the map 
$\mathsf{s}$ induces a transformation of $\cA^*$ and a transformation of 
$\cA^{\bN}$.  We use the same notation for the substitution and both 
induced transformations.
\end{definition}

Let $\mathsf{s}$ be a substitution over an alphabet $\cA$.  Suppose $z$ is 
a letter in $\cA$ such that the word $\mathsf{s}(z)$ begins with $z$ and 
has length greater than $1$.  That is, $\mathsf{s}(z)=zw$ for some nonempty 
word $w\in\cA^*$.  Consider a sequence of words 
$z,\mathsf{s}(z),\mathsf{s}^2(z),\dots$ obtained by repeatedly applying the 
substitution to the letter $z$.  For any $k\ge1$ we have $\mathsf{s}^k(z)= 
\mathsf{s}^{k-1}(zw)=\mathsf{s}^{k-1}(z)\mathsf{s}^{k-1}(w)$.  Note that 
$\mathsf{s}^{k-1}(w)\ne\varnothing$ since $w\ne\varnothing$.  Hence every 
word in the sequence is a prefix of the next word and of smaller length 
than the next word.  It follows that, as $k\to\infty$, words 
$\mathsf{s}^k(z)$ in a sense converge to an infinite string.  To be 
precise, there exists a unique string $\xi\in\cA^{\bN}$ such that each 
$\mathsf{s}^k(z)$ is a prefix of $\xi$.  Clearly, $\mathsf{s}(\xi)=\xi$.  
Moreover, $\xi$ is the only infinite string that begins with the letter $z$ 
and is preserved by the substitution $\mathsf{s}$.

\begin{definition}\label{def-subshift}
The \textbf{substitution subshift} generated by a substitution $\mathsf{s}$ 
over an alphabet $\cA$ is a subshift $\si_{\cA}|\Omega(\xi)$, where $\xi$ 
is an infinite string preserved by $\mathsf{s}$.
\end{definition}

In this paper we are interested in substitution subshifts generated by 
three substitutions.  The first one is the \emph{Thue-Morse} substitution 
$\TM$ over the alphabet $\{0,1\}$ given by $\TM(0)=01$, $\TM(1)=10$.  The 
second one is the so-called \emph{period doubling} substitution $\pd$ over 
the alphabet $\{x,y\}$ given by $\pd(x)=xy$, $\pd(y)=xx$.  The third one is 
a substitution $\sub$ over the alphabet $\{a,c,d\}$ given by $\sub(a)=aca$, 
$\sub(c)=d$, $\sub(d)=c$.

Since the word $\TM(0)=01$ begins with $0$ and the word $\TM(1)=10$ begins 
with $1$, the substitution $\TM$ preserves a unique infinite string 
$\xiTM\in\{0,1\}^{\bN}$ that begins with $0$ and another infinite string 
$\xiTM'$ that begins with $1$.

\begin{lemma}\label{TM-unique}
$\xiTM$ and $\xiTM'$ are obtained from each other by changing all 0s to 1s 
and all 1s to 0s.  Both strings generate the same substitution subshift.
\end{lemma}

\begin{proof}
Consider a substitution $\mathsf{s}_{(01)}$ over the alphabet $\{0,1\}$ 
given by $\mathsf{s}_{(01)}(0)=1$, $\mathsf{s}_{(01)}(1)=0$.  We obtain 
that $\mathsf{s}_{(01)}\TM(0)=\TM\mathsf{s}_{(01)}(0)=10$ and 
$\mathsf{s}_{(01)}\TM(1)=\TM\mathsf{s}_{(01)}(1)=01$.  It follows that 
$\mathsf{s}_{(01)}\TM(\xi)=\TM\mathsf{s}_{(01)}(\xi)$ for any word or 
infinite string $\xi$ over the alphabet $\{0,1\}$.  In particular, 
$\TM(\mathsf{s}_{(01)}(\xiTM))=\mathsf{s}_{(01)}(\TM(\xiTM)) 
=\mathsf{s}_{(01)}(\xiTM)$.  Since the string $\mathsf{s}_{(01)}(\xiTM)$ 
begins with $1$, we have $\mathsf{s}_{(01)}(\xiTM)=\xiTM'$ due to 
uniqueness of $\xiTM'$.  In other words, $\xiTM'$ is obtained from $\xiTM$ 
by changing all 0s to 1s and all 1s to 0s.  Then also 
$\mathsf{s}_{(01)}(\xiTM')=\xiTM$.

For any $k\ge1$ we obtain $\TM^{k+1}(0)=\TM^k(01)=\TM^k(0)\TM^k(1)$ and 
$\TM^{k+1}(1)=\TM^k(10)=\TM^k(1)\TM^k(0)$.  As a consequence, the word 
$\TM^k(1)$ is a subword of $\TM^{k+1}(0)$ while the word $\TM^k(0)$ is a 
subword of $\TM^{k+1}(1)$.  Any subword $w'$ of the string $\xiTM'$ is also 
a subword of its prefix $\TM^k(1)$ for all sufficiently large $k$.  Since 
$\TM^k(1)$ is a subword of $\TM^{k+1}(0)$, which is a prefix of the string 
$\xiTM$, it follows that $w'$ is a subword of $\xiTM$ as well.  Conversely, 
if $w$ is a subword of the string $\xiTM$, then for some $k\ge1$ it is a 
subword of the prefix $\TM^k(0)$, which is a subword of $\TM^{k+1}(1)$, 
which is a prefix of $\xiTM'$.  Hence $w$ also a subword of $\xiTM'$.  Thus 
the strings $\xiTM$ and $\xiTM'$ have the same subwords.  Therefore 
$\Om(\xiTM)=\Om(\xiTM')$ so that both strings generate the same subshift.
\end{proof}

Since the word $\pd(x)=xy$ begins with $x$, the substitution $\pd$ 
preserves a unique infinite string $\xipd\in\{x,y\}^{\bN}$ that begins with 
$x$.  Since the word $\pd(y)=xx$ does not begin with $y$, the substitution 
$\pd$ preserves no infinite string that begins with $y$.  Likewise, the 
substitution $\sub$ preserves a unique infinite string 
$\xiGE\in\{a,c,d\}^{\bN}$ that begins with $a$ and no infinite string that 
begins with $c$ or $d$.

Let us explore the dynamics of the substitution subshifts 
$\si_{\{0,1\}}|\Om(\xiTM)$, $\si_{\{x,y\}}|\Om(\xipd)$ and 
$\si_{\{a,c,d\}}|\Om(\xiGE)$.  First we need detailed descriptions of the 
strings $\xiTM$, $\xipd$ and $\xiGE$.

\begin{lemma}\label{xi-TM}
The $n$-th symbol of the string $\xiTM$ is $0$ if the number of 1s in the 
binary expansion of the number $n-1$ is even, and $1$ otherwise.
\end{lemma}

\begin{proof}
For any $n\in\bN$ let $x_n$ denote the $n$-th symbol of $\xiTM$.  Then 
$\xiTM=x_1x_2x_3\ldots$.  As $\TM(\xiTM)=\xiTM$, we also have 
$\xiTM=\TM(x_1)\TM(x_2)\TM(x_3)\ldots$.  Since $\TM(0)$ and $\TM(1)$ are 
both words of length $2$, we obtain that $\TM(x_i)=x_{2i-1}x_{2i}$ for all 
$i\ge1$.  Since $\TM(0)=01$ and $\TM(1)=10$, we further obtain that 
$x_{2i-1}=x_i$ while $x_{2i}\ne x_i$.

For any $i\ge2$ the binary expansion of the number $(2i-1)-1=2(i-1)$ is 
obtained from the binary expansion of $i-1$ by appending $0$ to it while
the binary expansion of $2i-1$ is obtained by appending $1$.  Since 
$x_{2i-1}=x_i$ and $x_{2i}\ne x_i$, it follows that the lemma holds for 
$n=2i-1$ and $n=2i$ whenever it holds for $n=i$.  Besides, the lemma holds 
for $n=1$ and $n=2$ as $x_1=0$ and $x_2=1$.  Since any $n\ge3$ can be 
represented as $2i-1$ or $2i$ for some integer $i$, $2\le i<n$, the lemma 
follows by strong induction on $n$.
\end{proof}

\begin{lemma}\label{xi-pd}
Let $n=2^km$, where $k$ is a nonnegative integer and $m$ is a positive odd 
integer.  Then the $n$-th letter of the string $\xipd$ is $x$ if $k$ is 
even, and $y$ if $k$ is odd.
\end{lemma}

\begin{proof}
For any $n\in\bN$ let $x_n$ denote the $n$-th letter of $\xipd$.  Then 
$\xipd=x_1x_2x_3\ldots$.  As $\pd(\xipd)=\xipd$, we also have 
$\xipd=\pd(x_1)\pd(x_2)\pd(x_3)\ldots$.  Since $\pd(x)$ and $\pd(y)$ are 
both words of length $2$, we obtain that $\pd(x_i)=x_{2i-1}x_{2i}$ for all 
$i\ge1$.  Since $\pd(x)=xy$ and $\pd(y)=xx$, we further obtain that 
$x_{2i}=y$ if $x_i=x$, and $x_{2i}=x$ if $x_i=y$.  Besides, $x_n=x$ for any 
odd $n$.  Now the lemma easily follows by induction on $k$.
\end{proof}

\begin{lemma}\label{xi-GE}
Let $n=2^km$, where $k$ is a nonnegative integer and $m$ is a positive odd 
integer.  Then the $n$-th letter of the string $\xiGE$ is $a$ if $k=0$, $c$ 
if $k$ is odd, and $d$ if $k$ is even and positive.
\end{lemma}

\begin{proof}
Let $W$ be the set of all words of the form $ay_1ay_2\ldots ay_s$, where 
each $y_i$ is $c$ or $d$.  We have $\sub(ac)=acad$ and $\sub(ad)=acac$.  It 
follows that the set $W$ is invariant under the substitution $\sub$.  
Besides, $\sub$ doubles the length of every word in $W$.  Since $ac$ is a 
prefix of the infinite string $\xiGE$ and $\sub(\xiGE)=\xiGE$, each of the 
words $\sub(ac),\sub^2(ac),\sub^3(ac),\dots$ is also a prefix of $\xiGE$.  
All these words belong to $W$ and their length gets arbitrarily large.  We 
conclude that $\xiGE=x_1x_2x_3\ldots$, where $x_n=a$ if $n$ is odd, and 
$x_n=c$ or $d$ if $n$ even.

We have $\xiGE=\sub(\xiGE)=\sub(x_1x_2)\sub(x_3x_4)\sub(x_5x_6)\ldots$. 
Since $\sub(ac)$ and $\sub(ad)$ are both words of length $4$, it follows 
that $\sub(x_{2i-1}x_{2i})=x_{4i-3}x_{4i-2}x_{4i-1}x_{4i}$ for all
$i\ge1$.  Since $\sub(ac)=acad$ and $\sub(ad)=acac$, it further follows 
that $x_{4i}=d$ if $x_{2i}=c$, and $x_{4i}=c$ if $x_{2i}=d$.  Besides, 
$x_n=c$ for any even $n$ not divisible by $4$.  Now the lemma easily
follows by induction on $k$.
\end{proof}

The description of the string $\xiGE$ obtained in Lemma \ref{xi-GE} allows 
to produce explicit examples of bi-infinite strings in the set $\Om(\xiGE)$.

\begin{lemma}\label{xiGE-examples}
Let $\overleftarrow{\xiGE}$ denote the left-infinite string obtained by 
writing $\xiGE$ backwards.  Then the bi-infinite strings 
$\overleftarrow{\xiGE}c.\xiGE$ and $\overleftarrow{\xiGE}d.\xiGE$ belong to 
$\Om(\xiGE)$.
\end{lemma}

\begin{proof}
Let $\om=\ldots x_{-1}x_0.x_1x_2\ldots$ be one of the strings 
$\overleftarrow{\xiGE}c.\xiGE$ and $\overleftarrow{\xiGE}d.\xiGE$.  That 
is, $\xiGE=x_1x_2x_3\ldots$, $x_{-n}=x_n$ for all $n\ge1$, and $x_0=c$ or 
$d$.  We need to show that every subword of $\om$ is also a subword of 
$\xiGE$.  Lemma \ref{xi-GE} implies that for any $n\ge1$ the letter 
$x_n$ depends only on the largest power of $2$ that divides $n$.  As a 
consequence, $x_n=x_{2^k-n}=x_{2^k+n}$ for any integers $n,k\ge1$ such that 
$n<2^k$.  It follows that for any $k\ge1$ the words $w_k=x_1x_2\ldots 
x_{2^k-1}$, $w^-_k=x_{1-2^k}\ldots x_{-2}x_{-1}$ and 
$w^+_k=x_{2^k+1}x_{2^k+2}\ldots x_{2^{k+1}-1}$ are the same.

Any subword $u$ of the string $\om$ is also a subword of $w^-_kx_0w_k$ for 
all sufficiently large $k$.  By the above, $w^-_kx_0w_k=w_kx_0w^+_k$.  By 
Lemma \ref{xi-GE}, $x_{2^k}=c$ if $k$ is odd, and $x_{2^k}=d$ if $k$ is 
even and positive.  It follows that $u$ is a subword of $w_kx_{2^k}w^+_k$ 
for some $k$.  Since $w_kx_{2^k}w^+_k$ is a prefix of the string $\xiGE$, 
the word $u$ is also a subword of $\xiGE$.
\end{proof}

\begin{definition}\label{def-minimal}
A homeomorphism $T:X\to X$ of a compact topological space $X$ onto itself 
is called \textbf{minimal} if there is no closed set $Y\subset X$ different 
from the empty set and $X$ such that $T(Y)=Y$.  An equivalent condition is 
that for any point $x\in X$ the \emph{two-sided orbit} $\{T^n(x)\mid 
n\in\bZ\}$ is dense in $X$.
\end{definition}

Minimality of a homeomorphism $T$ is a sign of very regular recurrence in 
the dynamical system defined by $T$.

\begin{proposition}\label{3-subshifts-minimal}
The subshifts $\si_{\{0,1\}}|\Om(\xiTM)$, $\si_{\{x,y\}}|\Om(\xipd)$ and 
$\si_{\{a,c,d\}}|\Om(\xiGE)$ are minimal.
\end{proposition}

Note that minimality of the first two subshifts in the proposition is well 
known.  We are going to present a unified proof for all three subshifts.
First let us observe that minimality of a subshift can be described in 
purely combinatorial terms.

\begin{lemma}\label{same-subwords}
A subshift $\si_{\cA}|Y$ is minimal if and only if any two bi-infinite 
strings in $Y$ have the same subwords.
\end{lemma}

\begin{proof}
For any nonempty word $w\in\cA^*$ of length $k\ge1$ and any 
$n\in\bZ$, let $\Sigma_{w,n}$ be the set of all bi-infinite strings $\ldots 
y_{-1}y_0.y_1y_2\ldots$ in $\cA^{\bZ}$ such that $y_ny_{n+1}\ldots 
y_{n+k-1}=w$.  Each of the sets $\Sigma_{w,n}$ is a cylinder.  It is easy 
to observe that these cylinders form a base of the topology on 
$\cA^{\bZ}$.  Further, for any $w\in\cA^*\setminus\{\varnothing\}$ let 
$U_w$ be the set of all strings in $\cA^{\bZ}$ that admit $w$ as a 
subword.  Obviously, $U_w$ is the union of the cylinders $\Sigma_{w,n}$ 
over all $n\in\bZ$.  Hence $U_w$ is an open set.  Besides, 
$\si_{\cA}(\Sigma_{w,n})=\Sigma_{w,n-1}$ for all $n\in\bZ$, which implies 
that $\si_{\cA}(U_w)=U_w$.

Consider an arbitrary subshift $\si_{\cA}|Y$.  It follows from the above 
that for any nonempty word $w\in\cA^*$, the set difference $Y\setminus U_w$ 
is a closed set and $\si_{\cA}(Y\setminus U_w)=Y\setminus U_w$.  Assuming 
the subshift $\si_{\cA}|Y$ is minimal, we obtain that $Y\setminus U_w$ is 
either the empty set or $Y$.  Hence $w$ is either a subword of every string 
in $Y$ or a subword of no string in $Y$.  Since this holds for any nonempty 
word $w$, we conclude that any two strings in $Y$ have the same subwords.

Now assume the subshift $\si_{\cA}|Y$ is not minimal, that is, there is a 
nonempty closed set $Y_0\subset Y$ different from $Y$ such that 
$\si_{\cA}(Y_0)=Y_0$.  Take any point $\om\in Y\setminus Y_0$.  Since the 
cylinders of the form $\Sigma_{w,n}$ constitute a base of the topology, we 
can find one of them that contains $\om$ and is disjoint from the closed 
set $Y_0$.  Since $\si_{\cA}(Y_0)=Y_0$, it follows that the set $U_w$ is 
also disjoint from $Y_0$.  We obtain that $w$ is a subword of the string 
$\om$ but not a subword of any string in $Y_0$.  Thus two strings in $Y$ do 
not always have the same subwords.
\end{proof}

For subshifts of the form $\si_{\cA}|\Om(\xi)$, where $\xi\in\cA^{\bZ}$, 
minimality is equivalent to a very restrictive combinatorial property of 
the string $\xi$ (usually called \emph{uniform recurrence}).

\begin{lemma}\label{subwords-in-subwords}
The subshift $\si_{\cA}|\Om(\xi)$, where $\xi\in\cA^{\bZ}$, is minimal if 
and only if for any nonempty word $w\in\cA^*$ that occurs infinitely many 
times as a subword of $\xi$, there exists $N\ge1$ such that $w$ is a 
subword of every word of length at least $N$ that is a subword of the 
string $\xi$.
\end{lemma}

\begin{proof}
Take any left-infinite string $\zeta\in\cA^{-\bN}$.  Let $\ldots 
x_{-1}x_0.x_1x_2\ldots$ be the bi-infinite string $\zeta.\xi$.  First we 
show that any nonempty word $w$ that occurs infinitely often as a subword 
of $\xi$ is also a subword of some string $\om\in\Om(\xi)$.  Let $s$ be the 
length of $w$.  There is a sequence of integers $\{n_k\}$, $0\le 
n_1<n_2<n_3<\ldots$, such that $x_{n_k+1}x_{n_k+2}\ldots x_{n_k+s}=w$ for 
all $k\ge1$.  Since $\cA^{\bZ}$ is compact, the sequence 
$\{\si^{n_k}(\zeta.\xi)\}_{k\ge1}$ has a limit point $\om$.  By 
construction, $\om=\zeta_1.w\xi_1$ for some $\xi_1\in\cA^{\bN}$ and 
$\zeta_1\in\cA^{-\bN}$.  In particular, $w$ is a subword of the string 
$\om$.  Clearly, $\om$ is also a limit point of the sequence 
$\zeta.\xi,\si_{\cA}(\zeta.\xi),\si_{\cA}^2(\zeta.\xi),\dots$.  Hence 
$\om\in\Om(\xi)$ due to Lemma \ref{limit-points}.

Next we show that for any word $w$ that is not a subword of all 
sufficiently long subwords of $\xi$, there is a string $\om'\in\Om(\xi)$ 
such that $w$ is not a subword of $\om'$.  For any $N\ge1$ let $w_N$ be a 
subword of $\xi$ of length at least $N$ that does not admit $w$ as a 
subword.  We have $w_N=x_{n_N}x_{n_N+1}\ldots x_{m_N}$ for some $n_N$ and 
$m_N$, where $n_N\ge1$ and $m_N-n_N\ge N-1$.  The sequence 
$\{\si^{n_{2k}+k}(\zeta.\xi)\}_{k\ge1}$ has a limit point $\om'=\ldots 
y_{-1}y_0.y_1y_2\ldots$.  Since $n_{2k}+k\to\infty$ as $k\to\infty$, it 
follows that the string $\om'$ is also a limit point of the sequence 
$\zeta.\xi,\si_{\cA}(\zeta.\xi),\si_{\cA}^2(\zeta.\xi),\dots$.  Hence 
$\om'\in\Om(\xi)$ due to Lemma \ref{limit-points}.  Any subword $u$ of 
$\om'$ is also a subword of the word $u_s=y_{-s}y_{-s+1}\ldots y_{s-1}y_s$ 
for some $s\ge1$.  By construction, $x_{n_{2k}+k-s}\,x_{n_{2k}+k-s+1}\ldots 
x_{n_{2k}+k+s}=u_s$ for infinitely many values of $k$.  If $k\ge s+1$ then 
$n_{2k}+k-s>n_{2k}$ and $n_{2k}+k+s\le n_{2k}+2k-1\le m_{2k}$.  It follows 
that $u_s$ is a subword of the word $w_{2k}$ for some $k$.  Then $u$ is 
also a subword of $w_{2k}$ so that $u\ne w$.

Now we can prove the lemma.  Assume that the string $\xi$ does not have the 
combinatorial property formulated in the lemma.  That is, some nonempty 
word $w$ occurs infinitely many times as a subword of $\xi$, but $w$ is not 
a subword of all sufficiently long subwords of $\xi$.  By the above there 
exist strings $\om,\om'\in\Om(\xi)$ such that $w$ is a subword of $\om$ but 
not a subword of $\om'$.  Then it follows from Lemma \ref{same-subwords} 
that the subshift $\si_{\cA}|\Om(\xi)$ is not minimal.  Conversely, assume 
that the string $\xi$ has the combinatorial property formulated in the 
lemma.  Take an arbitrary $\om\in\Om(\xi)$ and let $w$ be a subword of 
$\om$.  By Lemma \ref{omega-xi}, the word $w$ occurs infinitely often as a 
subword of $\xi$.  Hence for some $N\ge1$ every subword of $\xi$ of length 
at least $N$ admits $w$ as a subword.  Clearly, any bi-infinite string 
$\om'\in\Om(\xi)$ has a subword $w'$ of length $N$.  Since $w'$ is also a 
subword of $\xi$, it admits $w$ as a subword.  Then $w$ is a subword of 
$\om'$ as well.  We obtain that every subword of $\om$ is, in fact, a 
subword of all strings in $\Om(\xi)$.  As the choice of $\om$ was 
arbitrary, we conclude that any two strings in $\Om(\xi)$ have the same 
subwords.  By Lemma \ref{same-subwords}, the subshift $\si_{\cA}|\Om(\xi)$ 
is minimal.
\end{proof}

\begin{proof}[Proof of Proposition \ref{3-subshifts-minimal}]
Let $\xi=x_1x_2x_3\ldots$, where $\xi=\xiTM$, $\xipd$ or $\xiGE$.  In the 
cases $\xi=\xipd$ and $\xi=\xiGE$, it follows from respectively Lemmas 
\ref{xi-pd} and \ref{xi-GE} that $x_{n_1}=x_{n_2}$ whenever $1\le n_1<2^k$ 
and $n_2-n_1=2^km$ for some integers $k,m\ge1$.  As a consequence, the word 
$x_1x_2\ldots x_{2^k-1}$ coincides with $x_{2^km+1}x_{2^km+2}\ldots 
x_{2^km+2^k-1}$ for any $k,m\ge1$.

In the case $\xi=\xiTM$, it follows from Lemma \ref{xi-TM} that 
$x_{n_1}=x_{n_2}$ whenever $1\le n_1\le2^k$ and $n_2-n_1=2^km$ for some 
integers $k,m\ge1$ such that the binary expansion of the number $m$ has an 
even number of 1s.  Note that for any $s\ge1$ exactly one of the numbers 
$2s$ and $2s+1$ has an even number of 1s in its binary expansion.  As a 
consequence, the word $x_1x_2\ldots x_{2^k}$ coincides with the word 
$x_{M+1}x_{M+2}\ldots x_{M+2^k}$ for $M=2^k(2s)$ or $M=2^k(2s+1)$.

For any $k\ge1$ let $Z_k$ be the set of all integers $m\ge1$ such that the 
prefix of length $2^k-1$ of the string $\xi$ coincides with the subword 
$x_{2^km+1}x_{2^km+2}\ldots x_{2^km+2^k-1}$.  By the above, $Z_k=\bN$ if 
$\xi=\xipd$ or $\xiGE$.  If $\xi=\xiTM$ then for any $s\ge1$ the set $Z_k$ 
contains $2s$ or $2s+1$.  It follows that $Z_k$ contains at least one of
any three consecutive positive integers.

Now we are going to show that the string $\xi$ has the combinatorial 
property formulated in Lemma \ref{subwords-in-subwords}, which will 
complete the proof of the proposition.  Any subword $w$ of $\xi$ is also a 
subword of the prefix $x_1x_2\ldots x_{2^k-1}$ for some $k$.  Let 
$N=2^{k+2}$ and consider an arbitrary subword $w'$ of $\xi$ of length at 
least $N$.  We have $w'=x_{n_1}x_{n_1+1}\ldots x_{n_2}$ for some $n_1$ and 
$n_2$, where $n_1\ge1$ and $n_2-n_1\ge N-1$.  Let $2^km$ be the largest 
multiple of $2^k$ not exceeding $n_2$.  Then $n_2\le 2^k(m+1)-1$, which 
implies that $n_1\le n_2-N+1\le2^k(m-3)$.  In particular, $m-3\ge1$.  Hence 
at least one of the integers $m-3$, $m-2$ and $m-1$ belongs to the set 
$Z_k$.  Let $m_0$ be such an integer.  Then $x_{2^km_0+1}x_{2^km_0+2}\ldots 
x_{2^km_0+2^k-1}=x_1x_2\ldots x_{2^k-1}$.  Since $2^km_0+1>2^k(m-3)\ge n_1$ 
and $2^km_0+2^k-1\le 2^km-1<n_2$, it follows that $x_1x_2\ldots x_{2^k-1}$ 
is a subword of $w'$.  Then $w$ is a subword of $w'$ as well.
\end{proof}

In the remainder of this section we are going to establish relations 
between the subshifts $\si_{\{0,1\}}|\Om(\xiTM)$,
$\si_{\{x,y\}}|\Om(\xipd)$ and $\si_{\{a,c,d\}}|\Om(\xiGE)$.

\begin{definition}\label{def-factor}
Let $T_1:X_1\to X_1$ and $T_2:X_2\to X_2$ be homeomorphisms of topological 
spaces.  Suppose there exists a continuous map $f:X_1\to X_2$ such that $f$ 
is onto and $fT_1=T_2f$ so that the following diagram is commutative:
\[\begin{array}{ccc}
X_1 & \stackrel{T_1}{\longrightarrow} & X_1\\[0.2em]
f\Big\downarrow\phantom{f} && \phantom{f}\Big\downarrow f\\[0.25em]
X_2 & \stackrel{T_2}{\longrightarrow} & X_2
\end{array}
\]
Then we say that $T_2$ is a \textbf{continuous factor} of $T_1$ while $T_1$ 
is a \textbf{continuous extension} of $T_2$.  In the case the map $f$ is 
$k$-to-one for some $k\ge1$, we say that $T_2$ is a $k$-to-one continuous 
factor of $T_1$.  In the case $f$ is a homeomorphism, we say that the map 
$T_2$ is \textbf{topologically conjugate} to $T_1$. 
\end{definition}

\begin{proposition}\label{3-subshifts-factors}
The subshift $\si_{\{x,y\}}|\Om(\xipd)$ is a two-to-one continuous factor 
of the subshift $\si_{\{0,1\}}|\Om(\xiTM)$.  The subshift 
$\si_{\{a,c,d\}}|\Om(\xiGE)$ is topologically conjugate to 
$\si_{\{x,y\}}|\Om(\xipd)$.
\end{proposition}

To prove Proposition \ref{3-subshifts-factors}, we are going to use the 
so-called \emph{block factor maps}, which are helpful for establishing 
relations between subshifts.  Let $\cA_1$ and $\cA_2$ be two alphabets.  
Suppose that for some $k\ge1$ we have a function $\phi:\cA_1^k\to\cA_2$.  
This function induces a map $f_\phi:\cA_1^{\bZ}\to\cA_2^{\bZ}$ that sends 
any bi-infinite string $\ldots x_{-1}x_0.x_1x_2\ldots$ over the alphabet 
$\cA_1$ to a bi-infinite string $\ldots y_{-1}y_0.y_1y_2\ldots$ over 
$\cA_2$ given by $y_n=\phi(x_n,x_{n+1},\dots,x_{n+k-1})$ for all 
$n\in\bZ$.  It follows from the construction that for any cylinder 
$\Sigma\subset\cA_2^{\bZ}$ the pre-image $f_\phi^{-1}(\Sigma)$ is a finite 
union of cylinders in $\cA_1^{\bZ}$.  Hence the map $f_\phi$ is 
continuous.  Also by construction, $f_\phi$ intertwines the shifts 
$\si_{\cA_1}$ and $\si_{\cA_2}$, that is, 
$f_\phi\si_{\cA_1}=\si_{\cA_2}f_\phi$.  The map $f_\phi$ is called a 
$k$-block factor map.  Note that each shift is itself a $2$-block factor 
map.

\begin{lemma}\label{block-factor-0}
Given an infinite string $\xi_1=x_1x_2x_3\ldots\in\cA_1^{\bN}$, let 
$\xi_2=y_1y_2y_3\ldots$ be an infinite string in $\cA_2^{\bN}$ such that 
$y_n=\phi(x_n,x_{n+1},\dots,x_{n+k-1})$ for all $n\ge1$.  Then 
$f_\phi(\Om(\xi_1))=\Om(\xi_2)$.
\end{lemma}

\begin{proof}
Take any left-infinite string $\zeta_1\in\cA_1^{-\bN}$.  Then 
$f_\phi(\zeta_1.\xi_1)=\zeta_2.\xi_2$ for some left-infinite string 
$\zeta_2\in\cA_2^{-\bN}$.  By Lemma \ref{limit-points}, $\Om(\xi_1)$ is the 
set of all limit points of the sequence $\zeta_1.\xi_1, 
\si_{\cA_1}(\zeta_1.\xi_1),\si_{\cA_1}^2(\zeta_1.\xi_1),\dots$ while 
$\Om(\xi_2)$ is the set of all limit points of the sequence $\zeta_2.\xi_2, 
\si_{\cA_2}(\zeta_2.\xi_2),\si_{\cA_2}^2(\zeta_2.\xi_2),\dots$.  Since 
$f_\phi\si_{\cA_1}=\si_{\cA_2}f_\phi$, it follows by induction on $m$ that 
$\si_{\cA_2}^m(\zeta_2.\xi_2)= 
f_\phi\bigl(\si_{\cA_1}^m(\zeta_1.\xi_1)\bigr)$ for all $m\ge0$.  Hence the 
first sequence is mapped by $f_\phi$ onto the second sequence.  Assume 
$\om\in\Om(\xi_1)$.  Then $\om$ is the limit of some subsequence 
$\{\si_{\cA_1}^{m_i}(\zeta_1.\xi_1)\}_{i\ge1}$ of the first sequence.  
Continuity of the map $f_\phi$ implies that 
$\si_{\cA_2}^{m_i}(\zeta_2.\xi_2)\to f_\phi(\om)$ as $i\to\infty$ so that 
$f_\phi(\om)\in\Om(\xi_2)$.  Therefore 
$f_\phi(\Om(\xi_1))\subset\Om(\xi_2)$.

Conversely, assume $\om'\in\Om(\xi_2)$.  Let 
$\{\si_{\cA_2}^{m_j}(\zeta_2.\xi_2)\}_{j\ge1}$ be a subsequence of the 
second sequence converging to $\om'$.  The respective subsequence 
$\{\si_{\cA_1}^{m_j}(\zeta_1.\xi_1)\}_{j\ge1}$ of the first sequence need 
not converge, but it has a limit point due to compactness of
$\cA_1^{\bZ}$.  Any limit point $\om$ is also a limit point of the 
first sequence so that $\om\in\Om(\xi_1)$.  Besides, $f_\phi(\om)=\om'$ due 
to continuity of the map $f_\phi$.  We conclude that $\Om(\xi_2)\subset 
f_\phi(\Om(\xi_1))$.  Thus $f_\phi(\Om(\xi_1))=\Om(\xi_2)$.
\end{proof}

First consider a function $\phi_1:\{0,1\}^2\to\{x,y\}$ given by 
$\phi_1(0,1)=\phi_1(1,0)=x$ and $\phi_1(0,0)=\phi_1(1,1)=y$.

\begin{lemma}\label{block-factor-1}
The $2$-block factor map $f_{\phi_1}$ is two-to-one when restricted to the 
set $\Om(\xiTM)$, and $f_{\phi_1}(\Om(\xiTM))=\Om(\xipd)$.
\end{lemma}

\begin{proof}
For any $n\ge1$ let $x_n$ denote the $n$-th symbol of $\xiTM$ and $y_n$ 
denote the $n$-th symbol of $\xipd$.  First we show that 
$y_n=\phi_1(x_n,x_{n+1})$ for all $n\ge1$.  If $n$ is odd, the last digit 
in the binary expansion of $n$ is $1$.  Changing that digit to $0$, we 
obtain the binary expansion of $n-1$.  Hence the numbers of 1s in those 
expansions differ by $1$.  It follows from Lemma \ref{xi-TM} that 
$x_{n+1}\ne x_n$.  Then $\phi_1(x_n,x_{n+1})=x$.  Lemma \ref{xi-pd} implies 
that $y_n=x$ as well.

If $n$ is even, it is of the form $2^km$, where $k,m\ge1$ and $m$ is odd.  
Then the last $k+1$ digits in the binary expansion of the number $n$ are 1 
followed by $k$ 0s.  The last $k+1$ digits in the binary expansion of $n-1$ 
are 0 followed by $k$ 1s.  The other digits are the same in both 
expansions.  Hence the numbers of 1s in those expansions differ by $k-1$.  
Lemma \ref{xi-TM} implies that $x_{n+1}\ne x_n$ if $k$ is even, and 
$x_{n+1}=x_n$ if $k$ is odd.  By Lemma \ref{xi-pd}, $y_n=x$ if $k$ is even, 
and $y_n=y$ if $k$ is odd.  Either way, $\phi_1(x_n,x_{n+1})=y_n$.

Since $y_n=\phi_1(x_n,x_{n+1})$ for all $n\ge1$, Lemma \ref{block-factor-0} 
implies that $f_{\phi_1}(\Om(\xiTM))=\Om(\xipd)$.  It remains to show that 
the map $f_{\phi_1}$ is two-to-one when restricted to the set 
$\Om(\xiTM)$.  Consider the substitution $\mathsf{s}_{(01)}$ over the 
alphabet $\{0,1\}$ given by $\mathsf{s}_{(01)}(0)=1$, 
$\mathsf{s}_{(01)}(1)=0$.  Let $f_{(01)}$ denote the induced $1$-block 
factor map.  Lemma \ref{block-factor-0} implies that 
$f_{(01)}(\Om(\xiTM))=\Om\bigl(\mathsf{s}_{(01)}(\xiTM)\bigr)$.  It follows 
from Lemma \ref{TM-unique} that $\mathsf{s}_{(01)}(\xiTM)=\xiTM'$ and 
$\Om(\xiTM')=\Om(\xiTM)$.  Hence $f_{(01)}(\Om(\xiTM))=\Om(\xiTM)$.  Note 
that $\phi_1\bigl(\mathsf{s}_{(01)}(z_1),\mathsf{s}_{(01)}(z_2)\bigr)=
\phi_1(z_1,z_2)$ for all $z_1,z_2\in\{0,1\}$.  It follows that 
$f_{\phi_1}\bigl(f_{(01)}(\om)\bigr)=f_{\phi_1}(\om)$ for any string 
$\om=\ldots z_{-1}z_0.z_1z_2\ldots$ in $\{0,1\}^{\bZ}$.  Clearly, 
$f_{(01)}(\om)\ne\om$.  Now assume that $f_{\phi_1}(\om')=f_{\phi_1}(\om)$ 
for some $\om'=\ldots z'_{-1}z'_0.z'_1z'_2\ldots$.  Then 
$\phi_1(z_n,z_{n+1})=\phi_1(z'_n,z'_{n+1})$ for all $n\in\bZ$.  This 
implies that $z'_n=z_n$ if and only if $z'_{n+1}=z_{n+1}$.  It follows by 
induction that either $z'_n=z_n$ for all $n$ or $z'_n\ne z_n$ for all $n$.  
Hence $\om'=\om$ or $\om'=f_{(01)}(\om)$.  Since 
$f_{(01)}(\Om(\xiTM))=\Om(\xiTM)$, we conclude that the map $f_{\phi_1}$ is 
two-to-one when restricted to $\Om(\xiTM)$.
\end{proof}

Next consider a function $\phi_2:\{a,c,d\}\to\{x,y\}$ given by 
$\phi_2(a)=\phi_2(d)=x$ and $\phi_2(c)=y$.

\begin{lemma}\label{block-factor-2}
The $1$-block factor map $f_{\phi_2}$ is one-to-one when restricted to the 
set $\Om(\xiGE)$, and $f_{\phi_2}(\Om(\xiGE))=\Om(\xipd)$.
\end{lemma}

\begin{proof}
For any $n\ge1$ let $x_n$ denote the $n$-th letter of $\xiGE$ and $y_n$ 
denote the $n$-th letter of $\xipd$.  It follows from the descriptions of 
the strings $\xipd$ and $\xiGE$ (Lemmas \ref{xi-pd} and \ref{xi-GE}) that 
$\phi_2(x_n)=y_n$ for all $n\ge1$.  Therefore 
$f_{\phi_2}(\Om(\xiGE))=\Om(\xipd)$ due to Lemma \ref{block-factor-0}.

It remains to prove that the map $f_{\phi_2}$ is one-to-one when restricted 
to $\Om(\xiGE)$.  Consider any strings $\om=\ldots z_{-1}z_0.z_1z_2\ldots 
\in\Om(\xipd)$ and $\om'=\ldots z'_{-1}z'_0.z'_1z'_2\ldots \in\Om(\xiGE)$ 
such that $f_{\phi_2}(\om')=\om$, that is, $\phi_2(z'_n)=z_n$ for all 
$n\in\bZ$.  We need to show that the string $\om'$ is uniquely determined 
by $\om$.  By Lemma \ref{xi-pd}, $y_n=y$ if $n$ is even but not divisible 
by $4$.  Since any four consecutive integers include such a number, it 
follows that $xxxx$ is not a subword of $\xipd$.  Then $xxxx$ is also not a 
subword of $\om$.  Hence the string $\om$ contains the letter $y$.  Choose 
any $k\in\bZ$ such that $z_k=y$.  Further, Lemma \ref{xi-GE} implies that 
$x_n=a$ if and only if $n$ is odd.  Therefore the words $aa$, $cc$, $cd$, 
$dc$ and $dd$ are not subwords of $\xiGE$.  Then they are also not subwords 
of $\om'$.  As a consequence, either $z'_n=a$ if $n$ is odd and $z'_n\ne a$ 
if $n$ is even, or else $z'_n=a$ if $n$ is even and $z'_n\ne a$ if $n$ is 
odd.  Since $\phi_2(z'_k)=z_k=y$, we obtain that $z'_k=c$.  It follows that 
$z'_n=a$ if and only if $n-k$ is odd.  In the case $n-k$ is even, we have 
$z'_n=c$ if $z_n=y$ and $z'_n=d$ if $z_n=x$.  Thus the string $\om'$ is 
uniquely determined by the string $\om$.
\end{proof}

Now we are ready to complete the proof of Proposition 
\ref{3-subshifts-factors}.

\begin{proof}[Proof of Proposition \ref{3-subshifts-factors}]
Since the map $f_{\phi_1}$ defined above is a $2$-block factor map of 
$\{0,1\}^{\bZ}$ to $\{x,y\}^{\bZ}$, it is continuous and satisfies 
$f_{\phi_1}\si_{\{0,1\}}=\si_{\{x,y\}}f_{\phi_1}$.  By Lemma 
\ref{block-factor-1}, $f_{\phi_1}(\Om(\xiTM))=\Om(\xipd)$.  Hence the 
subshift $\si_{\{x,y\}}|\Om(\xipd)$ is a continuous factor of the subshift 
$\si_{\{0,1\}}|\Om(\xiTM)$.  Besides, $f_{\phi_1}$ is two-to-one when 
restricted to $\Om(\xiTM)$ so that $\si_{\{x,y\}}|\Om(\xipd)$ is a 
two-to-one continuous factor of $\si_{\{0,1\}}|\Om(\xiTM)$.

The map $f_{\phi_2}$ defined above is a $1$-block factor map of 
$\{a,c,d\}^{\bZ}$ to $\{x,y\}^{\bZ}$.  Hence it is continuous and satisfies 
$f_{\phi_2}\si_{\{a,c,d\}}=\si_{\{x,y\}}f_{\phi_1}$.  By Lemma 
\ref{block-factor-2}, $f_{\phi_2}$ is one-to-one when restricted to the 
set $\Om(\xiGE)$, and $f_{\phi_2}(\Om(\xiGE))=\Om(\xipd)$.  Hence the 
restriction of $f_{\phi_2}$ to $\Om(\xiGE)$ is a continuous invertible map 
onto $\Om(\xipd)$.  Since $\Om(\xiGE)$ is a closed subset of the compact 
topological space $\{a,c,d\}^{\bZ}$, it is compact as well.  It follows 
that the restriction of $f_{\phi_2}$ to $\Om(\xiGE)$ is a homeomorphism 
onto $\Om(\xipd)$.  Therefore the subshift $\si_{\{x,y\}}|\Om(\xipd)$ is 
topologically conjugate to the subshift $\si_{\{a,c,d\}}|\Om(\xiGE)$.  Then 
$\si_{\{a,c,d\}}|\Om(\xiGE)$ is topologically conjugate to 
$\si_{\{x,y\}}|\Om(\xipd)$ as well.
\end{proof}

\section{Topological full groups}\label{sect-TFG}

Let $X$ be a \emph{Cantor set}.  This means that $X$ is a compact 
metrizable topological space that is totally disconnected and has no 
isolated points.  The topology on the Cantor set is generated by 
\emph{clopen} (i.e., both closed and open) sets.  For instance, the space 
$\cA^{\bZ}$ of bi-infinite strings over an alphabet $\cA$ is a Cantor set.  
Each cylinder is a clopen subset of $\cA^{\bZ}$.

\begin{definition}\label{def-TFG}
Let $T:X\to X$ be a homeomorphism of the Cantor set $X$ onto itself.  The 
\textbf{topological full group} of $T$, denoted $[[T]]$, is a transformation
group consisting of all homeomorphisms $F:X\to X$ that can be given by
$F(x)=T^{\nu(x)}(x)$, $x\in X$ for some continuous function $\nu:X\to\bZ$.
\end{definition}

If a function $\nu:X\to\bZ$ is continuous then it is locally constant and
takes only finitely many values.  Then nonempty level sets of $\nu$ form a 
finite partition of the Cantor set $X$ into clopen sets.  Thus every 
element of the topological full group $[[T]]$ is ``piecewise'' a power of 
$T$.  The group $[[T]]$ is countable (as there are only countably many 
clopen subsets of $X$).

The topological full group $[[T]]$ is \emph{ample}, which means that any 
homeomorphism of the Cantor set $X$ that locally coincides with elements of 
the group is itself in $[[T]]$.  Moreover, it is the smallest ample 
collection of homeomorphisms that contains the cyclic group generated by 
$T$.  For more details and proofs of these facts, see Section 2 of the 
paper \cite{GrigVor24}.

In the case $T$ is a minimal homeomorphism of a Cantor set (such a 
transformation is called a \emph{Cantor minimal system}), the (isomorphism 
class of the) topological full group $[[T]]$ is an almost complete 
invariant of the dynamics of $T$ as shown by Giordano, Putnam and Skau 
\cite{GPS}.

\begin{theorem}[\cite{GPS}]\label{flip-conjugacy}
Suppose $T_1:X_1\to X_1$ and $T_2:X_2\to X_2$ are minimal homeomorphisms of 
Cantor sets.  Then the topological full groups $[[T_1]]$ and $[[T_2]]$ are 
isomorphic if and only if $T_1$ is topologically conjugate to $T_2$ or 
$T_2^{-1}$.
\end{theorem}

Next we show that Theorem \ref{flip-conjugacy} applies to the substitution 
subshifts $\si_{\{0,1\}}|\Om(\xiTM)$, $\si_{\{x,y\}}|\Om(\xipd)$ and 
$\si_{\{a,c,d\}}|\Om(\xiGE)$, and, moreover, the topological full groups of 
these subshifts are complete invariants of their dynamics.

\begin{proposition}\label{3-subshifts-Cantor}
The subshifts $\si_{\{0,1\}}|\Om(\xiTM)$, $\si_{\{x,y\}}|\Om(\xipd)$ and 
$\si_{\{a,c,d\}}|\Om(\xiGE)$ are Cantor minimal systems.  Each of these 
subshifts is topologically conjugate to its inverse.
\end{proposition}

Minimality of all three subshifts is already established by Proposition 
\ref{3-subshifts-minimal}, but we need to show that $\Om(\xiTM)$, 
$\Om(\xipd)$ and $\Om(\xiGE)$ are Cantor sets.

\begin{lemma}\label{finite-or-Cantor}
Let $T$ be a homeomorphism of a Cantor set $X$ and let $Y\subset X$ be a 
nonempty closed set such that $T(Y)=Y$.  Suppose that the restriction $T|Y$ 
is a minimal homeomorphism of $Y$.  Then the subset $Y$ is either finite or 
a Cantor set.
\end{lemma}

\begin{proof}
Since $Y$ is a closed subset of the Cantor set $X$, it is compact, 
metrizable and totally disconnected.  Therefore $Y$ is itself a Cantor set 
if and only if it has no isolated points.  Since $T$ is a homeomorphism of 
$X$ and $T(Y)=Y$, isolated points of $Y$ are mapped to isolated points of 
$Y$ while non-isolated points are mapped to non-isolated points.  It 
follows that $T(Y_0)=Y_0$, where $Y_0$ is the set of all non-isolated 
points of $Y$.  As $Y_0$ is a closed subset of $Y$, minimality of the 
restriction $T|Y$ implies that $Y_0$ is the empty set or $Y$.  In the case 
$Y_0$ is empty, every point of $Y$ is isolated.  Then compactness of $Y$ 
implies that this set is finite.  In the case $Y_0=Y$, there are no 
isolated points in $Y$.  Then $Y$ is a Cantor set.
\end{proof}

\begin{lemma}\label{when-finite}
The set $\Om(\xi)$ is finite if and only if the infinite string $\xi$ is 
eventually periodic.
\end{lemma}

\begin{proof}
Let $\xi=x_1x_2x_3\ldots$ be an infinite string over an alphabet $\cA$.  
First assume that $\xi$ is eventually periodic.  That is, $\xi=wuuu\ldots$ 
for some words $w,u\in\cA^*$, where $u\ne\varnothing$.  Let $n_0$ be the 
length of $w$ and $p$ be the length of $u$.  Then for any left-infinite 
string $\zeta\in\cA^{-\bN}$ the sequence 
$\si^{n_0}(\zeta.\xi),\si^{n_0+p}(\zeta.\xi),\si^{n_0+2p}(\zeta.\xi),\dots$ 
converges in $\cA^{\bZ}$ to the bi-infinite string $\om=\ldots 
uu.uu\ldots$.  Consequently, for every $i$, $0\le i\le p-1$, we have 
$\si^{n_0+i+kp}(\zeta.\xi)\to\si^i(\om)$ as $k\to\infty$.  Any subsequence 
of the sequence $\zeta.\xi,\si(\zeta.\xi),\si^2(\zeta.\xi),\dots$ shares 
infinitely many terms with at least one of the subsequences
$\{\si^{n_0+i+kp}(\zeta.\xi)\}_{k\ge1}$, $0\le i\le p-1$.  It follows that 
$\om,\si(\om),\dots,\si^{p-1}(\om)$ are the only limit points of that 
sequence.  By Lemma \ref{limit-points}, these are also the only points of 
the set $\Om(\xi)$ so that $\Om(\xi)$ is finite.

Conversely, assume that the set $\Om(\xi)$ is finite.  Take any 
$\om\in\Om(\xi)$.  Since $\si(\Om(\xi))=\Om(\xi)$, it follows that 
$\si^p(\om)=\om$ for some $p\ge1$.  Then $\om=\ldots uu.uu\ldots$, where 
$u$ is a word of length $p$.  We claim that for some $k\ge1$ the string 
$\xi$ has no subwords of the form $u^kv=uu\ldots uv$, where $v$ is a word 
of length $p$ different from $u$.  Assume the contrary.  Then there exists 
a sequence of words $v_1,v_2,v_3,\dots$ of length $p$ different from $u$ 
such that for any $k\ge1$ the word $u^kv_k$ is a subword of $\xi$.  This 
means existence of nonnegative integers $n_1,n_2,n_3,\dots$ such that 
$x_{n_k+1}x_{n_k+2}\ldots x_{n_k+(k+1)p}=u^kv_k$ for all $k\ge1$.  Take 
any left-infinite string $\zeta\in\cA^{-\bN}$.  Since $\cA^{\bZ}$ is 
compact, the sequence $\{\si^{n_k+kp}(\zeta.\xi)\}_{k\ge1}$ has a limit 
point $\om'$.  By construction, $\om'=\ldots uuu.v\eta$ for some word $v$ 
of length $p$ different from $u$ and some $\eta\in\cA^{\bN}$.  Since 
$n_k+kp\to\infty$ as $k\to\infty$, the string $\om'$ is also a limit point 
of the sequence $\zeta.\xi,\si(\zeta.\xi),\si^2(\zeta.\xi),\dots$.  Hence 
$\om'\in\Om(\xi)$ due to Lemma \ref{limit-points}.  Then each of the  
strings $\si^p(\om'),\si^{2p}(\om'),\si^{3p}(\om'),\dots$ belongs to 
$\Om(\xi)$ as well.  It is easy to observe that all these strings are 
distinct, which contradicts the fact that $\Om(\xi)$ is a finite set.

We have established that for some $k\ge1$ words of the form $u^kv$, where 
$v$ is a word of length $p$ different from $u$, do not occur as subwords of 
the string $\xi$.  The word $u^k$ does occur as a subword of $\xi$ since it 
is a subword of $\om$.  Hence $x_{n+1}x_{n+2}\ldots x_{n+kp}=u^k$ for some 
$n\ge0$.  Note that whenever $x_{n+ip+1}x_{n+ip+2}\ldots x_{n+ip+p}=u$ for 
$k$ consecutive values of $i$, the same should hold for the next value of 
$i$.  Now it follows by induction that $x_{n+ip+1}x_{n+ip+2}\ldots 
x_{n+ip+p}=u$ for all $i\ge0$.  Thus the string $\xi$ is eventually 
periodic.
\end{proof}

For any word $w$, let $\overleftarrow{w}$ denote the same word written 
backwards.  The word $w$ is called a \emph{palindrome} if 
$\overleftarrow{w}=w$.

\begin{lemma}\label{palindrome}
Suppose that any subword of a string $\xi\in\cA^{\bN}$ can be extended to a 
palindrome subword of $\xi$.  Then the subshift $\si_{\cA}|\Om(\xi)$ is 
topologically conjugate to its own inverse.
\end{lemma}

\begin{proof}
Consider a map $R:\cA^{\bZ}\to\cA^{\bZ}$ that flips any bi-infinite string 
$\ldots y_{-2}y_{-1}y_0.y_1y_2\ldots$ in $\cA^{\bZ}$ around turning it into 
the string $\ldots y_2y_1y_0.y_{-1}y_{-2}\ldots$.  The map $R$ is an 
involution that maps cylinders onto cylinders.  This implies that $R$ is a 
homeomorphism.  It is easy to see that $R\si_{\cA}=\si_{\cA}^{-1}R$.  It 
follows that for any subshift $\si_{\cA}|Y$, the restriction 
$\si_{\cA}|R(Y)$ is also a subshift.  Moreover, $\si_{\cA}|R(Y)$ is 
topologically conjugate to the inverse of $\si_{\cA}|Y$.  Hence to prove 
the lemma, it is enough to show that $R(\Om(\xi))=\Om(\xi)$.  A string 
$\om\in\cA^{\bZ}$ belongs to $\Om(\xi)$ if and only if every subword of 
$\om$ is also a subword of $\xi$.  By construction, a word $u$ is a subword 
of $R(\om)$ if and only if $\overleftarrow{u}$ is a subword of $\om$.  It 
follows that a string $\eta\in\cA^{\bZ}$ belongs to $R(\Om(\xi))$ if and 
only if for every subword $u$ of $\eta$, the word $\overleftarrow{u}$ is a 
subword of $\xi$.

If a word $u$ is a subword of a word $w$, then the word $\overleftarrow{u}$ 
is a subword of $\overleftarrow{w}$.  In the case $w$ is a palindrome, $u$ 
is a subword of $w$ if and only if $\overleftarrow{u}$ is a subword of 
$w$.  Since any subword of the string $\xi$ can be extended to a palindrome 
subword of $\xi$, any word $u\in\cA^*$ is a subword of $\xi$ if and only if 
$\overleftarrow{u}$ is a subword of $\xi$.  Now it follows from the above 
that $R(\Om(\xi))=\Om(\xi)$.
\end{proof}

\begin{proof}[Proof of Proposition \ref{3-subshifts-Cantor}]
By Proposition \ref{3-subshifts-minimal}, the subshifts 
$\si_{\{0,1\}}|\Om(\xiTM)$, $\si_{\{x,y\}}|\Om(\xipd)$ and 
$\si_{\{a,c,d\}}|\Om(\xiGE)$ are minimal.  Let us prove that $\Om(\xiTM)$, 
$\Om(\xipd)$ and $\Om(\xiGE)$ are Cantor sets.  In view of Lemma 
\ref{finite-or-Cantor}, it is enough to show that all three sets are 
infinite.  By Proposition \ref{3-subshifts-factors}, the subshift 
$\si_{\{x,y\}}|\Om(\xipd)$ is a continuous factor of
$\si_{\{0,1\}}|\Om(\xiTM)$ while the subshift $\si_{\{a,c,d\}}|\Om(\xiGE)$ 
is topologically conjugate to $\si_{\{x,y\}}|\Om(\xipd)$.  This means, in 
particular, that the set $\Om(\xiTM)$ can be mapped onto $\Om(\xipd)$ while 
$\Om(\xipd)$ can be mapped onto $\Om(\xiGE)$.  It follows that the sets 
$\Om(\xiTM)$ and $\Om(\xipd)$ are infinite if $\Om(\xiGE)$ is infinite.

It remains to prove that the set $\Om(\xiGE)$ is infinite.  In view of 
Lemma \ref{when-finite}, it is enough to show that the string $\xiGE$ is 
not eventually periodic.  If an infinite string $\xi=x_1x_2x_3\ldots$ is 
eventually periodic with period $p$, then for any $n\ge1$ the sequence of 
letters $x_n,x_{n+p},x_{n+2p},\dots$ is eventually constant.  We are going 
to show that the latter is not the case when $\xi=\xiGE$.  Any integer 
$p\ge1$ can be represented as $2^km$, where $k\ge0$, $m\ge1$ and $m$ is 
odd.  Let $n_p=2^k$.  Then for any $s\ge1$ the number 
$n_p+2sp=2^k+2^{k+1}sm$ is divisible by $2^k$ but not divisible by 
$2^{k+1}$.  Since $m$ is odd, the number $m^2$ leaves remainder $1$ after 
division by $4$, and so does the number $(4s+1)m^2$ for any $s\ge1$.  
Therefore the number $n_p+(4s+1)mp=2^k+2^k(4s+1)m^2$ is divisible by 
$2^{k+1}$ but not divisible by $2^{k+2}$.  Now it follows from Lemma 
\ref{xi-GE} that $x_{n_p+2sp}=x_{2^k}$ for all $s\ge1$, 
$x_{n_p+(4s+1)mp}=x_{2^{k+1}}$ for all $s\ge1$, and $x_{2^{k+1}}\ne 
x_{2^k}$.  As a consequence, the sequence 
$x_{n_p},x_{n_p+p},x_{n_p+2p},\dots$ is not eventually constant.

We proceed to the second statement of the proposition.  Suppose 
$\mathsf{s}$ is a substitution over an alphabet $\cA$ that substitutes 
every symbol of $\cA$ with a palindrome.  Then for any word $w=x_1x_2\ldots 
x_n\in\cA^*$,
\[
\overleftarrow{\mathsf{s}(w)}=
\overleftarrow{\mathsf{s}(x_1)\mathsf{s}(x_2)\ldots\mathsf{s}(x_n)}
=\overleftarrow{\mathsf{s}(x_n)}\ldots\overleftarrow{\mathsf{s}(x_2)}
\,\overleftarrow{\mathsf{s}(x_1)}= \mathsf{s}(x_n)\ldots\mathsf{s}(x_2) 
\mathsf{s}(x_1)=\mathsf{s}\bigl(\overleftarrow{w}\bigr).
\]
As a consequence, the substitution $\mathsf{s}$ transforms palindromes into 
palindromes.  Note that the words $\sub(a)=aca$, $\sub(c)=d$ and 
$\sub(d)=c$ are palindromes.  The words $\TM^2(0)=0110$ and $\TM^2(1)=1001$ 
are also palindromes.  By the above the substitutions $\sub$ and $\TM^2$ 
transform palindromes into palindromes.  It follows by induction that for 
any $k\ge1$ the words $\sub^k(a)$ and $\TM^{2k}(0)$ are palindromes.  
Recall that each $\sub^k(a)$ is a subword of the string $\xiGE$.  Moreover, 
any subword of $\xiGE$ is also a subword of $\sub^k(a)$ for all 
sufficiently large $k$.  Likewise, each $\TM^{2k}(0)$ is a subword of the 
string $\xiTM$ and any subword of $\xiTM$ is also a subword of 
$\TM^{2k}(0)$ for all sufficiently large $k$.  Now it follows from Lemma 
\ref{palindrome} that each of the subshifts $\si_{\{a,c,d\}}|\Om(\xiGE)$ 
and $\si_{\{0,1\}}|\Om(\xiTM)$ is topologically conjugate to its inverse.

By Proposition \ref{3-subshifts-factors}, the subshift 
$\si_{\{a,c,d\}}|\Om(\xiGE)$ is topologically conjugate to the subshift
$\si_{\{x,y\}}|\Om(\xipd)$.  By the above $\si_{\{a,c,d\}}|\Om(\xiGE)$ is 
topologically conjugate to its inverse.  Hence there exist homeomorphisms 
$f_1:\Om(\xipd)\to\Om(\xiGE)$ and $f_2:\Om(\xiGE)\to\Om(\xiGE)$ such that 
$f_1\si_{\{x,y\}}=\si_{\{a,c,d\}}f_1$ on $\Om(\xipd)$ and 
$f_2\si_{\{a,c,d\}}=\si_{\{a,c,d\}}^{-1}f_2$ on $\Om(\xiGE)$.  Then the map 
$f=f_1^{-1}f_2f_1$ is a homeomorphism of $\Om(\xipd)$ onto itself.  It is 
easy to check that $f\si_{\{x,y\}}=\si_{\{x,y\}}^{-1}f$ on $\Om(\xipd)$.  
Thus the subshift $\si_{\{x,y\}}|\Om(\xipd)$ is topologically conjugate to 
its inverse.
\end{proof}

\begin{definition}\label{def-aperiodic}
A transformation $T:X\to X$ of a set $X$ is called \textbf{aperiodic} if it 
has no periodic points, that is, $T^n(x)\ne x$ for all points $x\in X$ and 
integers $n\ge1$.
\end{definition}

The main result of this section is the following embedding theorem.

\begin{theorem}\label{TFG-embeds}
Suppose that $T_1$ and $T_2$ are aperiodic homeomorphisms of Cantor sets 
and $T_2$ is a continuous factor of $T_1$.  Then the topological full group 
$[[T_2]]$ embeds into $[[T_1]]$.
\end{theorem}

To prove Theorem \ref{TFG-embeds}, we are going to use an isomorphic model 
of the group $[[T]]$, where $T$ is an aperiodic homeomorphism of a Cantor 
set (onto itself).  The model is due to Putnam \cite{Putnam89} (see the 
construction of a group $G$ right before Theorem 5.2 in \cite{Putnam89}) 
and it actually predates the group $[[T]]$.

\begin{lemma}\label{variable-power}
Suppose $T$ is a homeomorphism of a Cantor set $X$ and $\ell:X\to\bZ$ is a 
continuous function.  Then the map $h_{T,\ell}:X\to X$ defined by 
$h_{T,\ell}(x)=T^{\ell(x)}(x)$, $x\in X$, is continuous.
\end{lemma}

\begin{proof}
For any $n\in\bZ$ the map $h_{T,\ell}$ coincides with a continuous map 
$T^n$ on the set $\ell^{-1}(n)$.  Since the function $\ell$ is continuous, 
the set $\ell^{-1}(n)$ is clopen, which implies that $h_{T,\ell}$ is 
continuous on that set as well.  As $X$ is the disjoint union of the sets 
$\ell^{-1}(n)$, $n\in\bZ$, the lemma follows.
\end{proof}

Let $C(X,\bZ)$ denote the set of all continuous functions $\ell:X\to\bZ$.  
Given a homeomorphism $T$ of a Cantor set $X$, for any functions 
$\ell_1,\ell_2\in C(X,\bZ)$ we define a function 
$\ell_1\oplus_T\ell_2:X\to\bZ$ by $(\ell_1\oplus_T\ell_2)(x)= 
\ell_1\bigl(T^{\ell_2(x)}(x)\bigr)+\ell_2(x)$ for all $x\in X$.  Note that 
the formula involves addition and composition of functions.  We never use 
multiplication of functions in what follows.

\begin{lemma}
$\oplus_T$ is a well defined operation on $C(X,\bZ)$.
\end{lemma}

\begin{proof}
We need to show that for any continuous functions $\ell_1,\ell_2:X\to\bZ$, 
the function $\ell_1\oplus_T\ell_2$ is continuous as well.  Note that 
$\ell_1\oplus_T\ell_2=\ell_1h_{T,\ell_2}+\ell_2$.  The map $h_{T,\ell_2}$ 
is continuous due to Lemma \ref{variable-power}.  Then $\ell_1h_{T,\ell_2}$ 
is continuous as a composition of two continuous functions.  Finally, 
$\ell_1\oplus_T\ell_2$ is continuous as the sum of two continuous functions 
$\ell_1h_{T,\ell_2}$ and $\ell_2$.
\end{proof}

\begin{lemma}\label{h-composed}
$h_{T,\ell_1\oplus_T\ell_2}=h_{T,\ell_1}h_{T,\ell_2}$ for all 
$\ell_1,\ell_2\in C(X,\bZ)$.
\end{lemma}

\begin{proof}
For any $x\in X$ we have $(\ell_1\oplus_T\ell_2)(x)= 
\ell_1(h_{T,\ell_2}(x))+\ell_2(x)$.  Then
\[
h_{T,\ell_1\oplus_T\ell_2}(x)=T^{\ell_1(h_{T,\ell_2}(x))+\ell_2(x)}(x)
=T^{\ell_1(h_{T,\ell_2}(x))}T^{\ell_2(x)}(x).
\]
On the other hand,
\[
h_{T,\ell_1}(h_{T,\ell_2}(x))=T^{\ell_1(h_{T,\ell_2}(x))}h_{T,\ell_2}(x)
=T^{\ell_1(h_{T,\ell_2}(x))}T^{\ell_2(x)}(x).
\]
Hence $h_{T,\ell_1\oplus_T\ell_2}=h_{T,\ell_1}h_{T,\ell_2}$ everywhere on 
$X$.
\end{proof}

\begin{lemma}
$\oplus_T$ is an associative operation on $C(X,\bZ)$.
\end{lemma}

\begin{proof}
For any functions $\ell,\ell'\in C(X,\bZ)$, we have $\ell\oplus_T\ell'=\ell 
h_{T,\ell'}+\ell'$.  Besides, $h_{T,\ell\oplus_T\ell'}= 
h_{T,\ell}h_{T,\ell'}$ due to Lemma \ref{h-composed}.  Given arbitrary 
functions $\ell_1,\ell_2,\ell_3\in C(X,\bZ)$, we obtain
\begin{gather*}
(\ell_1\oplus_T\ell_2)\oplus_T\ell_3=
(\ell_1h_{T,\ell_2}+\ell_2)\oplus_T\ell_3=
(\ell_1h_{T,\ell_2}+\ell_2)h_{T,\ell_3}+\ell_3=
\ell_1h_{T,\ell_2}h_{T,\ell_3}+\ell_2h_{T,\ell_3}+\ell_3, \\
\ell_1\oplus_T(\ell_2\oplus_T\ell_3)=
\ell_1h_{T,\ell_2\oplus_T\ell_3}+(\ell_2\oplus_T\ell_3)=
\ell_1h_{T,\ell_2}h_{T,\ell_3}+\ell_2h_{T,\ell_3}+\ell_3.
\end{gather*}
Thus $(\ell_1\oplus_T\ell_2)\oplus_T\ell_3= 
\ell_1\oplus_T(\ell_2\oplus_T\ell_3)$.
\end{proof}

Let $\0_X$ denote the zero function on $X$.  It is easy to see that $\0_X$ 
is the identity element of the operation $\oplus_T$.  Hence $C(X,\bZ)$ is a 
monoid with respect to $\oplus_T$.  We denote by $C^\times_T(X,\bZ)$ the 
set of all invertible elements of this monoid.  Then $C^\times_T(X,\bZ)$ is 
a group with respect to the operation $\oplus_T$.

Next we define a map $H_T:C(X,\bZ)\to C(X,X)$ by $H_T(\ell)=h_{T,\ell}$ for 
all $\ell\in C(X,\bZ)$.  The map $H_T$ is well defined due to Lemma 
\ref{variable-power}.  In view of Lemma \ref{h-composed}, $H_T$ is a 
homomorphism of the monoid $C(X,\bZ)$ with the operation $\oplus_T$ to the 
monoid $C(X,X)$ with the operation of composition.

\begin{lemma}\label{image-of-group}
$H_T\bigl(C^\times_T(X,\bZ)\bigr)=[[T]]$.
\end{lemma}

\begin{proof}
It follows from Definition \ref{def-TFG} that 
$[[T]]=H_T(C(X,\bZ))\cap\Homeo(X)$, where $\Homeo(X)$ is the group of all 
homeomorphisms of $X$.  Hence it is enough to show that for any function 
$\ell\in C(X,\bZ)$, the map $H_T(\ell)=h_{T,\ell}$ is a homeomorphism if 
and only if $\ell$ is invertible with respect to the operation $\oplus_T$.  
First assume that $\ell$ is invertible, and let $\ell'$ be its inverse.  We 
have $\ell\oplus_T\ell'=\ell'\oplus_T\ell=\0_X$.  Then it follows from 
Lemma \ref{h-composed} that $h_{T,\ell}h_{T,\ell'}=h_{T,\ell'}h_{T,\ell} 
=h_{T,\0_X}=\id_X$.  Hence the map $h_{T,\ell}$ is invertible and its 
inverse $h_{T,\ell'}$ is continuous (due to Lemma \ref{variable-power}).  
Therefore $h_{T,\ell}$ is a homeomorphism.

Conversely, assume that $h_{T,\ell}$ is a homeomorphism of $X$.  Define a 
function $\ell':X\to\bZ$ by $\ell'(x)=-\ell\bigl(h_{T,\ell}^{-1}(x)\bigr)$ 
for all $x\in X$.  Since $h_{T,\ell}$ is a homeomorphism, the function 
$\ell'$ is well defined and continuous.  Hence $\ell'\in C(X,\bZ)$.  By 
construction,
\[
\ell'\oplus_T\ell=\ell'h_{T,\ell}+\ell=
(-\ell)h_{T,\ell}^{-1}h_{T,\ell}+\ell= -\ell+\ell=\0_X.
\]
Then $h_{T,\ell'}h_{T,\ell}=h_{T,\0_X}=\id_X$ due to Lemma 
\ref{h-composed}.  Since the map $h_{T,\ell}$ is invertible, it follows 
that $h_{T,\ell'}=h_{T,\ell}^{-1}$.  Consequently, $\ell\oplus_T\ell'=\ell 
h_{T,\ell'}+\ell'=\ell h_{T,\ell}^{-1}+\ell'=-\ell'+\ell'=\0_X$.  We 
conclude that the function $\ell$ is invertible with respect to the 
operation $\oplus_T$ and $\ell'$ is its inverse.
\end{proof}

\begin{lemma}\label{model-isomorphism}
Suppose that the homeomorphism $T$ is aperiodic.  Then the homomorphism 
$H_T$ is one-to-one.  Moreover, the restriction of $H_T$ to the group 
$C^\times_T(X,\bZ)$ is an isomorphism onto the group $[[T]]$.
\end{lemma}

\begin{proof}
If $T$ is aperiodic, then $T^n(x)\ne x$ for all points $x\in X$ and 
positive integers $n$.  As $T$ is invertible, it follows that for any $x\in 
X$ and $n_1,n_2\in\bZ$ we have $T^{n_1}(x)\ne T^{n_2}(x)$ whenever $n_1\ne 
n_2$.  In particular, $h_{T,\ell_1}(x)=T^{\ell_1(x)}(x)\ne 
T^{\ell_2(x)}(x)=h_{T,\ell_2}(x)$ whenever $\ell_1(x)\ne\ell_2(x)$.  As a 
consequence, the homomorphism $H_T$ is one-to-one.

By Lemma \ref{image-of-group}, $H_T\bigl(C^\times_T(X,\bZ)\bigr)=[[T]]$.  
Hence the restriction of $H_T$ to the group $C^\times_T(X,\bZ)$ is a 
homomorphism onto the group $[[T]]$.  In the case when the homomorphism 
$H_T$ is one-to-one, the restriction is an isomorphism onto $[[T]]$.
\end{proof}

Any continuous function $f:X_1\to X_2$ induces a map $\Phi_f:C(X_2,\bZ)\to 
C(X_1,\bZ)$ given by $\Phi_f(\ell)=\ell f$ for all $\ell\in C(X_2,\bZ)$.

\begin{lemma}\label{monoid-embed}
Suppose $T_1:X_1\to X_1$ and $T_2:X_2\to X_2$ are homeomorphisms of Cantor 
sets and $f:X_1\to X_2$ is a continuous function such that $fT_1=T_2f$.  
Then the map $\Phi_f$ is a homomorphism of the monoid $C(X_2,\bZ)$ with 
operation $\oplus_{T_2}$ to the monoid $C(X_1,\bZ)$ with operation 
$\oplus_{T_1}$.
\end{lemma}

\begin{proof}
First we prove that $fT_1^n=T_2^nf$ for all $n\in\bZ$.  For $n\ge0$, the 
proof is by induction on $n$.  Indeed, the case $n=0$ is trivial, and if 
$fT_1^k=T_2^kf$ for some $k\ge0$ then
\[
fT_1^{k+1}=(fT_1^k)T_1=(T_2^kf)T_1=T_2^k(fT_1)=T_2^k(T_2f)=T_2^{k+1}f.
\]
In the case $n<0$, we have $-n>0$ so that $fT_1^{-n}=T_2^{-n}f$ by the 
above.  Then
\[
fT_1^n=(T_2^nT_2^{-n})fT_1^n=T_2^n(T_2^{-n}f)T_1^n=T_2^n(fT_1^{-n})T_1^n
=T_2^nf(T_1^{-n}T_1^n)=T_2^nf.
\]
Given functions $\ell_1,\ell_2\in C(X_2,\bZ)$, we need to show that 
$\Phi_f(\ell_1\oplus_{T_2}\ell_2)=\Phi_f(\ell_1)\oplus_{T_1} 
\Phi_f(\ell_2)$.  For any $x\in X_1$ we obtain
\begin{gather*}
\bigl(\Phi_f(\ell_1\oplus_{T_2}\ell_2)\bigr)(x)=
(\ell_1\oplus_{T_2}\ell_2)(f(x))=
\ell_1\bigl(T_2^{\ell_2(f(x))}(f(x))\bigr)+\ell_2(f(x)), \\
\bigl(\Phi_f(\ell_1)\oplus_{T_1}\Phi_f(\ell_2)\bigr)(x)=
(\ell_1f\oplus_{T_1}\ell_2f)(x)=
\ell_1f\bigl(T_1^{\ell_2f(x)}(x)\bigr)+\ell_2f(x).
\end{gather*}
Since $\ell_2(f(x))=\ell_2f(x)$ is an integer, the map $fT_1^{\ell_2f(x)}$ 
coincides with the map $T_2^{\ell_2f(x)}f$.  It follows that 
$\bigl(\Phi_f(\ell_1\oplus_{T_2}\ell_2)\bigr)(x)= 
\bigl(\Phi_f(\ell_1)\oplus_{T_1}\Phi_f(\ell_2)\bigr)(x)$.
\end{proof}

\begin{lemma}\label{group-embed}
Under the assumptions of Lemma \ref{monoid-embed}, the restriction of 
$\Phi_f$ to the group $C^\times_{T_2}(X_2,\bZ)$ is a homomorphism to the 
group $C^\times_{T_1}(X_1,\bZ)$.  This group homomorphism is one-to-one 
provided that the function $f$ is onto.
\end{lemma}

\begin{proof}
By Lemma \ref{monoid-embed}, the map $\Phi_f$ is a homomorphism of the 
monoid $C(X_2,\bZ)$ with operation $\oplus_{T_2}$ to the monoid 
$C(X_1,\bZ)$ with operation $\oplus_{T_1}$.  Take any function $\ell\in 
C(X_2,\bZ)$ invertible with respect to $\oplus_{T_2}$ and let $\ell'$ be 
its inverse.  Then $\ell\oplus_{T_2}\ell'=\ell'\oplus_{T_2}\ell 
=\0_{X_2}$, which implies that $\Phi_f(\ell)\oplus_{T_1}\Phi_f(\ell')= 
\Phi_f(\ell')\oplus_{T_1}\Phi_f(\ell)=\Phi_f(\0_{X_2})=\0_{X_1}$.  Hence 
the function $\Phi_f(\ell)$ is invertible with respect to $\oplus_{T_1}$ 
and $\Phi_f(\ell')$ is its inverse.  We obtain that $\Phi_f$ maps 
invertible elements of the monoid $C(X_2,\bZ)$ to invertible elements of 
the monoid $C(X_1,\bZ)$.  Consequently, the restriction of $\Phi_f$ to the 
group $C^\times_{T_2}(X_2,\bZ)$ is a homomorphism to the group 
$C^\times_{T_1}(X_1,\bZ)$.

If the map $f$ is onto then for any functions $\ell_1$ and $\ell_2$ on 
$X_2$ we have $\ell_1f=\ell_2f$ if and only if $\ell_1=\ell_2$.  It follows 
that the map $\Phi_f$ is one-to-one, and so is its restriction to 
$C^\times_{T_2}(X_2,\bZ)$.
\end{proof}

Now we are ready to prove Theorem \ref{TFG-embeds}.

\begin{proof}[Proof of Theorem \ref{TFG-embeds}]
Suppose $T_1:X_1\to X_1$ and $T_2:X_2\to X_2$ are aperiodic homeomorphisms 
of Cantor sets.  If $T_2$ is a continuous factor of $T_1$, there exists a 
continuous map $f:X_1\to X_2$ such that $f$ is onto and $fT_1=T_2f$.  Lemma 
\ref{group-embed} implies that the group $C^\times_{T_2}(X_2,\bZ)$ with 
operation $\oplus_{T_2}$ embeds into the group $C^\times_{T_1}(X_1,\bZ)$ 
with operation $\oplus_{T_1}$.  Since the homeomorphisms $T_1$ and $T_2$ 
are aperiodic, it follows from Lemma \ref{model-isomorphism} that the group 
$C^\times_{T_1}(X_1,\bZ)$ is isomorphic to the topological full group 
$[[T_1]]$ while the group $C^\times_{T_2}(X_2,\bZ)$ is isomorphic to 
$[[T_2]]$.  We conclude that the group $[[T_2]]$ embeds into $[[T_1]]$.
\end{proof}

Cantor minimal systems are obviously aperiodic.  In view of Proposition 
\ref{3-subshifts-factors}, Theorem \ref{TFG-embeds} implies that the 
topological full groups of the subshifts $\si|\Om(\xipd)$ and 
$\si|\Om(\xiGE)$ embed into the topological full group of the subshift 
$\si|\Om(\xiTM)$.  In particular, every subgroup of $[[\si|\Omega(\xipd)]]$ 
or $[[\si|\Omega(\xiGE)]]$ embeds into $[[\si|\Omega(\xiTM)]]$.  In the 
remainder of this section we are going to describe a subgroup of the group 
$[[\si|\Omega(\xiGE)]]$ that will be later proved to have intermediate 
growth.

We begin with a general construction of involutions in a topological full 
group.

\begin{definition}\label{def-2-cycle}
Let $U$ be a clopen subset of a Cantor set $X$.  For any homeomorphism 
$T:X\to X$ such that the image $T(U)$ is disjoint from $U$, we define a 
\textbf{generalized 2-cycle} $\de_{U;T}:X\to X$ by
\[
\de_{U;T}(x)=
\begin{cases}
T(x) & \text{if $x\in U$,}\\
T^{-1}(x) & \text{if $x\in T(U)$,}\\
x & \text{otherwise.}
\end{cases}
\]
\end{definition}

The generalized $2$-cycle $\de_{U;T}$ maps disjoint clopen sets $U$ and 
$T(U)$ onto each other while fixing the other points of $X$.  By 
construction, $\de_{U;T}$ is continuous and $\de_{U;T}^{-1}=\de_{U;T}$.  
Therefore, whenever $\de_{U;T}$ is defined, it belongs to the topological 
full group $[[T]]$.

Given an alphabet $\cA$, for any letter $z\in\cA$ let $[.z]_{\cA}$ denote 
the set of all bi-infinite strings $\ldots x_{-1}x_0.x_1x_2\ldots$ in 
$\cA^{\bZ}$ such that $x_1=z$, and let $[z.]_{\cA}$ denote the set of all 
strings such that $x_0=z$.  The sets $[.z]_{\cA}$ and $[z.]_{\cA}$ are 
cylinders in $\cA^{\bZ}$.  Clearly, the shift $\si_{\cA}$ maps $[.z]_{\cA}$ 
onto $[z.]_{\cA}$.  Let $\si_{\cA}|Y$ be a subshift such that $Y$ is a 
Cantor set.  Suppose that the letter $z$ is never doubled in elements of 
$Y$, that is, the word $zz$ is a subword of no string in $Y$.  Then the 
intersection $[.z]_{\cA}\cap [z.]_{\cA}\cap Y$ is empty.  Since 
$\si_{\cA}(Y)=Y$, it follows that the set $[.z]_{\cA}\cap Y$, which is a 
clopen subset of the Cantor set $Y$, is disjoint from 
$\si_{\cA}([.z]_{\cA}\cap Y)=[z.]_{\cA}\cap Y$.  Hence the generalized 
$2$-cycle $\de_{U;T}$, where $U=[.z]_{\cA}\cap Y$ and $T=\si_{\cA}|Y$, is 
defined.  We will use the short notation $\de_z$ for this involution.  

The generalized $2$-cycle $\de_z$ is uniquely determined when its domain 
$Y$ is specified.  It belongs to the topological full group of the subshift 
$\si|Y$.  The action of $\de_z$ on a bi-infinite string $\ldots 
x_{-1}x_0.x_1x_2\ldots$ can be described as follows.  If one of two letters 
$x_0$ and $x_1$ adjacent to the dot is $z$, then the dot is carried across 
that letter.  Otherwise the string is not changed.

It follows from Lemma \ref{xi-GE} that for any two consecutive letters in 
the infinite string $\xiGE$, one is $a$ while the other one is $c$ or $d$.  
Hence the string $\xiGE$ contains no double letters, that is, the words 
$aa$, $cc$ and $dd$ are not subwords of $\xiGE$.  Consequently, the same is 
true for any bi-infinite string in $\Om(\xiGE)$.  Therefore each of the 
generalized $2$-cycles $\de_a$, $\de_c$ and $\de_d$ is defined on 
$\Om(\xiGE)$.

\begin{theorem}\label{sub-main}
The subgroup of the topological full group $[[\si|\Omega(\xiGE)]]$ 
generated by involutions $\de_a$, $\de_c$ and $\de_d$ has intermediate 
growth.
\end{theorem}

To prove Theorem \ref{sub-main}, we are going to identify a known group of 
intermediate growth that is isomorphic to the subgroup in the theorem.  
This will be done in the next section.

\section{Groups of intermediate growth}\label{sect-growth}

Let $G$ be an infinite, finitely generated group.  Suppose $A$ is a finite 
generating set for $G$.  Let $A_\pm$ denote the subset of $G$ consisting of 
all generators in $A$ and their inverses.  Any word over the alphabet 
$A_\pm$ can be interpreted as a product in the group $G$.  This gives rise 
to a map $h_{G,A}:A_\pm^*\to G$, which is a homomorphism of the monoid 
$A_\pm^*$ onto $G$.  For any element $g\in G$, the length of the shortest 
word in the pre-image $h_{G,A}^{-1}(g)$ is called the \emph{length} of $g$ 
(relative to the generating set $A$).  For any integer $n\ge1$ let 
$\ga_{G,A}(n)$ be the number of all elements of the group $G$ of length at 
most $n$.  The function $\ga_{G,A}$ is called the \emph{growth function} of 
$G$.  The asymptotic behavior of $\ga_{G,A}(n)$ as $n\to\infty$ is referred 
to as the \emph{growth} of the group $G$.  Clearly, $\ga_{G,A}(n)$ does not 
exceed the number of words of length at most $n$ in $A_\pm^*$, which is 
$1+m+m^2+\dots+m^n=(m^{n+1}-1)/(m-1)$, where $m$ is the number of elements 
in $A_\pm$.  Hence the growth of $G$ is at most exponential.

\begin{definition}\label{def-growth}
Let $G$ be an infinite, finitely generated group with a finite generating 
set $A$.  We say that the group $G$ has \textbf{exponential growth} if 
$\ga_{G,A}(n)\ge c^n$ for some $c>0$ and all sufficiently large $n$.  We 
say that the group $G$ has \textbf{polynomial growth} if $\ga_{G,A}(n)\le 
n^\alpha$ for some $\alpha>0$ and all sufficiently large $n$.  We say that 
the group $G$ has \textbf{intermediate growth} if the growth of $G$ is 
slower than exponential but faster than polynomial, namely, for any 
$c,\alpha>0$ we have $n^\alpha<\ga_{G,A}(n)<c^n$ for all sufficiently large 
$n$.
\end{definition}

Definition \ref{def-growth} is not sensitive to the choice of a generating 
set for the group $G$.  Suppose $B$ is a different finite generating set 
for $G$.  Let $k_1$ be the maximal length of elements of $B$ relative to 
the generating set $A$.  Then any product of $n$ elements of $B$ or their 
inverses can be written as a product of at most $k_1n$ elements of $A$ or 
their inverses.  Therefore $\ga_{G,B}(n)\le\ga_{G,A}(k_1n)$ for all 
$n\ge1$.  Likewise, $\ga_{G,A}(n)\le\ga_{G,B}(k_2n)$ for all $n\ge1$, where 
$k_2$ is the maximal length of elements of $A$ relative to the generating 
set $B$.  It follows that whenever the growth function $\ga_{G,A}$ is of 
exponential, polynomial or intermediate growth, the growth function 
$\ga_{G,B}$ exhibits the same kind of growth.

Note that if an infinite, finitely generated group is not of exponential or 
polynomial growth, it need not have intermediate growth.  Indeed, there are 
groups with growth oscillating between exponential and sub-exponential.

Now we are going to describe a family of groups of intermediate growth 
constructed by the first author in \cite{Grig84}.  Each group $\cG_\om$ of 
the family is labeled by an infinite ternary string
$\om\in\{0,1,2\}^{\bN}$.  The group $\cG_\om$ is generated by a set $A_\om$ 
consisting of four transformations $a$, $b_\om$, $c_\om$ and $d_\om$.  In 
the original construction in \cite{Grig84}, the generators were acting on 
the interval $[0,1]$.  We consider a modification in which the generators 
act on the set $\{0,1\}^{\bN}$ of infinite binary strings.  This 
modification agrees with the original construction when we identify any 
binary string $x_1x_2x_3\ldots$ with the number 
$2^{-1}x_1+2^{-2}x_2+2^{-3}x_3+\dots$.

The generator $a$ is common for all groups in the family.  It is defined by 
$a(0\xi)=1\xi$ and $a(1\xi)=0\xi$ for all $\xi\in\{0,1\}^{\bN}$.  That is, 
$a$ changes the first symbol in every binary string.  The generators 
$b_\om$, $c_\om$ and $d_\om$ are defined using the following recursive 
rules: for any $\xi\in\{0,1\}^{\bN}$, $l\in\{0,1,2\}$ and 
$\om\in\{0,1,2\}^{\bN}$ we have
\begin{eqnarray*}
& b_{2\om}(0\xi)=c_{1\om}(0\xi)=d_{0\om}(0\xi)=0\xi,\\
& b_{0\om}(0\xi)=b_{1\om}(0\xi)=c_{0\om}(0\xi)=c_{2\om}(0\xi)
=d_{1\om}(0\xi)=d_{2\om}(0\xi)=0a(\xi),\\
& b_{l\om}(1\xi)=1b_\om(\xi), \qquad c_{l\om}(1\xi)=1c_\om(\xi), \qquad
d_{l\om}(1\xi)=1d_\om(\xi).
\end{eqnarray*}
The string $1^\infty=111\ldots$ is the only element of $\{0,1\}^{\bN}$ 
fixed by $b_\om$, $c_\om$ and $d_\om$ for all $\om\in\{0,1,2\}^{\bN}$.  Any 
other string $\xi\in\{0,1\}^{\bN}$ is uniquely represented as 
$\xi=1^k0x\eta$, where $x\in\{0,1\}$, $\eta\in\{0,1\}^{\bN}$, $k\ge0$, and 
$1^k$ denotes the symbol $1$ repeated $k$ times ($1^0=\varnothing$).  Then 
two of the transformations $b_\om$, $c_\om$ and $d_\om$ change only the 
symbol $x$ in the string $\xi$ while the third one does not change $\xi$ at 
all.  Which one of the three fixes $\xi$ is determined by the $(k+1)$-th 
symbol of the ternary string $\om$.  Namely, $\xi$ is fixed by $b_\om$, 
$c_\om$ or $d_\om$ if the $(k+1)$-th symbol of $\om$ is respectively $2$, 
$1$ or $0$.

It follows from the construction that $a$, $b_\om$, $c_\om$ and $d_\om$ are 
involutions.  Besides, $b_\om$, $c_\om$ and $d_\om$ commute with one 
another, and $b_\om c_\om=d_\om$, $b_\om d_\om=c_\om$ and $c_\om 
d_\om=b_\om$.  Therefore one of the transformations $b_\om$, $c_\om$ and 
$d_\om$ can be dropped from the generating set $A_\om$ as redundant.

It is easy to see that for any $\xi\in\{0,1\}^{\bN}$ and any $n\ge1$ the 
first $n$ symbols of strings $a(\xi$), $b_\om(\xi)$, $c_\om(\xi)$ and 
$d_\om(\xi)$ are uniquely determined by the first $n$ symbols of the string 
$\xi$ and the first $n-1$ symbols of the string $\om$.  As a consequence, 
for any $n\ge1$ the natural action of the group $\cG_\om$ on 
$\{0,1\}^{\bN}$ induces an action on binary words of length $n$ regarded as 
prefixes of infinite binary strings.  Thus we obtain a length-preserving 
action of $\cG_\om$ on the set $\{0,1\}^*$ of all binary words.  Note that 
the action of the generators $a$, $b_\om$, $c_\om$ and $d_\om$ can be given 
by the same recursive formulas as their action on $\{0,1\}^{\bN}$, only 
this time $\xi$ is an arbitrary word in $\{0,1\}^*$ and we also need to let 
$a(\varnothing)=b_\om(\varnothing)= c_\om(\varnothing)=d_\om(\varnothing) 
=\varnothing$ for all $\om\in\{0,1,2\}^{\bN}$.

\begin{lemma}\label{transitive-on-levels}
The group $\cG_\om$ acts transitively on binary words of any given length.
\end{lemma}

\begin{proof}
The proof is by induction on the length $k$ of the binary words, 
simultaneously for all $\om\in\{0,1,2\}$.  The case $k=0$ is trivial as 
$\varnothing$ is the only word of length $0$.  Now assume that the lemma 
holds for words of some length $k\ge0$ and every $\om\in\{0,1,2\}$.  Any 
binary word of length $k+1$ can be written as $0w$ or $1w$, where $w$ is a 
word of length $k$.  Any infinite ternary string $\om$ can be written as 
$l\om'$, where $l\in\{0,1,2\}$ and $\om'\in\{0,1,2\}^{\bN}$.  We obtain 
that $b_\om(1w)=1b_{\om'}(w)$, $c_\om(1w)=1c_{\om'}(w)$ and 
$d_\om(1w)=1d_{\om'}(w)$.  Further, out of three words $b_\om(0w)$, 
$c_\om(0w)$ and $d_\om(0w)$, two coincide with $0a(w)$ while the third one 
is $0w$.  Consequently, two of three words $ab_\om a(1w)$, $ac_\om a(1w)$ 
and $ad_\om a(1w)$ coincide with $1a(w)$.  Let $S'$ be the set of all 
elements $g'\in\cG_{\om'}$ such that $g(1w)=1g'(w)$ for some $g\in\cG_\om$ 
and all words $w$ of length $k$.  It is easy to see that $S'$ is a subgroup 
of $\cG_{\om'}$.  By the above, $a,b_{\om'},c_{\om'},d_{\om'}\in S'$.  
Therefore $S'=\cG_{\om'}$.  By the inductive assumption, the group 
$\cG_{\om'}$ acts transitively on words of length $k$.  It follows that all 
words of length $k+1$ that begin with $1$ are in the same orbit of the 
action of $\cG_\om$.  Since the generator $a$ changes only the first symbol 
in each word of length $k+1$, we conclude that the group $\cG_\om$ acts 
transitively on binary words of length $k+1$.
\end{proof}

The set $\{0,1\}^{\bN}$ of infinite binary strings is endowed with the 
product topology.  The topology is generated by the cylinder sets of the 
form $w\{0,1\}^{\bN}=\bigl\{w\eta\mid\eta\in\{0,1\}^{\bN}\bigr\}$, where 
$w\in\{0,1\}^*$.  The fact that the action of the group $\cG_\om$ on 
$\{0,1\}^{\bN}$ induces a length-preserving action on finite binary strings 
implies that $\cG_\om$ acts on $\{0,1\}^{\bN}$ by homeomorphisms.

\begin{lemma}\label{density-of-orbits}
Every orbit of the group $\cG_\om$ is dense in $\{0,1\}^{\bN}$.
\end{lemma}

\begin{proof}
Let $O$ be any orbit of $\cG_\om$.  By Lemma \ref{transitive-on-levels}, 
the group $\cG_\om$ acts transitively on binary words of any given length.  
It follows that any binary word $w$ occurs as a prefix of some string in 
$O$.  In other words, the orbit $O$ has a point in the cylinder 
$w\{0,1\}^{\bN}$.  Thus $O$ is dense in $\{0,1\}^{\bN}$.
\end{proof}

In the case a ternary string $\om\in\{0,1,2\}^{\bN}$ is eventually 
constant, it is not hard to show that the group $\cG_\om$ has a cyclic 
subgroup of finite index.  Hence the growth of $\cG_\om$ is polynomial in 
this case.  In all other cases, the growth turns out to be intermediate.

\begin{theorem}[\cite{Grig84}]\label{Gri-growth}
If a string $\om\in\{0,1,2\}^{\bN}$ is not eventually constant, then the 
group $\cG_\om$ has intermediate growth.  Moreover,
$\gamma_{\cG_\om,A_\om}(n)\ge\exp(\alpha\sqrt{n})$ for some $\alpha>0$ and 
all sufficiently large $n$.  If, additionally, for some $m\ge1$ any $m$ 
consecutive symbols of $\om$ include $0$, $1$ and $2$ then 
$\gamma_{\cG_\om,A_\om}(n)\le\exp(n^\beta)$ for some $\beta<1$ and all 
sufficiently large $n$.
\end{theorem}

In this paper we are interested in the group $\cG_{(01)^\infty}$ labeled by 
the periodic string $(01)^\infty=010101\ldots$.  The growth of this group 
is much faster than the upper bound in Theorem \ref{Gri-growth}.  The 
precise rate of growth was established by Erschler \cite{Erschler04} as a 
particular case of the following general result.

\begin{theorem}[\cite{Erschler04}]\label{GE-growth}
Suppose that a string $\om\in\{0,1,2\}^{\bN}$ contains no $2$ and for 
some $m\ge1$ any $m$ consecutive symbols of $\om$ include both $0$ and 
$1$.  Then
\[
\exp\left(\frac{n}{\log^{2+\eps}(n)}\right)\le \gamma_{\cG_\om,A_\om}(n)\le
\exp\left(\frac{n}{\log^{1-\eps}(n)}\right)
\]
for any $\eps>0$ and all sufficiently (depending on $\eps$) large $n$.
\end{theorem}

Note that our notation for the groups $\cG_\om$ and their generators is 
different from Erschler's notation (cf.\@ Remark 1 in \cite{Erschler04}).  
We reformulated Theorem \ref{GE-growth} (Corollary 2$'$ in 
\cite{Erschler04}) accordingly.  The group $\cG_{(01)^\infty}$ was the 
primary example for the theorem.  It turns out that $\cG_{(01)^\infty}$ is 
isomorphic to the subgroup of the topological full group 
$[[\si|\Omega(\xiGE)]]$ constructed at the end of Section \ref{sect-TFG}.

\begin{proposition}\label{inj-homomorph}
There exists a one-to-one group homomorphism 
$h:\cG_{(01)^\infty}\to[[\si|\Omega(\xiGE)]]$ such that $h(a)=\de_a$, 
$h(c)=\de_c$ and $h(d)=\de_d$, where $c=c_{(01)^\infty}$, 
$d=d_{(01)^\infty}$.
\end{proposition}

\section{Schreier graphs}\label{sect-Schr}

In this section we prove Proposition \ref{inj-homomorph}, which will allow 
us to prove Theorem \ref{sub-main} as well as Theorem \ref{main}.  The 
proof of Proposition \ref{inj-homomorph} is going to rely on the use of the 
Schreier graphs.

A \emph{graph} is a combinatorial object that consists of \emph{vertices} 
and \emph{edges} related so that every edge joins two vertices or a vertex 
to itself (in the latter case the edge is called a \emph{loop} at that 
vertex).  Vertices joined by an edge are called its \emph{endpoints}.  We 
allow multiple edges joining the same pair of vertices as well as multiple 
loops at the same vertex.  We consider graphs with \emph{undirected} edges, 
which means that the two endpoints of a non-loop edge are not ordered.  
Also, our graphs have labeled edges, that is, every edge is assigned a 
\emph{label}.

The vertices of a graph are pictured as dots or small circles.  An edge is 
pictured as an arc joining its endpoints.  The label of an edge is written 
next to the edge.

\begin{definition}\label{def-covering}
Let $\Gamma_1$ be a graph with vertex set $V_1$ and edge set $E_1$.  Let 
$\Gamma_2$ be a graph with vertex set $V_2$ and edge set $E_2$.  An onto 
map $f:V_1\to V_2$ is called a \textbf{covering} of the graph $\Gamma_2$ by 
the graph $\Gamma_1$ (and then $\Gamma_1$ is said to \textbf{cover} 
$\Gamma_2$) if there exists an onto map $h:E_1\to E_2$ such that for any 
edge $e\in E_1$ the endpoints of $h(e)$ are images of the endpoints of $e$ 
under $f$ and the label of $h(e)$ matches the label of $e$.  The covering 
$f:V_1\to V_2$ is called an \textbf{isomorphism} of the graph $\Gamma_1$ 
onto the graph $\Gamma_2$ (and then $\Gamma_1$ is said to be 
\textbf{isomorphic} to $\Gamma_2$) if $f$ is bijective and the associated 
map $h:E_1\to E_2$ can be chosen to be bijective.
\end{definition}

A \emph{path} of length $n\ge0$ in a graph $\Gamma$ is a sequence of its 
vertices $v_0,v_1,\dots,v_n$ along with a sequence of edges $e_1,\dots,e_n$ 
such that for any $i$, $1\le i\le n$ the edge $e_i$ joins the vertex 
$v_{i-1}$ to $v_i$.  We say that the path connects the vertex $v_0$ to 
$v_n$.  The graph $\Gamma$ is called \emph{connected} if any two vertices 
of $\Gamma$ are connected by a path.  The length of the shortest path 
connecting a vertex $u$ to a vertex $v$ is called the \emph{distance} from 
$u$ to $v$ in the graph.  This distance function turns the vertex set of a 
connected graph into a metric space.

\begin{definition}\label{def-Schr}
Let $G$ be a group with a finite generating set $A$.  Let 
$\alpha:G\curvearrowright X$ be an action of $G$ on a set $X$ and $O$ be an 
orbit of that action.  The \textbf{Schreier graph} 
$\Gamma_{\Sch}(G,A;\al,O)$ of the orbit $O$ of the action $\alpha$ of the 
group $G$ relative to the generating set $A$ is a graph with vertex set $O$ 
and edges labeled by elements of $A$.  For any $u,v\in O$ and $l\in A$ the 
graph $\Gamma_{\Sch}(G,A;\al,O)$ has a (unique) edge with label $l$ 
connecting the vertex $u$ to $v$ if and only if $lu=v$ or $lv=u$ within the 
action $\al$.  In the case the generators in $A$ are indexed by letters of 
an alphabet $\cA$, we might regard edge labels in 
$\Gamma_{\Sch}(G,A;\al,O)$ as elements of $\cA$.
\end{definition}

If $l\in A$ is a generator of the group $G$ then any vertex $v\in O$ of the 
Schreier graph $\Gamma_{\Sch}(G,A;\al,O)$ is joined to $lv$ and $l^{-1}v$ 
by edges with label $l$.  It follows that the graph 
$\Gamma_{\Sch}(G,A;\al,O)$ is connected.  Moreover, for any $v\in O$ and 
$g\in G$ the distance between vertices $v$ and $gv$ in this graph does not 
exceed the length of the element $g$ relative to the generating set $A$.

We denote by $\Sch(G,A)$ the set of all Schreier graphs of the form 
$\Gamma_{\Sch}(G,A;\al,O)$.  If a generator $l\in A$ is an involution 
($l^{-1}=l$), then the conditions $lu=v$ and $lv=u$ are equivalent for all 
$u,v\in O$.  Hence the action of $l$ on the orbit $O$ is uniquely recovered 
from the graph $\Gamma_{\Sch}(G,A;\al,O)$.  In the case when all generators 
in $A$ are involutions, the graph $\Gamma_{\Sch}(G,A;\al,O)$ allows to 
reconstruct the action of the entire group $G$ on $O$.  Therefore in that 
case, elements of the set $\Sch(G,A)$ encode all transitive actions of the 
group $G$.

\begin{definition}\label{def-I_A}
Given an alphabet $\cA$, we define a group $\cI_{\cA}$ by 
$\cI_{\cA}=\langle\cA\mid z^2=1,\ z\in\cA\rangle$.  That is, elements of 
$\cA$ are considered nontrivial involutions, and the group $\cI_{\cA}$ is 
freely generated by them.
\end{definition}

\begin{lemma}\label{Sch-I_A}
A graph $\Gamma$ with labeled edges belongs to $\Sch(\cI_{\cA},\cA)$ if and 
only if the following conditions hold:
\begin{itemize}
\item
the graph $\Gamma$ is connected;
\item
all edge labels are in $\cA$;
\item
for any vertex $v$ of $\Gamma$ and any $z\in\cA$ there is a unique edge 
with $v$ as an endpoint and labeled by $z$.
\end{itemize}
\end{lemma}

\begin{proof}
Any Schreier graph is connected.  For a graph of the form 
$\Gamma_{\Sch}(\cI_{\cA},\cA;\al,O)$, all edge labels are in $\cA$.  
Furthermore, for any $u,v\in O$ and $z\in\cA$, the relations $zu=v$ and 
$zv=u$ are equivalent since $z$ is an involution.  Hence for any $v\in O$ 
and $z\in\cA$, the vertex $v$ is an endpoint of a unique edge with label 
$z$, which joins $v$ to $zv$.

Conversely, assume that a graph $\Gamma$ satisfies all three conditions of 
the lemma.  Let $V$ be the vertex set of $\Gamma$.  For any $z\in\cA$ and 
$v\in V$, the vertex $v$ is an endpoint of a unique edge with label $z$.  
Let $T_z(v)$ denote the vertex of $\Gamma$ joined to $v$ by that edge.  
This defines a family of transformations $T_z$, $z\in\cA$ of the set $V$.  
Since edges of the graph $\Gamma$ are undirected, it follows that each 
$T_z$ is an involution.  Then there exists a homomorphism 
$\phi:\cI_{\cA}\to\mathrm{Sym}(V)$ of the group $\cI_{\cA}$ to the group 
$\mathrm{Sym}(V)$ of all invertible transformations of $V$ such that 
$\phi(z)=T_z$ for all $z\in\cA$.  The homomorphism $\phi$ induces an action 
$\al:\cI_{\cA}\curvearrowright V$ given by $gv=\bigl(\phi(g)\bigr)(v)$ for 
all $g\in\cI_{\cA}$ and $v\in V$.  Note that any two vertices of $\Gamma$ 
joined by an edge are in the same orbit of the action $\al$.  Since the 
graph $\Gamma$ is connected, it follows that the action is transitive.  By 
construction, $\Gamma=\Gamma_{\Sch}(\cI_{\cA},\cA;\al,V)$.
\end{proof}

Let $\xi$ be a word $x_1x_2\ldots x_n\in\cA^*$, an infinite string 
$x_1x_2x_3\ldots\in\cA^{\bN}$ or a bi-infinite string $\ldots 
x_{-1}x_0.x_1x_2\ldots\in\cA^{\bZ}$.  Suppose that there are no double 
letters in $\xi$.  Then we define a graph $\Gamma_{\cA}(\xi)$ with labeled 
edges as follows.  The vertex set $V$ of the graph is $\{0,1,\dots,n\}$ if 
$\xi$ is a word of length $n\ge0$, $\{0,1,2,\dots\}$ if $\xi$ is an 
infinite string, and $\bZ$ if $\xi$ is a bi-infinite string.  For any 
$i\in\bZ$ such that $\{i-1,i\}\subset V$, we connect the vertex $i$ to 
$i-1$ by a unique edge with label $x_i$.  These are the only non-loop edges 
of $\Gamma_{\cA}(\xi)$.  Further, for any $i\in V$ and $x\in\cA$ such that 
the vertex $i$ is not connected to $i-1$ or $i+1$ by an edge with label 
$x$, we add a unique loop at $i$ with label $x$.  These are the only loops 
of $\Gamma_{\cA}(\xi)$.

By construction, $\Gamma_{\cA}(\xi)$ is a connected graph with a linear 
structure.  All edge labels are letters in $\cA$.  Since $\xi$ contains no 
double letters, for any $i\in V$ and $x\in\cA$ the vertex $i$ is an 
endpoint of a unique edge with label $x$.  Hence the graph 
$\Gamma_{\cA}(\xi)$ belongs to $\Sch(\cI_{\cA},\cA)$ due to Lemma 
\ref{Sch-I_A}.  For examples of such graphs, see Figures \ref{fig-lvl123}, 
\ref{fig-halfline} and \ref{fig-line}.

Suppose $\si_{\cA}|Y$ is a subshift over an alphabet $\cA$ such that for 
any letter $z\in\cA$ the generalized $2$-cycle $\de_z$ is defined on $Y$.  
Let $G_Y$ be the subgroup of the topological full group $[[\si_{\cA}|Y]]$ 
generated by involutions $\de_z$, $z\in\cA$.  The group $G_Y$ is uniquely 
determined by $\cA$ and $Y$.

\begin{lemma}\label{Schr-subshift}
Under the above assumptions, for any bi-infinite string $\xi\in Y$ the 
orbit of $\xi$ under the action of the group $G_Y$ is 
$\{\si_{\cA}^i(\xi)\mid i\in\bZ\}$.  If the string $\xi$ is not periodic 
then the Schreier graph of its orbit relative to the generating set 
$\{\de_z\mid z\in\cA\}$ is isomorphic to $\Gamma_{\cA}(\xi)$.
\end{lemma}

\begin{proof}
Since $\xi\in Y$ and each of the generalized $2$-cycles $\de_z$, $z\in\cA$ 
is defined on $Y$, the string $\xi$ has no double letters.  Let $\xi=\ldots 
x_{-1}x_0.x_1x_2\ldots$.  Then $\si_{\cA}^i(\xi)=\ldots 
x_{i-1}x_i.x_{i+1}x_{i+2}\ldots$ for any $i\in\bZ$.  In the Schreier graph 
of the orbit of $\si_{\cA}^i(\xi)$ under the action of the group $G_Y$, for 
any $z\in\cA$ the vertex $\si_{\cA}^i(\xi)$ is an endpoint of a unique edge 
with label $z$.  This edge joins $\si_{\cA}^i(\xi)$ to 
$\si_{\cA}^{i-1}(\xi)$ if $z=x_i$, to $\si_{\cA}^{i+1}(\xi)$ if 
$z=x_{i+1}$, and to itself otherwise.

Let $O$ be the orbit of the string $\xi$ under the action of the group 
$G_Y$ and $\Gamma$ be the Schreier graph of that orbit relative to the 
generating set $\{\de_z\mid z\in\cA\}$ (with edge labels in $\cA$).  Since 
each generator is piecewise a power of the subshift $\si_{\cA}|Y$, the set 
$\{\si_{\cA}^i(\xi)\mid i\in\bZ\}$ is invariant under the action of $G_Y$.  
Therefore the orbit $O$ is contained in this set.  On the other hand, it 
follows from the above that for any $i\in\bZ$ the strings 
$\si_{\cA}^i(\xi)$ and $\si_{\cA}^{i-1}(\xi)$ are in the same orbit of the 
group $G_Y$.  We conclude that $O=\{\si_{\cA}^i(\xi)\mid i\in\bZ\}$.

Consider a map $f:\bZ\to O$ given by $f(i)=\si_{\cA}^i(\xi)$ for all 
$i\in\bZ$.  The map $f$ is onto.  It follows from the definition of the 
graph $\Gamma_{\cA}(\xi)$ that $f$ is a covering of the Schreier graph 
$\Gamma$ by $\Gamma_{\cA}(\xi)$.  In the case when the string $\xi$ is not 
periodic, the map $f$ is bijective.  Then $f$ is an isomorphism of the 
graph $\Gamma_{\cA}(\xi)$ onto $\Gamma$.
\end{proof}

Our next goal is to describe some Schreier graphs of the group 
$\cG_{(01)^\infty}$ relative to the generating set $\{a,c,d\}$ (see Lemmas 
\ref{Schr-level} and \ref{Schr-half-line} below).

\begin{figure}[t]
\centerline{\includegraphics[scale=1.25]{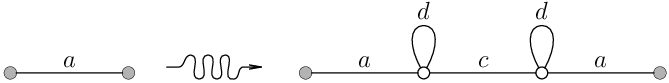}}
\caption{Edge substitution acting on a non-loop edge with label $a$.}
\label{fig-edgesub}
\end{figure}

\begin{lemma}\label{edge-sub}
Suppose $w$ is a nonempty word without double letters over the alphabet 
$\{a,c,d\}$.  Then the graph $\Gamma_{\{a,c,d\}}(w)$ can be transformed 
into a graph isomorphic to $\Gamma_{\{a,c,d\}}\bigl(\sub(w)\bigr)$ via the 
following edge substitution procedure:
\begin{itemize}
\item
every old edge with label $a$ joining two different vertices $u$ and $v$ is 
removed; instead we insert two new vertices $u'$ and $v'$, and add five new 
edges: an edge with label $a$ joining $u$ to $u'$, an edge with label $c$ 
joining $u'$ to $v'$, an edge with label $a$ joining $v'$ to $v$, and two 
loops with label $d$ at the vertices $u'$ and $v'$ (see Figure 
\ref{fig-edgesub});

\item
every old loop with label $a$ is preserved;

\item
for every old edge with label $c$, the label is changed to $d$;

\item
for every old edge with label $d$, the label is changed to $c$.
\end{itemize}
\end{lemma}

\begin{proof}
The procedure transforms the graph $\Gamma_{\{a,c,d\}}(w)$ into a graph 
$\Gamma$.  First we show that the graph $\Gamma$ satisfies all three 
conditions in Lemma \ref{Sch-I_A} so that 
$\Gamma\in\Sch(\cI_{\{a,c,d\}},\{a,c,d\})$.  Any two vertices of the graph 
$\Gamma_{\{a,c,d\}}(w)$ joined by an edge are also connected in $\Gamma$ 
(by an edge or by a path of length $3$).  Any newly added vertex in 
$\Gamma$ is connected to an old vertex by an edge.  Since 
$\Gamma_{\{a,c,d\}}(w)$ is a connected graph, it follows that the graph 
$\Gamma$ is connected as well.  Any vertex $v$ of $\Gamma_{\{a,c,d\}}(w)$ 
is an endpoint of three edges, which are labeled $a$, $c$ and $d$.  In the 
graph $\Gamma$, the edges with labels $c$ and $d$ are retained while their 
labels are exchanged.  The edge with label $a$ is retained if it is a 
loop.  Otherwise it is removed but another edge with an endpoint $v$ is 
added, which has label $a$.  Any newly added vertex of $\Gamma$ is an 
endpoint of three edges, labeled $a$, $c$ and $d$.  Thus the three 
conditions in Lemma \ref{Sch-I_A} are satisfied.

Let $\Gamma'$ and $\Gamma'_{\{a,c,d\}}\bigl(\sub(w)\bigr)$ be graphs 
obtained from respectively $\Gamma$ and 
$\Gamma_{\{a,c,d\}}\bigl(\sub(w)\bigr)$ by removing all loops.  The action 
of the edge substitution on non-loop edges mimics the action of 
the substitution $\sub$ on labels of those edges.  It follows that the 
graphs $\Gamma'$ and $\Gamma'_{\{a,c,d\}}\bigl(\sub(w)\bigr)$ are 
isomorphic.  Since the graphs $\Gamma$ and 
$\Gamma_{\{a,c,d\}}\bigl(\sub(w)\bigr)$ are both in 
$\Sch(\cI_{\{a,c,d\}},\{a,c,d\})$, for any vertex $v$ of either of the 
graphs and any letter $z\in\{a,c,d\}$ the loop with label $z$ at $v$ exists 
if and only if there is no edge with label $z$ joining $v$ to another 
vertex.  As a consequence, any isomorphism of the graph $\Gamma'$ onto 
$\Gamma'_{\{a,c,d\}}\bigl(\sub(w)\bigr)$ is also an isomorphism of $\Gamma$ 
onto $\Gamma_{\{a,c,d\}}\bigl(\sub(w)\bigr)$.
\end{proof}

\begin{figure}[t]
\centerline{\includegraphics{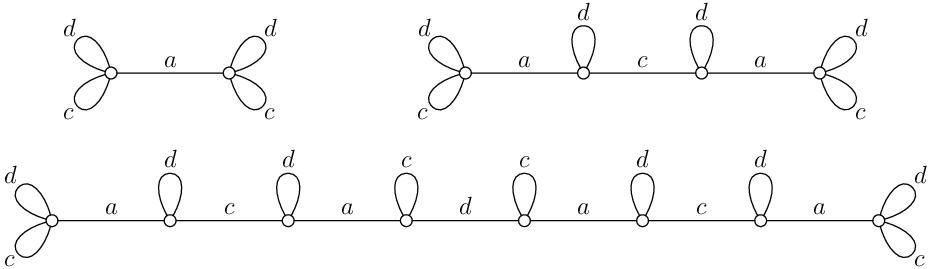}}
\caption{Graphs $\Gamma_{\{a,c,d\}}(a)$, 
$\Gamma_{\{a,c,d\}}\bigl(\sub(a)\bigr)$ and 
$\Gamma_{\{a,c,d\}}\bigl(\sub^2(a)\bigr)$.}
\label{fig-lvl123}
\end{figure}

\begin{lemma}\label{Schr-level}
The Schreier graph of the action of the group $\cG_{(01)^\infty}$ on binary 
words of length $k\ge1$ is isomorphic to the graph 
$\Gamma_{\{a,c,d\}}\bigl(\sub^{k-1}(a)\bigr)$.
\end{lemma}

\begin{proof}
For any $k\ge1$ let $V_k$ denote the set of all binary words of length 
$k$.  By Lemma \ref{transitive-on-levels}, the group $\cG_{(01)^\infty}$ 
acts 
transitively on $V_k$.  Let $\Gamma_k$ be the Schreier graph of the only 
orbit of that action.  We prove that $\Gamma_k$ is isomorphic to the graph 
$\Gamma_{\{a,c,d\}}\bigl(\sub^{k-1}(a)\bigr)$ by induction on $k$.  The 
case $k=1$ is easy.  Indeed, both $\Gamma_1$ and $\Gamma_{\{a,c,d\}}(a)$ 
have two vertices, $0$ and $1$.  Since $a(0)=1$, $a(1)=0$, $c(0)=d(0)=0$ 
and $c(1)=d(1)=1$, it follows that the identity map on the set $\{0,1\}$ is 
an isomorphism of the graph $\Gamma_1$ onto $\Gamma_{\{a,c,d\}}(a)$.

Now assume that the graph $\Gamma_k$ is isomorphic to
$\Gamma_{\{a,c,d\}}\bigl(\sub^{k-1}(a)\bigr)$ for some $k\ge1$.  We need to 
derive from this that the graph $\Gamma_{k+1}$ is isomorphic to
$\Gamma_{\{a,c,d\}}\bigl(\sub^k(a)\bigr)$.  Let $\Gamma$ be the graph 
obtained from $\Gamma_k$ by applying the edge substitution procedure 
described in Lemma \ref{edge-sub}.  Since $\Gamma_k$ is isomorphic to
$\Gamma_{\{a,c,d\}}\bigl(\sub^{k-1}(a)\bigr)$, it follows from Lemma 
\ref{edge-sub} that the graph $\Gamma$ is isomorphic to 
$\Gamma_{\{a,c,d\}}\bigl(\sub^k(a)\bigr)$.  It remains to prove that 
$\Gamma$ is also isomorphic to $\Gamma_{k+1}$.  Let $V$ be the vertex set 
of the graph $\Gamma$.  Then $V_k\subset V$ and any vertex of $\Gamma$ not 
in $V_k$ is joined to a vertex in $V_k$ by a unique edge, which has label 
$a$.  Since the generator $a$ fixes no nonempty word, the graph $\Gamma_k$ 
has no loops with label $a$.  It follows that the graph $\Gamma$ has twice 
as many vertices as $\Gamma_k$, and any vertex $v\in V_k$ is joined in 
$\Gamma$ by an edge to a unique vertex $v'\in V\setminus V_k$.  We let 
$f(v)=1v$ and $f(v')=0v$.  This defines a bijective map $f:V\to V_{k+1}$.  
We are going to show that $f$ is an isomorphism of the graph $\Gamma$ onto 
$\Gamma_{k+1}$.  After analyzing the structure of $\Gamma$, this task can 
be reformulated as follows.  For any $u,v\in V_k$ we have to show that 
$a(v)=u$ implies $c(0v)=0u$, $c(v)=u$ implies $d(1v)=1u$, and $d(v)=u$ 
implies $c(1v)=1u$.  Equivalently, $c(0v)=0a(v)$, $c(1v)=1d(v)$ and  
$d(1v)=1c(v)$ for all $v\in V_k$.  Also, we have to show that $a(1v)=0v$ 
and $d(0v)=0v$ for all $v\in V_k$.

Recall that $c=c_{(01)^\infty}$ and $d=d_{(01)^\infty}$.  It follows from 
the definition of the groups $\cG_\om$ in Section \ref{sect-growth} that 
$c_{(10)^\infty}=d_{(01)^\infty}=d$ and $d_{(10)^\infty}=c_{(01)^\infty} 
=c$.  Now for any $v\in V_k$ we obtain
\begin{eqnarray*}
& c(1v)=c_{(01)^\infty}(1v)=c_{0(10)^\infty}(1v)
=1c_{(10)^\infty}(v)=1d(v),\\
& d(1v)=d_{(01)^\infty}(1v)=d_{0(10)^\infty}(1v)
=1d_{(10)^\infty}(v)=1c(v),\\
& c(0v)=c_{(01)^\infty}(0v)=0a(v),\qquad
d(0v)=d_{(01)^\infty}(0v)=0v
\end{eqnarray*}
and, obviously, $a(1v)=0v$, just as required.
\end{proof}

\begin{lemma}\label{half-line-covers}
Suppose that for some infinite string $\xi$ and nonempty word $w$ without 
double letters over an alphabet $\cA$, the graph $\Gamma_{\cA}(\xi)$ covers 
the graph $\Gamma_{\cA}(w)$.  Then the string $\xi$ is of the form 
$wa_1\overleftarrow{w}a_2wa_3\overleftarrow{w}a_4\ldots$ or 
$\overleftarrow{w}a_1wa_2\overleftarrow{w}a_3wa_4\ldots$, where 
$\overleftarrow{w}$ is the word $w$ written backwards and each $a_i$ is a 
letter in $\cA$.
\end{lemma}

\begin{proof}
Let $\xi=x_1x_2x_3\ldots$ and $w=y_1y_2\ldots y_n$, where each $x_i$ and 
$y_j$ is in $\cA$.  The vertex set $V$ of the graph $\Gamma_{\cA}(\xi)$ is 
$\{0,1,2,\dots\}$.  The vertex set $V_0$ of the graph $\Gamma_{\cA}(w)$ is 
$\{0,1,\dots,n\}$.  For any vertex $v$ of either of the graphs and any 
$z\in\cA$, the vertex $v$ is an endpoint of a unique edge with label $z$ in 
this graph.  We denote that edge by $e_{v,z}$ for the graph 
$\Gamma_{\cA}(\xi)$ and by $\eps_{v,z}$ for the graph $\Gamma_{\cA}(w)$.  
Further, for any integer $i\ge1$ let $\tilde e_i$ denote the only edge of 
$\Gamma_{\cA}(\xi)$ joining vertices $i$ and $i-1$.  In the case $i\le n$, 
let $\tilde\eps_i$ denote the only edge joining the same vertices in the 
graph $\Gamma_{\cA}(w)$.  Note that the edge $\tilde e_i$ has label $x_i$.  
Hence $\tilde e_i=e_{i,x_i}=e_{i-1,x_i}$.  Likewise, the edge 
$\tilde\eps_i$ has label $y_i$ so that 
$\tilde\eps_i=\eps_{i,y_i}=\eps_{i-1,y_i}$.

Suppose $f:V\to V_0$ is a covering of the graph $\Gamma_{\cA}(w)$ by the 
graph $\Gamma_{\cA}(\xi)$ and $h:E\to E_0$ is the associated map of the 
edge set $E$ of $\Gamma_{\cA}(\xi)$ onto the edge set $E_0$ of 
$\Gamma_{\cA}(w)$.  Then $h(e_{v,z})=\eps_{f(v),z}$ for all $v\in V$ and 
$z\in\cA$.  Hence for any $v\in V$ the map $h$ establishes a one-to-one 
correspondence between edges of $\Gamma_{\cA}(\xi)$ that have $v$ as an 
endpoint and edges of $\Gamma_{\cA}(w)$ that have $f(v)$ as an endpoint.

For any integer $i\ge1$, the edge $\tilde e_i$ joins vertices $i$ and $i-1$ 
in the graph $\Gamma_{\cA}(\xi)$.  Therefore the edge $h(\tilde e_i)$ joins 
vertices $f(i)$ and $f(i-1)$ in the graph $\Gamma_{\cA}(w)$.  It follows 
that $f(i-1)$ is either $f(i)-1$ or $f(i)$ or $f(i)+1$.  If $f(i-1)=f(i)-1$ 
then $h(\tilde e_i)=\tilde\eps_{f(i)}$.  Hence $x_i=y_{f(i)}$ since the 
edge $h(\tilde e_i)$ has the same label as $\tilde e_i$.  If 
$f(i-1)=f(i)+1$ then $h(\tilde e_i)=\tilde\eps_{f(i)+1}$ so that 
$x_i=y_{f(i)+1}$.

Any vertex $v\in V$ is an endpoint of one or two non-loop edges in $E$.  
Likewise, $f(v)$ is an endpoint of one or two non-loop edges in $E_0$.  We 
are going to explore all possibilities here.  First assume that 
$0<f(v)<n$.  Then $f(v)$ is an endpoint of two non-loop edges 
$\tilde\eps_{f(v)}$ and $\tilde\eps_{f(v)+1}$.  We have 
$h(e)=\tilde\eps_{f(v)}$ and $h(e')=\tilde\eps_{f(v)+1}$ for some edges 
$e,e'\in E$ that have $v$ as an endpoint.  Since $h$ clearly maps loops to 
loops, neither $e$ nor $e'$ is a loop.  It follows that $v>0$, one of the 
edges $e$ and $e'$ is $\tilde e_v$, and the other is $\tilde e_{v+1}$.  If 
$h(\tilde e_v)=\tilde\eps_{f(v)}$ and $h(\tilde 
e_{v+1})=\tilde\eps_{f(v)+1}$, then $f(v-1)=f(v)-1$ and $f(v+1)=f(v)+1$.  
If $h(\tilde e_v)=\tilde\eps_{f(v)+1}$ and $h(\tilde 
e_{v+1})=\tilde\eps_{f(v)}$, then $f(v-1)=f(v)+1$ and $f(v+1)=f(v)-1$.

Next consider the case $v=0$.  It follows from the above that $f(0)=0$ or 
$n$ so that $f(0)$ is an endpoint of a single non-loop edge $\eps\in E_0$.  
We have $h(e)=\eps$ for some edge $e\in E$ that has $0$ as an endpoint.  
The edge $e$ is not a loop.  Therefore $e=\tilde e_1$.  Then $f(1)$ is the 
other endpoint of the edge $\eps$.  That is, $f(1)=1$ if $f(0)=0$, and 
$f(1)=n-1$ if $f(0)=n$.

Next assume that $v>0$ and $f(v)=0$.  The edge $\tilde\eps_1$ is the only 
non-loop edge in $E_0$ with an endpoint $0$.  We have $h(e)=\tilde\eps_1$ 
for some edge $e\in E$ with an endpoint $v$.  The edge $e$ is not a loop.  
Therefore $e=\tilde e_v$ or $\tilde e_{v+1}$.  If $f(\tilde 
e_v)=\tilde\eps_1$ then $f(\tilde e_{v+1})\ne\tilde\eps_1$, which implies 
that $f(\tilde e_{v+1})$ is a loop at $0$.  Hence $f(v-1)=1$ and 
$f(v+1)=0$.  If $f(\tilde e_{v+1})=\tilde\eps_1$ then $f(\tilde e_v)$ is a 
loop so that $f(v-1)=0$ and $f(v+1)=1$.

The only remaining possibility is $v>0$ and $f(v)=n$.  Similarly to the 
previous case, here we obtain that one of the vertices $f(v-1)$ and 
$f(v+1)$ is $n-1$ while the other is $n$.

By the above any triple of consecutive elements of the sequence 
$f(0),f(1),f(2),\dots$ is of the form $i-1,i,i+1$ or $i+1,i,i-1$ or $0,0,1$ 
or $1,0,0$ or $n-1,n,n$ or $n,n,n-1$.  Besides, the pair $f(0),f(1)$ is 
either $0,1$ or $n,n-1$.  As a consequence, the sequence naturally splits 
into blocks of $n+1$ consecutive elements where each block is either 
$0,1,2,\dots,n-1,n$ or $n,n-1,\dots,2,1,0$ and adjacent blocks are never 
the same.  Any of these blocks is of the form
$f\bigl(m(n+1)\bigr),f\bigl(m(n+1)+1\bigr),\dots,f\bigl(m(n+1)+n\bigr)$ for 
some integer $m\ge0$.  If the block coincides with $0,1,2,\dots,n-1,n$ then 
it follows from the above that the subword $x_{m(n+1)}x_{m(n+1)+1}\ldots 
x_{m(n+1)+n-1}$ of the string $\xi$ coincides with the word $w$.  If the 
block coincides with $n,n-1,\dots,2,1,0$, it follows that the same subword 
coincides with the word $\overleftarrow{w}$.  We conclude that 
$\xi=wa_1\overleftarrow{w}a_2wa_3\overleftarrow{w}a_4\ldots$ or 
$\xi=\overleftarrow{w}a_1wa_2\overleftarrow{w}a_3wa_4\ldots$, where 
$a_i=x_{(n+1)i-1}$ for all $i\ge1$.
\end{proof}

\begin{figure}[t]
\centerline{\includegraphics[scale=0.95]{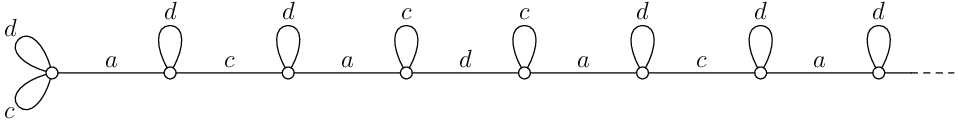}}
\caption{Graph $\Gamma_{\{a,c,d\}}(\xiGE)$.}
\label{fig-halfline}
\end{figure}

\begin{lemma}\label{Schr-half-line}
The Schreier graph of the action of the group $\cG_{(01)^\infty}$ on the 
orbit of the string $1^\infty$ is isomorphic to the graph 
$\Gamma_{\{a,c,d\}}(\xiGE)$.
\end{lemma}

\begin{proof}
Let $O$ be the orbit of $1^\infty$ under the action of the group 
$\cG_{(01)^\infty}$ and $\Gamma$ be the Schreier graph of that orbit 
relative to the generating set $\{a,c,d\}$.  We have 
$a(1^\infty)=01^\infty$ and $c(1^\infty)=d(1^\infty)=1^\infty$.  Hence 
$1^\infty$ is joined to $01^\infty$ in the graph $\Gamma$ by an edge with 
label $a$, and there are two loops at $1^\infty$ with labels $c$ and $d$.  
Any string $\zeta\in O$ different from $1^\infty$ can be written as 
$1^k0\eta$ for some $k\ge0$ and $\eta\in\{0,1\}^{\bN}$.  Recall that 
$c=c_{(01)^\infty}$ and $d=d_{(01)^\infty}$.  If $k$ even then 
$c(\zeta)=1^k0a(\eta)$ and $d(\zeta)=\zeta$.  If $k$ odd then 
$c(\zeta)=\zeta$ and $d(\zeta)=1^k0a(\eta)$.  Since $cd=dc$, in either case 
we have $cd(\zeta)=1^k0a(\eta)$.  Note that $\zeta$, $a(\zeta)$ and 
$cd(\zeta)$ are three different strings.  Hence the vertex $\zeta$ of the 
graph $\Gamma$ is joined to $a(\zeta)$ by an edge with label $a$, to 
$cd(\zeta)$ by an edge with label $c$ or $d$, and to itself by one loop.

We obtained that every vertex $\zeta\ne1^\infty$ of the Schreier graph 
$\Gamma$ is joined by edges to two other vertices while $1^\infty$ is 
joined by an edge to only one other vertex.  Besides, $\Gamma$ has no 
multiple non-loop edges.  Since the graph $\Gamma$ is connected, it follows 
that it is isomorphic to the graph $\Gamma_{\{a,c,d\}}(\xi)$ for some 
infinite string $\xi\in\{a,c,d\}^{\bN}$.  Moreover, the isomorphism 
$f:\{0,1,2,\dots\}\to O$ of $\Gamma_{\{a,c,d\}}(\xi)$ onto $\Gamma$ is 
unique and satisfies $f(0)=1^\infty$.  To be precise, it follows from the 
above that $f(2m)=(cda)^m(1^\infty)$ and $f(2m+1)=a(cda)^m(1^\infty)$ for 
all integers $m\ge0$.  Now it remains to show that $\xi=\xiGE$.

For any $k\ge1$ let $V_k$ denote the set of all binary words of length 
$k$.  The action of the group $\cG_{(01)^\infty}$ on $V_k$ is transitive 
due to Lemma \ref{transitive-on-levels}.  Let $\Gamma_k$ be the Schreier 
graph of the only orbit of that action.  Consider a map $f_k:O\to V_k$ that 
assigns to any infinite string in $O$ its prefix of length $k$.  As the 
orbit $O$ is dense in $\{0,1\}^{\bN}$ (due to Lemma 
\ref{density-of-orbits}), the map $f_k$ is onto.  Since the action of the 
group $\cG_{(01)^\infty}$ on $V_k$ is induced by its action on 
$\{0,1\}^{\bN}$, it follows that $f_k$ is a covering of the graph 
$\Gamma_k$ by the Schreier graph $\Gamma$ of the orbit $O$.  By Lemma 
\ref{Schr-level}, the graph $\Gamma_k$ is isomorphic to
$\Gamma_{\{a,c,d\}}\bigl(\sub^{k-1}(a)\bigr)$.  Since $\Gamma$ is 
isomorphic to the graph $\Gamma_{\{a,c,d\}}(\xi)$, we obtain that the 
latter covers the graph $\Gamma_{\{a,c,d\}}\bigl(\sub^{k-1}(a)\bigr)$.  
Then Lemma \ref{half-line-covers} implies that the string $\xi$ is of the 
form $wa_1\overleftarrow{w}a_2wa_3\overleftarrow{w}a_4\ldots$ or 
$\overleftarrow{w}a_1wa_2\overleftarrow{w}a_3wa_4\ldots$, where 
$w=\sub^{k-1}(a)$, $\overleftarrow{w}$ is $w$ written backwards, and each 
$a_i$ is a letter in $\{a,c,d\}$.  Note that $\overleftarrow{w}=w$ (as 
shown in the proof of Proposition \ref{3-subshifts-Cantor}).  We conclude 
that $w=\sub^{k-1}(a)$ is a prefix of the string $\xi$.  Thus $\xi=\xiGE$ 
since this is the only string that has each of the words $\sub^{k-1}(a)$, 
$k\ge1$ as a prefix.
\end{proof}

\begin{definition}\label{def-kernel}
Let $\Gamma\in\Sch(\cI_{\cA},\cA)$ and $V$ be the vertex set of the graph 
$\Gamma$.  Then $\Gamma=\Gamma_{\Sch}(\cI_{\cA},\cA;\al_\Gamma,V)$ for a 
unique transitive action $\al_\Gamma:\cI_{\cA}\curvearrowright V$.  The 
action $\al_\Gamma$ gives rise to a group homomorphism 
$\psi_\Gamma:\cI_{\cA}\to\mathrm{Sym}(V)$, where $\mathrm{Sym}(V)$ is the 
group of all invertible transformations of the set $V$.  Namely, 
$\psi_\Gamma(g)$ is how the element $g\in\cI_{\cA}$ acts on $V$.  We denote 
by $\cK(\Gamma)$ the \emph{kernel} of $\psi_\Gamma$.  The kernel consists 
of all elements of $\cI_{\cA}$ that act trivially on $V$ within the action 
$\al_\Gamma$.
\end{definition}

Let $G$ be a group generated by a finite set $A$ of involutions.  Suppose 
that the generators of $G$ are labeled by elements of an alphabet $\cA$: 
$A=\{g_z\mid z\in\cA\}$.  Then for any Schreier graph in $\Sch(G,A)$ we 
assume that edge labels are in $\cA$.

\begin{lemma}\label{factor-by-kernel}
Under the above assumptions, let $\Gamma=\Gamma_{\Sch}(G,A;\al,O)$ for some 
orbit $O$ of an action $\al$ of the group $G$.  Then 
$\Gamma\in\Sch(\cI_{\cA},\cA)$ and there exists a group homomorphism 
$\phi:G\to\cI_{\cA}/\cK(\Gamma)$ such that $\phi(g_z)=z\cK(\Gamma)$ for all 
$z\in\cA$.  In the case when the restriction of the action $\al$ to the 
orbit $O$ is faithful, $\phi$ is an isomorphism.
\end{lemma}

\begin{proof}
Since the group $G$ is generated by involutions $g_z$, $z\in\cA$, there 
exists a group homomorphism $\pi:\cI_{\cA}\to G$ such that $\pi(z)=g_z$ for 
all $z\in\cA$.  This homomorphism is onto.  The restriction of the action 
$\al$ to the orbit $O$ gives rise to a group homomorphism 
$\psi_O:G\to\mathrm{Sym}(O)$.  Then the composition $\psi=\psi_O\pi$ is a 
homomorphism of the group $\cI_{\cA}$ to $\mathrm{Sym}(O)$.  The 
homomorphism $\psi$ induces an action $\beta:\cI_{\cA}\curvearrowright O$ 
given by $gv=\bigl(\psi(g)\bigr)(v)$ for all $g\in\cI_{\cA}$ and $v\in O$.  
Since $O$ is an orbit of the action $\al$, the natural action of the group 
$\psi_O(G)$ on $O$ is transitive.  As the map $\pi$ is onto, we have 
$\psi(\cI_{\cA})=\psi_O(G)$.  Hence the action $\beta$ is transitive as 
well.

Since $\psi(z)=\psi_O(g_z)$ for all $z\in\cA$, it follows that the graph 
$\Gamma=\Gamma_{\Sch}(G,A;\al,O)$ is also the Schreier graph 
$\Gamma_{\Sch}(\cI_{\cA},\cA;\beta,O)$.  In particular, 
$\Gamma\in\Sch(\cI_{\cA},\cA)$.  Besides, $\cK(\Gamma)$ is the kernel of 
the homomorphism $\psi$.  Therefore $\psi$ factors to a group isomorphism 
$\rho:\cI_{\cA}/\cK(\Gamma)\to\psi(\cI_{\cA})$ such that 
$\rho\bigl(h\cK(\Gamma)\bigr)=\psi(h)$ for all $h\in\cI_{\cA}$.  As 
$\psi(\cI_{\cA})=\psi_O(G)$, the composition $\phi=\rho^{-1}\psi_O$ is a 
well-defined map of $G$ to $\cI_{\cA}/\cK(\Gamma)$.  Since both $\psi_O$ 
and $\rho$ are group homomorphisms, so is $\phi$.  By construction,
$\phi(g_z)=\rho^{-1}(\psi(z))=z\cK(\Gamma)$ for all $z\in\cA$.

Note that the homomorphism $\phi$ is always onto.  In the case when the 
restriction of the action $\al$ to the orbit $O$ is faithful, the 
homomorphism $\psi_O$ is one-to-one.  Then $\phi$ is also one-to-one, which 
implies that it is an isomorphism.
\end{proof}

\begin{lemma}\label{kernel-cover}
Suppose $\Gamma_1,\Gamma_2\in\Sch(\cI_{\cA},\cA)$ are graphs such that 
$\Gamma_1$ covers $\Gamma_2$.  Then $\cK(\Gamma_1)\subset\cK(\Gamma_2)$.
\end{lemma}

\begin{proof}
Let $V_1$ be the vertex set of $\Gamma_1$ and $V_2$ be the vertex set of 
$\Gamma_2$.  There exist transitive actions 
$\al_1:\cI_{\cA}\curvearrowright V_1$ and $\al_2:\cI_{\cA}\curvearrowright 
V_2$ such that $\Gamma_1=\Gamma_{\Sch}(\cI_{\cA},\cA;\al_1,V_1)$ and 
$\Gamma_2=\Gamma_{\Sch}(\cI_{\cA},\cA;\al_2,V_2)$.  The actions give rise 
to group homomorphisms $\psi_1:\cI_{\cA}\to\mathrm{Sym}(V_1)$ and 
$\psi_2:\cI_{\cA}\to\mathrm{Sym}(V_2)$.  Then $\cK(\Gamma_1)$ is the kernel 
of $\psi_1$ and $\cK(\Gamma_2)$ is the kernel of $\psi_2$.  Suppose 
$f:V_1\to V_2$ is a covering of the graph $\Gamma_2$ by the graph 
$\Gamma_1$.  Then the equality $\bigl(\psi_1(z)\bigr)(v)=u$ implies 
$\bigl(\psi_2(z)\bigr)(f(v))=f(u)$ for all $u,v\in V_1$ and $z\in\cA$.  
Equivalently, $\psi_2(z)f=f\psi_1(z)$ for all $z\in\cA$.  It is easy to 
verify that the set of all elements $g\in\cI_{\cA}$ satisfying the relation 
$\psi_2(g)f=f\psi_1(g)$ is a subgroup.  Since every generator $z\in\cA$ of 
the group $\cI_{\cA}$ belongs to this set, we conclude that 
$\psi_2(g)f=f\psi_1(g)$ for all $g\in\cI_{\cA}$.

If $g\in\cK(\Gamma_1)$ then $\psi_1(g)$ is the identity map on $V_1$.  
Hence $\psi_2(g)f=f\psi_1(g)=f$.  Since the map $f$ is onto, it follows 
that $\psi_2(g)$ is the identity map on $V_2$.  Equivalently, 
$g\in\cK(\Gamma_2)$.  Thus $\cK(\Gamma_1)\subset\cK(\Gamma_2)$.
\end{proof}

\begin{figure}[t]
\centerline{\includegraphics[scale=1]{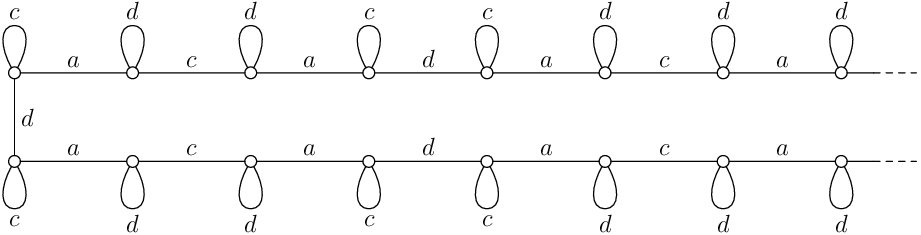}}
\caption{Graph $\Gamma_{\{a,c,d\}} 
\bigl(\protect\overleftarrow{\xiGE}d.\xiGE\bigr)$.}
\label{fig-line}
\end{figure}

\begin{lemma}\label{Sch-line-halfline}
Let $\xi$ be an infinite string without double letters over an alphabet 
$\cA$, $x$ be a letter in $\cA$ different from the first letter of $\xi$, 
and $\overleftarrow{\xi}$ be the left-infinite string obtained by writing 
$\xi$ backwards.  Then 
$\cK\bigl(\Gamma_{\cA}\bigl(\overleftarrow{\xi}\!x.\xi\bigr)\bigr) 
\subset\cK\bigl(\Gamma_{\cA}(\xi)\bigr)$.  Moreover, if 
$\al:\cI_{\cA}\curvearrowright\bZ$ is the transitive action such that 
$\Gamma_{\Sch}(\cI_{\cA},\cA;\al,\bZ)= 
\Gamma_{\cA}\bigl(\overleftarrow{\xi}\!x.\xi\bigr)$, then every element of 
$\cK\bigl(\Gamma_{\cA}(\xi)\bigr)$ moves at most finitely many points 
within that action.
\end{lemma}

\begin{proof}
Since the string $\xi$ has no double letters and $x$ is different from the 
first letter of $\xi$, the bi-infinite string $\overleftarrow{\xi}\!x.\xi$ 
also has no double letters.  Hence the graph 
$\Gamma_{\cA}\bigl(\overleftarrow{\xi}\!x.\xi\bigr)$ is defined.  The 
graphs $\Gamma_{\cA}\bigl(\overleftarrow{\xi}\!x.\xi\bigr)$ and 
$\Gamma_{\cA}(\xi)$ both belong to $\Sch(\cI_{\cA},\cA)$.  Therefore there 
exist unique transitive actions $\al:\cI_{\cA}\curvearrowright\bZ$ and 
$\beta:\cI_{\cA}\curvearrowright V$, where $V=\{0,1,2\dots\}$, such that 
$\Gamma_{\cA}\bigl(\overleftarrow{\xi}\!x.\xi\bigr)= 
\Gamma_{\Sch}(\cI_{\cA},\cA;\al,\bZ)$ and 
$\Gamma_{\cA}(\xi)=\Gamma_{\Sch}(\cI_{\cA},\cA;\beta,V)$.  The actions give 
rise to group homomorphisms $\psi_\al:\cI_{\cA}\to\mathrm{Sym}(\bZ)$ and 
$\psi_\beta:\cI_{\cA}\to\mathrm{Sym}(V)$.  Then 
$\cK\bigl(\Gamma_{\cA}\bigl(\overleftarrow{\xi}\!x.\xi\bigr)\bigr)$ is the 
kernel of $\psi_\al$ and $\cK\bigl(\Gamma_{\cA}(\xi)\bigr)$ is the kernel 
of $\psi_\beta$.

Consider a map $f:\bZ\to V$ given by $f(i)=i$ if $i\ge0$ and $f(i)=-1-i$ if 
$i<0$.  It is easy to observe that $f$ is a covering of the graph 
$\Gamma_{\cA}(\xi)$ by the graph 
$\Gamma_{\cA}\bigl(\overleftarrow{\xi}\!x.\xi\bigr)$.  Then Lemma 
\ref{kernel-cover} implies that 
$\cK\bigl(\Gamma_{\cA}\bigl(\overleftarrow{\xi}\!x.\xi\bigr)\bigr) 
\subset\cK\bigl(\Gamma_{\cA}(\xi)\bigr)$.  Moreover, just as in the proof 
of Lemma \ref{kernel-cover}, it follows that $\psi_\beta(g)f=f\psi_\al(g)$ 
for all $g\in\cI_{\cA}$.

Suppose $g\in\cK\bigl(\Gamma_{\cA}(\xi)\bigr)$.  Then $f(gu)=f(u)$ for all 
$u\in\bZ$, where $gu=\bigl(\psi_\al(g)\bigr)(u)$.  Any $v\in V$ has exactly 
two pre-images under the map $f$: $v$ and $-1-v$.  Therefore within the 
action $\al$, the element $g$ either exchanges $v$ and $-1-v$, or else 
fixes them both.  Recall that the distance between vertices $v$ and $gv$ in 
the Schreier graph $\Gamma_{\Sch}(\cI_{\cA},\cA;\al,\bZ)$ cannot exceed the 
length $|g|$ of $g$ relative to the generating set $\cA$.  The distance 
from $v$ to $-1-v$ is clearly $2v+1$.  Hence $g$ fixes $v$ and $-1-v$ 
whenever $2v+1>|g|$.  This leaves only finitely many vertices in 
$\Gamma_{\Sch}(\cI_{\cA},\cA;\al,\bZ)$ that the element $g$ can possibly 
move within the action $\al$.
\end{proof}

\begin{lemma}\label{GE-line-halfline}
$\cK\bigl(\Gamma_{\{a,c,d\}}\bigl(\overleftarrow{\xiGE}c.\xiGE\bigr)\bigr)
=\cK\bigl(\Gamma_{\{a,c,d\}}\bigl(\overleftarrow{\xiGE}d.\xiGE\bigr)\bigr)
=\cK\bigl(\Gamma_{\{a,c,d\}}(\xiGE)\bigr)$.
\end{lemma}

\begin{proof}
Let $x$ be either of the letters $c$ and $d$.  The graph 
$\Gamma_{\{a,c,d\}}\bigl(\overleftarrow{\xiGE}x.\xiGE\bigr)$ is also the 
Schreier graph $\Gamma_{\Sch}(\cI_{\{a,c,d\}},\{a,c,d\};\al,\bZ)$ for some 
transitive action $\al:\cI_{\{a,c,d\}}\curvearrowright\bZ$.  Then 
$\cK\bigl(\Gamma_{\{a,c,d\}}\bigl(\overleftarrow{\xiGE}x.\xiGE\bigr)\bigr)$ 
is the kernel of the group homomorphism $\psi_\al:\cI_{\{a,c,d\}}\to 
\mathrm{Sym}(\bZ)$ associated to the action $\al$.  By Lemma 
\ref{Sch-line-halfline}, 
$\cK\bigl(\Gamma_{\{a,c,d\}}\bigl(\overleftarrow{\xiGE}x.\xiGE\bigr)\bigr) 
\subset\cK\bigl(\Gamma_{\{a,c,d\}}(\xiGE)\bigr)$ and, moreover, 
every element of $\cK\bigl(\Gamma_{\{a,c,d\}}(\xiGE)\bigr)$ moves at most 
finitely many points within the action $\al$.  To prove that 
$\cK\bigl(\Gamma_{\{a,c,d\}}\bigl(\overleftarrow{\xiGE}x.\xiGE\bigr)\bigr)= 
\cK\bigl(\Gamma_{\{a,c,d\}}(\xiGE)\bigr)$, it is enough to show that every 
element of the group $\cI_{\{a,c,d\}}$ that moves at most finitely many 
points within the action $\al$ is, in fact, acting trivially.

By Lemma \ref{xiGE-examples}, the bi-infinite string 
$\overleftarrow{\xiGE}x.\xiGE$ belongs to the set $\Om(\xiGE)$.  Let $G$ be 
the subgroup of the topological full group of the subshift 
$\si|\Omega(\xiGE)$ generated by the generalized $2$-cycles $\de_a$, 
$\de_c$ and $\de_d$.  Let $O$ be the orbit of 
$\overleftarrow{\xiGE}x.\xiGE$ under the action of $G$ and $\Gamma$ be the 
Schreier graph of that orbit relative to the generating set 
$\{\de_a,\de_c,\de_d\}$ (with edge labels in $\{a,c,d\}$).  Lemma 
\ref{Schr-subshift} implies that $O$ coincides with the orbit of 
$\overleftarrow{\xiGE}x.\xiGE$ under the action of the cyclic group 
generated by the subshift $\si|\Omega(\xiGE)$.  Since the subshift is a 
Cantor minimal system (due to Proposition \ref{3-subshifts-Cantor}), it 
follows that the orbit $O$ is dense in $\Om(\xiGE)$.  As $\Om(\xiGE)$ is a 
Cantor set, the orbit $O$ is infinite, which implies that the string 
$\overleftarrow{\xiGE}x.\xiGE$ is not periodic.  Then it follows from Lemma 
\ref{Schr-subshift} that the Schreier graph $\Gamma$ is isomorphic to 
$\Gamma_{\{a,c,d\}}\bigl(\overleftarrow{\xiGE}x.\xiGE\bigr)$.  Let 
$f:O\to\bZ$ be an isomorphism of $\Gamma$ onto 
$\Gamma_{\{a,c,d\}}\bigl(\overleftarrow{\xiGE}x.\xiGE\bigr)$.  Then the 
equality $\de_z(v)=u$ is equivalent to $\bigl(\psi_\al(z)\bigr)(f(v))=f(u)$ 
for all $u,v\in O$ and $z\in\{a,c,d\}$.  Consequently, everywhere on $O$ we 
have $\psi_\al(z)f=f\de_z$ for all $z\in\{a,c,d\}$.  Since the generators 
of the group $G$ are involutions, there exists a group homomorphism 
$\pi:\cI_{\{a,c,d\}}\to G$ such that $\pi(z)=\de_z$ for all 
$z\in\{a,c,d\}$.  It is easy to verify that the set of all elements 
$g\in\cI_{\cA}$ satisfying the relation $\psi_\al(g)f=f\pi(g)$ everywhere 
on $O$ is a subgroup.  Since every generator $z\in\cA$ of the group 
$\cI_{\cA}$ belongs to this set, we conclude that $\psi_\al(g)f=f\pi(g)$ on 
$O$ for all $g\in\cI_{\cA}$.

Suppose $g\in\cI_{\{a,c,d\}}$ is an element that moves at most finitely 
many points within the action $\al$.  Since $\psi_\al(g)f=f\pi(g)$ on $O$ 
and the map $f$ is bijective, it follows that $\pi(g)$ moves a point $v\in 
O$ if and only if $g$ moves $f(v)$ within the action $\al$.  Hence $\pi(g)$ 
moves at most finitely many points of the orbit $O$.  As $\Om(\xiGE)$ is a 
Cantor set, any dense subset of $\Om(\xiGE)$ remains dense after removing 
finitely many points from it.  Therefore the set of all points of the orbit 
$O$ fixed by $\pi(g)$ is dense in $\Om(\xiGE)$.  Since $\pi(g)$ is a 
homeomorphism of $\Om(\xiGE)$, it has to coincide with the identity map.  
Hence $g$ acts trivially on $\bZ$ within the action $\al$.
\end{proof}

Now we are ready to prove Proposition \ref{inj-homomorph}.

\begin{proof}[Proof of Proposition \ref{inj-homomorph}]
To distinguish the generators of the group $\cG_{(01)^\infty}$ from the 
generators of the group $\cI_{\{a,c,d\}}$, in this proof we use notation 
$\tilde a,\tilde c,\tilde d$ for the former ones.

Let $O_1$ denote the orbit of the infinite string $1^\infty$ under the 
natural action of the group $\cG_{(01)^\infty}$ on $\{0,1\}^{\bN}$.  Let 
$\Gamma_1$ be the Schreier graph of that orbit.  By Lemma 
\ref{factor-by-kernel}, there exists a group homomorphism 
$\phi_1:\cG_{(01)^\infty}\to\cI_{\{a,c,d\}}/\cK(\Gamma_1)$ such that 
$\phi_1(\tilde a)=a\cK(\Gamma_1)$, $\phi_1(\tilde c)=c\cK(\Gamma_1)$ and 
$\phi_1(\tilde d)=d\cK(\Gamma_1)$.  The orbit $O_1$ is dense in 
$\{0,1\}^{\bN}$ due to Lemma \ref{density-of-orbits}.  Since the group 
$\cG_{(01)^\infty}$ acts on $\{0,1\}^{\bN}$ by homeomorphisms, it follows 
that each element of $\cG_{(01)^\infty}$ is uniquely determined by its 
restriction to $O_1$.  In other words, the restriction of the natural 
action of $\cG_{(01)^\infty}$ to the orbit $O_1$ is faithful.  Then Lemma 
\ref{factor-by-kernel} implies that $\phi_1$ is an isomorphism.

Let $G$ be the subgroup of the topological full group 
$[[\si|\Omega(\xiGE)]]$ generated by the generalized $2$-cycles $\de_a$, 
$\de_c$ and $\de_d$.  By Lemma \ref{xiGE-examples}, the bi-infinite string 
$\overleftarrow{\xiGE}d.\xiGE$ belongs to $\Om(\xiGE)$.  Let $O_2$ denote 
the orbit of this string under the natural action of $G$ on $\Om(\xiGE)$.  
Let $\Gamma_2$ be the Schreier graph of the orbit $O_2$.  By Lemma 
\ref{factor-by-kernel}, there exists a group homomorphism 
$\phi_2:G\to\cI_{\{a,c,d\}}/\cK(\Gamma_2)$ such that 
$\phi_2(\de_a)=a\cK(\Gamma_2)$, $\phi_2(\de_c)=c\cK(\Gamma_2)$ and 
$\phi_2(\de_d)=d\cK(\Gamma_2)$.  It was shown in the proof of Lemma 
\ref{GE-line-halfline} that the orbit $O_2$ is dense in $\Om(\xiGE)$.  
Therefore each element of the group $G$ is uniquely determined by its 
restriction to $O_2$ so that the restriction of the natural action of $G$ 
to the orbit $O_2$ is faithful.  Then Lemma \ref{factor-by-kernel} implies 
that $\phi_2$ is an isomorphism.

By Lemma \ref{Schr-half-line}, the Schreier graph $\Gamma_1$ is isomorphic 
to $\Gamma_{\{a,c,d\}}(\xiGE)$.  It was shown in the proof of Lemma 
\ref{GE-line-halfline} that the Schreier graph $\Gamma_2$ is isomorphic to 
$\Gamma_{\{a,c,d\}}\bigl(\overleftarrow{\xiGE}d.\xiGE\bigr)$.  Since 
isomorphic graphs cover each other, it follows from Lemma 
\ref{kernel-cover} that $\cK(\Gamma_1)= 
\cK\bigl(\Gamma_{\{a,c,d\}}(\xiGE)\bigr)$ and $\cK(\Gamma_2)=
\cK\bigl(\Gamma_{\{a,c,d\}}\bigl(\overleftarrow{\xiGE}d.\xiGE\bigr)\bigr)$. 
Besides, $\cK\bigl(\Gamma_{\{a,c,d\}}(\xiGE)\bigr)= 
\cK\bigl(\Gamma_{\{a,c,d\}}\bigl(\overleftarrow{\xiGE}d.\xiGE\bigr)\bigr)$ 
due to Lemma \ref{GE-line-halfline}.  Therefore 
$\cK(\Gamma_1)=\cK(\Gamma_2)$ so that the factor groups 
$\cI_{\{a,c,d\}}/\cK(\Gamma_1)$ and $\cI_{\{a,c,d\}}/\cK(\Gamma_2)$ are the 
same.  Then the composition $h=\phi_2^{-1}\phi_1$ is a well-defined map of 
the group $\cG_{(01)^\infty}$ to $G$.  Since both $\phi_1$ and $\phi_2$ are 
group isomorphisms, so is $h$.  As $G$ is a subgroup of 
$[[\si|\Omega(\xiGE)]]$, we may regard $h$ as a one-to-one homomorphism of 
$\cG_{(01)^\infty}$ to $[[\si|\Omega(\xiGE)]]$.  By construction, $h(\tilde 
a)=\de_a$, $h(\tilde c)=\de_c$ and $h(\tilde d)=\de_d$.
\end{proof}

Now we are going to derive Theorems \ref{sub-main} and \ref{main} from 
Proposition \ref{inj-homomorph}.

\begin{proof}[Proof of Theorem \ref{sub-main}]
It follows from Theorem \ref{GE-growth} that the group $\cG_{(01)^\infty}$ 
has intermediate growth.  The group $\cG_{(01)^\infty}$ is generated by 
three involutions $a$, $c$ and $d$.  Proposition \ref{inj-homomorph} 
implies that $\cG_{(01)^\infty}$ is isomorphic to the subgroup of the 
topological full group $[[\si|\Omega(\xiGE)]]$ generated by involutions 
$\de_a$, $\de_c$ and $\de_d$.  Since isomorphic groups have the same 
growth, the theorem follows.
\end{proof}

\begin{proof}[Proof of Theorem \ref{main}]
By Proposition \ref{3-subshifts-factors}, the subshift 
$\si_{\{x,y\}}|\Om(\xipd)$ is a continuous factor of the subshift 
$\si_{\{0,1\}}|\Om(\xiTM)$ while the subshift $\si_{\{a,c,d\}}|\Om(\xiGE)$ 
is topologically conjugate to $\si_{\{x,y\}}|\Om(\xipd)$.  By Proposition 
\ref{3-subshifts-Cantor}, all three subshifts are Cantor minimal systems.  
In particular, all three are aperiodic.  Then it follows from Theorem 
\ref{TFG-embeds} that the topological full group 
$[[\si_{\{x,y\}}|\Om(\xipd)]]$ embeds into the topological full group 
$[[\si_{\{0,1\}}|\Om(\xiTM)]]$ while the
topological full group $[[\si_{\{a,c,d\}}|\Om(\xiGE)]]$ embeds into 
$[[\si_{\{x,y\}}|\Om(\xipd)]]$.  As a consequence, any subgroup of 
$[[\si_{\{a,c,d\}}|\Om(\xiGE)]]$ is isomorphic to a subgroup of 
$[[\si_{\{0,1\}}|\Om(\xiTM)]]$ and a subgroup of 
$[[\si_{\{x,y\}}|\Om(\xipd)]]$.  Proposition \ref{inj-homomorph} implies 
that the group $\cG_{(01)^\infty}$ embeds into the topological full group 
$[[\si|\Omega(\xiGE)]]$.  It follows that $\cG_{(01)^\infty}$ also embeds 
into $[[\si_{\{0,1\}}|\Om(\xiTM)]]$ and $[[\si_{\{x,y\}}|\Om(\xipd)]]$, the 
topological full groups of the substitution subshifts generated by the 
Thue-Morse substitution and the period doubling substitution.
\end{proof}

The group $\cG_{(01)^\infty}$ can be embedded into the topological full 
group $[[\si_{\{0,1\}}|\Om(\xiTM)]]$ of the Thue-Morse subshift in three 
steps.  First $\cG_{(01)^\infty}$ is embedded into the group 
$[[\si_{\{a,c,d\}}|\Om(\xiGE)]]$, then $[[\si_{\{a,c,d\}}|\Om(\xiGE)]]$ is 
embedded into (in fact, mapped isomorphically onto) the group 
$[[\si_{\{x,y\}}|\Om(\xipd)]]$, and finally $[[\si_{\{x,y\}}|\Om(\xipd)]]$ 
is embedded into $[[\si_{\{0,1\}}|\Om(\xiTM)]]$.  The first step is 
described explicitly by Proposition \ref{inj-homomorph}.  The other two 
steps can be described explicitly as well.  Each time the generators $a$, 
$c$ and $d$ of the group $\cG_{(01)^\infty}$ are mapped to generalized 
$2$-cycles of the form $\de_{U;S}$, where $S$ is an appropriate subshift.  
To describe the clopen sets $U$, we need another notation.  Given a 
nonempty word $w$ over an alphabet $\cA$, let $[.w]_{\cA}$ denote the set 
of all bi-infinite strings $\ldots x_{-1}x_0.x_1x_2\ldots$ in $\cA^{\bZ}$ 
such that $x_1x_2\ldots x_n=w$, where $n$ is the length of $w$.

\begin{proposition}\label{two-more-inj}
(i) There exists a one-to-one group homomorphism $h_1:\cG_{(01)^\infty}\to 
[[\si_{\{x,y\}}|\Om(\xipd)]]$ such that 
$h_1(z)=\de_{U_z\cap\Om(\xipd);S_1}$ for all $z\in\{a,c,d\}$, where 
$S_1=\si_{\{x,y\}}|\Om(\xipd)$, $U_a=[.xy]_{\{x,y\}}\cup [.xxxy]_{\{x,y\}}$,
$U_c=[.y]_{\{x,y\}}$ and $U_d=[.xxy]_{\{x,y\}}$.

(ii) There exists a one-to-one group homomorphism $h_2:\cG_{(01)^\infty}\to 
[[\si_{\{0,1\}}|\Om(\xiTM)]]$ such that 
$h_2(z)=\de_{W_z\cap\Om(\xiTM);S_2}$ for all $z\in\{a,c,d\}$, where 
$S_2=\si_{\{0,1\}}|\Om(\xiTM)$,
\begin{eqnarray*}
W_a &=& [.011]_{\{0,1\}}\cup [.100]_{\{0,1\}}\cup [.01011]_{\{0,1\}}\cup 
[.10100]_{\{0,1\}},\\
W_c &=& [.00]_{\{0,1\}}\cup [.11]_{\{0,1\}},\\
W_d &=& [.0100]_{\{0,1\}}\cup [.1011]_{\{0,1\}}.
\end{eqnarray*}
\end{proposition}

\begin{proof}
Suppose $T_1:X_1\to X_1$ and $T_2:X_2\to X_2$ are aperiodic homeomorphisms 
of Cantor sets, and $f:X_1\to X_2$ is a continuous map that is onto and 
satisfies $fT_1=T_2f$.  By Theorem \ref{TFG-embeds}, there exists a 
one-to-one homomorphism $\Psi:[[T_2]]\to[[T_1]]$.  An explicit description 
of $\Psi$ can be extracted from Lemmas \ref{model-isomorphism} and 
\ref{group-embed}.  Namely, if $F\in[[T_2]]$ is a homeomorphism given by 
$F(x)=T_2^{\nu(x)}(x)$, $x\in X_2$, where $\nu:X_2\to\bZ$ is a continuous 
function, then its image $\Psi(F)$ is given by 
$\bigl(\Psi(F)\bigr)(x)=T_1^{\nu(f(x))}(x)$ for all $x\in X_1$.  In the 
case when $F=\de_{U;T_2}$ for some clopen set $U\subset X_2$, we obtain 
that $\Psi(F)=\de_{f^{-1}(U);T_1}$.

By Proposition \ref{inj-homomorph}, there is a one-to-one homomorphism 
$h:\cG_{(01)^\infty}\to[[\si_{\{a,c,d\}}|\Om(\xiGE)]]$ such that 
$h(z)=\de_z$ for all $z\in\{a,c,d\}$.  Recall that $\de_z$ is the short 
notation for the generalized $2$-cycle 
$\de_{[.z]_{\{a,c,d\}}\cap\Om(\xiGE);S}$, where 
$S=\si_{\{a,c,d\}}|\Om(\xiGE)$.  Let 
$f_{\phi_2}:\{a,c,d\}^{\bZ}\to\{x,y\}^{\bZ}$ be the block factor map 
considered in Lemma \ref{block-factor-2}.  In view of the lemma, the 
restriction of $f_{\phi_2}$ to the set $\Om(\xiGE)$ can be regarded a 
homeomorphism onto $\Om(\xipd)$.  Let $g_1$ be the inverse of that 
homeomorphism.  Then $g_1S_1=Sg_1$, where $S_1=\si_{\{x,y\}}|\Om(\xipd)$.
Hence $g_1$ induces a one-to-one homomorphism $\Psi_1:[[S]]\to[[S_1]]$ as 
described above.  Then $h_1=\Psi_1h$ is a one-to-one homomorphism of the 
group $\cG_{(01)^\infty}$ to $[[S_1]]$.  For any $z\in\{a,c,d\}$ we obtain 
that $h_1(z)=\Psi_1(\de_z)=\de_{U'_z;S_1}$, where $U'_z= 
g_1^{-1}\bigl([.z]_{\{a,c,d\}}\cap\Om(\xiGE)\bigr)
=f_{\phi_2}\bigl([.z]_{\{a,c,d\}}\cap\Om(\xiGE)\bigr)$.

It follows from the description of the string $\xiGE$ (Lemma \ref{xi-GE}) 
that any occurrence of the letter $d$ in $\xiGE$ is followed by $ac$.  
Hence $[.d]_{\{a,c,d\}}\cap\Om(\xiGE)=[.dac]_{\{a,c,d\}}\cap\Om(\xiGE)$.  
Any occurrence of the letter $a$ is followed by either $c$ or $dac$.  Hence 
the set $[.a]_{\{a,c,d\}}\cap\Om(\xiGE)$ is the union of 
$[.ac]_{\{a,c,d\}}\cap\Om(\xiGE)$ and $[.adac]_{\{a,c,d\}}\cap\Om(\xiGE)$.  
Recall that the $1$-block factor map $f_{\phi_2}$ is induced by the 
function $\phi_2:\{a,c,d\}\to\{x,y\}$ given by $\phi_2(a)=\phi_2(d)=x$ and 
$\phi_2(c)=y$.  Therefore $f_{\phi_2}\bigl([.c]_{\{a,c,d\}}\bigr)\subset 
[.y]_{\{x,y\}}$, $f_{\phi_2}\bigl([.ac]_{\{a,c,d\}}\bigr)\subset 
[.xy]_{\{x,y\}}$, $f_{\phi_2}\bigl([.dac]_{\{a,c,d\}}\bigr)\subset 
[.xxy]_{\{x,y\}}$ and $f_{\phi_2}\bigl([.adac]_{\{a,c,d\}}\bigr)\subset 
[.xxxy]_{\{x,y\}}$.  Note that $[.c]_{\{a,c,d\}}$, $[.ac]_{\{a,c,d\}}$, 
$[.dac]_{\{a,c,d\}}$ and $[.adac]_{\{a,c,d\}}$ are disjoint sets that cover 
$\Om(\xiGE)$.  As observed in the proof of Lemma \ref{block-factor-2}, the 
word $xxxx$ is not a subword of the string $\xipd$.  Hence the set 
$[.xxxx]_{\{x,y\}}$ is disjoint from $\Om(\xipd)$.  Consequently, 
$\Om(\xipd)$ is covered by disjoint sets $[.y]_{\{x,y\}}$, 
$[.xy]_{\{x,y\}}$, $[.xxy]_{\{x,y\}}$ and $[.xxxy]_{\{x,y\}}$.  Since the 
restriction of $f_{\phi_2}$ to $\Om(\xiGE)$ is a one-to-one map onto 
$\Om(\xipd)$, it follows that 
$f_{\phi_2}\bigl([.c]_{\{a,c,d\}}\cap\Om(\xiGE)\bigr)=
[.y]_{\{x,y\}}\cap\Om(\xipd)$, 
$f_{\phi_2}\bigl([.ac]_{\{a,c,d\}}\cap\Om(\xiGE)\bigr)= 
[.xy]_{\{x,y\}}\cap\Om(\xipd)$, 
$f_{\phi_2}\bigl([.dac]_{\{a,c,d\}}\cap\Om(\xiGE)\bigr)=
[.xxy]_{\{x,y\}}\cap\Om(\xipd)$ and 
$f_{\phi_2}\bigl([.adac]_{\{a,c,d\}}\cap\Om(\xiGE)\bigr)=
[.xxxy]_{\{x,y\}}\cap\Om(\xipd)$.  As a consequence, for any 
$z\in\{a,c,d\}$ we have 
$f_{\phi_2}\bigl([.z]_{\{a,c,d\}}\cap\Om(\xiGE)\bigr)=U_z\cap\Om(\xipd)$, 
where $U_z$ is as defined in the proposition.  This completes the proof of 
the statement (i).

We proceed to the statement (ii).  Let 
$f_{\phi_1}:\{0,1\}^{\bZ}\to\{x,y\}^{\bZ}$ be the block factor map 
considered in Lemma \ref{block-factor-1}.  The map $f_{\phi_1}$ is 
two-to-one when restricted to $\Om(\xiTM)$ and $f_{\phi_1}(\Om(\xiTM))= 
\Om(\xipd)$.  It follows from the proof of Lemma \ref{block-factor-1} that 
$f_{\phi_1}$ is two-to-one on $\{0,1\}^{\bZ}$.  As a consequence, 
$f_{\phi_1}^{-1}(\Om(\xipd))=\Om(\xiTM)$.  Let $g_2$ be the restriction of 
$f_{\phi_1}$ to $\Om(\xiTM)$.  Then $g_2S_2=S_1g_2$, where 
$S_2=\si_{\{0,1\}}|\Om(\xiTM)$.  Hence $g_2$ induces a one-to-one 
homomorphism $\Psi_2:[[S_1]]\to[[S_2]]$ as described above.  Then 
$h_2=\Psi_2h_1$ is a one-to-one homomorphism of the group 
$\cG_{(01)^\infty}$ to $[[S_2]]$.  For any $z\in\{a,c,d\}$ we have
$h_2(z)=\Psi_2\bigl(\de_{U_z\cap\Om(\xipd);S_1}\bigr)
=\de_{g_2^{-1}(U_z\cap\Om(\xipd));S_2}$.  Note that 
$g_2^{-1}(U_z\cap\Om(\xipd))=f_{\phi_1}^{-1}(U_z)\cap\Om(\xiTM)$ since 
$f_{\phi_1}^{-1}(\Om(\xipd))=\Om(\xiTM)$.  Recall that the $2$-block factor
map $f_{\phi_1}$ is induced by the function $\phi_1:\{0,1\}^2\to\{x,y\}$
given by $\phi_1(0,1)=\phi_1(1,0)=x$ and 
$\phi_1(0,0)=\phi_1(1,1)=y$.  We obtain that
\begin{align*}
&f_{\phi_1}^{-1}([.y]_{\{x,y\}}) = [.00]_{\{0,1\}}\cup [.11]_{\{0,1\}},\\
&f_{\phi_1}^{-1}([.xy]_{\{x,y\}}) = [.011]_{\{0,1\}}\cup 
[.100]_{\{0,1\}},\\
&f_{\phi_1}^{-1}([.xxy]_{\{x,y\}}) = [.0100]_{\{0,1\}}\cup
[.1011]_{\{0,1\}},\\
&f_{\phi_1}^{-1}([.xxxy]_{\{x,y\}}) = [.01011]_{\{0,1\}}\cup 
[.10100]_{\{0,1\}}.
\end{align*}
It follows that for any $z\in\{a,c,d\}$ we have $f_{\phi_1}^{-1}(U_z)=W_z$, 
where $W_z$ is as defined in the proposition.  This completes the proof of 
the statement (ii).
\end{proof}

\bigskip

{\sc
\begin{raggedright}
Department of Mathematics\\
Texas A\&M University\\
College Station, TX 77843--3368
\end{raggedright}
}

\end{document}